\documentclass[10pt]{amsart}
\usepackage{amsmath,amsfonts,amsthm,amssymb}
\usepackage{graphics}
\usepackage[matrix,arrow]{xy}
\usepackage[colorlinks]{hyperref}
\usepackage{mathrsfs}  
\usepackage{eucal}
\usepackage{slashed}
\usepackage{mathtools}
\usepackage{hyperref}
\usepackage[T1]{fontenc}

\usepackage{marginnote}\usepackage[dvipsnames]{xcolor}
\usepackage[colorinlistoftodos]{todonotes}

\usepackage{tikz-cd}

\usepackage{bbm}

\usepackage[
style=alphabetic,minnames=50, maxnames=50
]{biblatex}

\newtheoremstyle{my}{1.5em}{0.5em}{\em}{}{\sc}{:}{0.5em}{}

\usepackage[
margin=9mm,
marginparwidth=17mm, 
marginparsep=3mm,
]{geometry}

\numberwithin{equation}{section}
\theoremstyle{plain}
\newtheorem{theorem}{Theorem}[section]
\newtheorem{idealtheorem}[theorem]{Idealized Theorem}

\newtheorem{lemma}[theorem]{Lemma}
\newtheorem{proposition}[theorem]{Proposition}

\newtheorem{corollary}[theorem]{Corollary}

\theoremstyle{definition}
\newtheorem{definition}[theorem]{Definition}
\newtheorem{notation}[theorem]{Notation}
\newtheorem{assumption}[theorem]{Assumption}

\newtheorem{remark}[theorem]{Remark}

\newtheorem{remarks}[theorem]{Remarks}
\newtheorem{example}[theorem]{Example}
\theoremstyle{remark}

\numberwithin{equation}{section}

\newcommand{\R}{\mathbb{R}}
\newcommand{\Z}{\mathbb{Z}}

\newcommand{\C}{\mathbb{C}}

\newcommand{\D}{\mathbb{D}}
\newcommand{\K}{\mathbb{K}}
\newcommand{\Map}{\textup{Map}}

\newcommand{\Hom}{\textup{Hom}}

\newcommand{\rank}{\textup{rank}}
\newcommand{\id}{\textup{id}}

\newcommand{\ev}{\textup{ev}}

\newcommand{\wh}[1]{\widehat{#1}}
\newcommand{\wt}[1]{\widetilde{#1}}
\newcommand{\ol}[1]{\overline{#1}}

\newcommand{\mc}[1]{\mathcal{#1}}
\newcommand{\mf}[1]{\mathfrak{#1}}
\newcommand{\ms}[1]{\mathscr{#1}}

\newcommand{\dR}{\textup{dR}}
\newcommand{\pt}{\textup{pt}}

\newcommand{\CO}{\mathcal{CO}}

\newcommand{\CC}{\textbf{\textup{CH}}}

\newcommand{\fiberprod}[2]{\,_{#1}\!\!\times_{#2} }
\newcommand{\pr}{\textup{pr}}

\newcommand{\conc}{\textup{conc}^\mathscr{L}}
\newcommand{\concOmega}{\textup{conc}^{\Omega_\star}}

\newcommand{\norm}[1]{\left\lVert{#1}\right\rVert}

\newcommand{\CFr}{C_*^\mathscr{L}}
\newcommand{\CCFr}{\ol{C}_*^\mathscr{L}}
\newcommand{\CBa}{C_*^{\Omega_\star}}

\newcommand{\HBa}{H_*^{\Omega_\star}}

\newcommand{\Cfr}[1]{C_{#1}^\mathscr{L}}
\newcommand{\Cba}[1]{C_{#1}^{\Omega_\star}}

\newcommand{\Hba}[1]{H_{#1}^{\Omega_\star}}

\newcommand{\MM}{\textup{\textsf{M}}}
\newcommand{\mm}{\textup{\textsf{m}}}
\newcommand{\NN}{\textup{\textsf{N}}}
\newcommand{\NNN}{\textup{\textsf{N}}_\star}

\newcommand{\Aut}{\textup{Aut}}

\newcommand{\LL}{\textsf{L}}

\newcommand{\vdim}{\textup{v}\dim}

\newcommand{\Acknowledgements}{{\em Acknowledgements.} }

\title{The topology of Lagrangian submanifolds via open-closed string topology}
\author{Shuhao Li}

\begin{document}

\begin{abstract}
	We study the topology of Lagrangian submanifolds in standard symplectic vector spaces $\C^n$ using ideas from open-closed string topology. Specifically, for a closed, oriented, spin Lagrangian $L$, we construct a (possibly curved) deformation of the dg associative algebra of chains on the based loop space of $L$. This is done via pushing forward moduli spaces of pseudo-holomorphic discs with boundaries on $L$ viewed as chains in the free loop space along a string topology closed-open map. As an application, we prove that if $\pi_2(L)=0$, then $L$ has non-vanishing Maslov class, generalizing previous results due to Viterbo \cite{ViterboTori}, Cieliebak-Mohnke \cite{CieliebakMohnke}, Fukaya \cite{FukayaLagrangian} and Irie \cite{Irie2}.
\end{abstract}

	\maketitle

{\hypersetup{linkcolor=black}
\tableofcontents}

\section{Introduction}

\subsection{Context}\label{Subsc: Context}
This paper is concerned with the topology of closed Lagrangian submanifolds of the standard symplectic vector spaces $(\C^n,\omega)$ with vanishing Maslov classes.

The \emph{Maslov class} $\mu_L\colon \pi_1(L)\to \Z$ is an important invariant of a Lagrangian submanifold $L\subset \C^n$, defined by Arnol'd in \cite{ArnoldMaslov}.
We briefly recall the construction.
Denote by $\mathcal{L}Gr(n)$ the \emph{Lagrangian Grassmannian} in dimension $n$, i.e. the space consisting of all linear Lagrangian subspaces in $\C^n$.
Arnol'd in \emph{op. cit.} showed that $\pi_1(\mathcal{L}Gr(n))\cong \Z$.
A Lagrangian embedding $i_L\colon L\hookrightarrow \C^n$ induces the Lagrangian Gauss map $Ti_L\colon L\to \mathcal{L}Gr(n)$, taking each $p\in L$ to the Lagrangian subspace $T_pL\subset T_p\C^n\cong \C^n$.
This induces a homomorphism $\mu_L\colon \pi_1L\to \pi_1\mc{L}Gr(n)\cong \Z$, which is called the Maslov class.
Maslov classes can be defined for Lagrangians in general symplectic manifolds $(M,\omega)$ as a homomorphism $\pi_2(M,L)\to \Z$, but for our purpose we shall restrict to the case $M=\C^n$.

The Maslov class of a Lagrangian plays the role of the relative first Chern class (see e.g. section 2.1.1 of \cite{FOOO1}).
Much as symplectic manifolds with vanishing first Chern classes (``symplectic Calabi-Yau's'') play a special role in the symplectic topology,
Lagrangians with vanishing Maslov classes (which we sometimes abbreviate as \emph{Maslov-zero Lagrangians}) also play a special role.
Their Floer theories are often graded  (see e.g. \cite{SeidelGraded}) and thus better behaved.
Also, just as the vanishing of the first Chern class is the symplectic topological version of the \emph{Calabi-Yau} condition, the vanishing of the Maslov class is the symplectic topological version of the \emph{special Lagrangian} condition (see e.g. \cite{HarveyLawson,JoyceSlag} for definitions; also see e.g. the remark after Lemma 3.1 in \cite{AurouxTDuality}).

Given a closed manifold $L$ of dimension $n$, the existence or non-existence of a Maslov-zero Lagrangian embedding $L\hookrightarrow \C^n$ is a well-studied question in symplectic topology (see e.g. \cite{AudinConj} and \cite{AudinLalondePolterovich} for the discussion of the \emph{Maslov class rigidity} phenomenon, or section 6.1.2 in \cite{FOOO1} for the discussion on the \emph{Maslov class conjecture}). 
This falls under the more general question of \emph{which homotopy classes of maps in $[L,\mc{L}Gr(n)]$ are realized by Lagrangian embeddings}. In contrast,  all classes in $[L,\mc{L}Gr(n)]$ are realized by Lagrangian \emph{immersions} (see \cite{LeesImmer,GromovPDR}). 
The proof of a generalized version of Audin's conjecture \cite{AudinConj}, a strong Maslov class rigidity phenomenon,  is also the key ingredient in the works of \cite{FukayaLagrangian,Irie2} on the classification of prime 3-manifolds admitting a Lagrangian embedding in $\C^3$.

Under various topological assumptions on $L$, non-existence of Maslov-zero embeddings into $\C^n$ is established:
\begin{enumerate}
	\item Viterbo \cite{ViterboTori} showed that any closed manifold admitting a metric of non-positive sectional curvature does not admit a Maslov-zero Lagrangian embedding in $\C^n$; 
	\item Polterovich \cite{PolterovichMaslov} proved non-vanishing of Maslov classes for certain Lagrangian surfaces (including certain non-compact ones) in $\C^2$;
	\item Fukaya-Oh-Ohta-Ono (see Theorem K in \cite{FOOO1}) showed that any closed, spin manifold with vanishing second Betti number does not admit a Maslov-zero Lagrangian embedding in $\C^n$; this builds on Oh's construction of spectral sequence in Floer theory in \cite{OhSpecSeq};
	\item Fukaya \cite{FukayaLagrangian} proposed a proof that any closed, aspherical, spin manifold does not admit a Maslov-zero Lagrangian embedding into $\C^n$, which was later realized by Irie \cite{Irie1,Irie2}. In fact, they show that in this situation, there always exists a curve with Maslov index 2; this is a version of  Audin's conjecture \cite{AudinConj} (the case of tori, which is the original statement of  Audin's conjecture, is first proved by \cite{CieliebakMohnke} using neck-stretching analysis for holomorphic curves);
\end{enumerate}
See also e.g. section 6.1.2 in \cite{FOOO1} for more history.

On the other hand, \cite{EEMS} proved that $S^2\times S^1$ admits a Maslov-zero Lagrangian embedding into $\C^3$. 
This is in contrast with the situation of closed \emph{special Lagrangian} submanifolds, which does not exist in $\C^n$ because special Lagrangians are calibrated submanifolds which are automatically minimal submanifolds (see e.g. \cite{HarveyLawson}), whereas one can always change the volume of closed Lagrangians in $\C^n$ by scaling.

\subsection{Result}
The main result of this paper is a new Maslov class rigidity phenomenon:
\begin{theorem}[Theorem \ref{Thm: MainRestate}]\label{Thm: Main}
	If $L$ is a closed, oriented, spin manifold and $\pi_2(L)=0$, then $L$ does not admit a Lagrangian embedding into $\C^n$ with vanishing Maslov class.
\end{theorem}
The asphericity condition in the results of \cite{ViterboTori,FukayaLagrangian,Irie2} requires the vanishing of all higher homotopy groups, whereas we only require vanishing of $\pi_2$. There are many examples of manifolds with vanishing $\pi_2$ but are not covered by the prior results listed above. For example:
\begin{itemize}
	\item All compact connected Lie groups have vanishing $\pi_2$ (\cite{BottLieGroup}); e.g. $T^2\times SU(2)$ is an example not covered by prior results;
	\item Arbitary connected sums between aspherical or spherical manifolds of dimension $\geq$ 4, e.g. $T^4\# T^4$, or more generally connected sums of manifolds admitting metrics with non-positive sectional curvature; in fact, the condition $\pi_2=0$ is preserved under taking connected sums for manifolds of dimension at least 4;
	\item In Remark 1.23 (a) \cite{CieliebakMohnke}, they pointed out the dichotomy that the technique in their paper works well for manifolds admitting metrics with non-positive curvature whereas traditional Floer theoretic techniques work well for simply-connected manifolds, and that a test case for combining these techniques is the product of a manifold of positive with a manifold of negative curvature.
	Our condition of $\pi_2=0$ is preserved under taking products, and thus applies to manifolds like the product of a sphere (of dimension $\geq 3$) with a hyperbolic manifold.
\end{itemize}
See section \ref{Subsc: Examples}.

\begin{remarks}
	\begin{enumerate}
		\item 	Unlike the case with aspherical manifolds, one cannot expect a version of Audin's conjecture to hold for manifolds with $\pi_2=0$. For example, for each $k>1$, $S^1\times S^{2k-1}$ has vanishing $\pi_2$ but admits a Lagrangian embedding with minimum Maslov number $2k$ by performing Polterovich surgery on the double point on Whitney's immersed sphere (Theorem 5 in \cite{PolterovichSurgery}).
		\item 	In fact, the actual topological condition we use in the proof is the vanishing of  the first Betti number of the based loop space of $L$.
		We state the theorem in terms of $\pi_2 L$ to compare with the previous Maslov class rigidity results.
		\item One of the initial motivations of this work is to understand the topology of Lagrangian 3-manifolds in $\C^3$. The existence or non-existence of Lagrangian embeddings of $T^3\# T^3$, or connected sums of hyperbolic 3-manifolds, into $\C^3$ is wide open (see e.g. Problem 11.1 in \cite{FukayaLagrangian}, Question 2.1 in \cite{SmithSurvey}, Question 5 in \cite{EvansKedra}). However, these connected sums of 3-manifolds have an essential 2-sphere in the connected sum region and thus do not fall under our theorem. We hope to return to these examples in future works.
	\end{enumerate}
\end{remarks}

\subsection{A sketch of the proof}

The proof of this theorem goes through the construction of a new Floer-theoretic invariant of Lagrangians in $\C^n$, which was conjectured and sketched by Abouzaid in \cite{AbouzaidLect}, for applications e.g. in the \emph{family Floer theory} construction of SYZ mirrors.
Here we state (imprecisely) an idealized version of this invariant to provide geometric pictures. 
For the precise statement, which eventually gives us the same geometric consequences as this idealized version by some additional homological algebra, see Theorem \ref{Cor: BasedAinfty}.

Let $\Omega_\star L$ be the based loop space of $L$ with basepoint $\star\in L$. 
Under a good choice of the chain model, chains on $\Omega_\star L$ has the structure of a dg associative algebra $(C_*\Omega_\star L, \partial, \bullet)$ given by the Pontryagin product. 
Moreover there is a unit $\underline{\star}\in C_0\Omega_\star L$ given by the constant based loop at $\star \in L$.

Let $\theta:= \sum_i x_i\,dy_i$, a primitive to the standard symplectic form $\omega =\sum_{i=1}^n dx_i\wedge dy_i$ on $\C^n$.
Define a homomorphism $E\colon \pi_1 L\to \R$ given by the \emph{symplectic energy} $E(\gamma):= \int_\gamma \theta$.
Under the decomposition of $\Omega_\star L$ into connected components $\Omega_\star(a)$ labeled by $a\in \pi_1 L$, there is a splitting $C_*\Omega_\star L=\bigoplus_{a\in \pi_1 L} C_*\Omega_\star(a)$ compatible with the dg associative algebra structure.
We define the \emph{energy filtration} $\{\ms{F}^\lambda C_*\Omega_\star L\}_{\lambda \in \R}$ on $C_*\Omega_\star L$ by \begin{align*}
	\ms{F}^\lambda C_*\Omega_\star L := \bigoplus_{E(a)>\lambda} C_*\Omega_\star(a).
\end{align*}
Denote by $\wh{C_*\Omega_\star L}$ the completion with respect to the energy filtration, and $\ms{F}^\lambda \wh{C_*\Omega_\star L}$ the corresponding filtration levels.

\begin{idealtheorem}\label{IdealThm}
	Let $L$ be a closed, oriented, spin Lagrangian submanifold of $\C^n$. Then there exists a constant $\hbar>0$ and a (gapped) curved dg associative algebra structure on $\wh{C_*\Omega_\star L}$ which is a deformation of the Pontryagin algebra $(C_*\Omega_\star L,\partial, \bullet)$, i.e. the data of 
	\begin{enumerate}
  		\setcounter{enumi}{-1}
  		\item A constant $\hbar >0$;
		\item $\mm_0\in \wh{C_*\Omega_\star L}$ which lives in $\ms{F}^{\hbar} \wh{C_*\Omega_\star L}$;
		\item $\mm_1\colon \wh{C_*\Omega_\star L}\to \wh{C_*\Omega_\star L}$ given by $\mm_1=\partial +\mm_{1,+}$ where $\partial$ is induced by the classical differential on $C_*\Omega_\star L$ and $\mm_{1,+}$ raises energy by at least $\hbar$;
		\item $\mm_2\colon \wh{C_*\Omega_\star L}^{\otimes 2} \to \wh{C_*\Omega_\star L}$, induced by the Pontryagin product $\bullet$ on $C_*\Omega_\star L$,
	\end{enumerate}
	satisfying 
	\begin{enumerate}
		\item $\mm_1(\mm_0)=0$;
		\item $\mm_1^2(\alpha) = [\mm_0, \alpha]$ for all $\alpha \in \wh{C_*\Omega_\star L}$, where the right-hand side is the graded commutator;
		\item Leibniz rule.
	\end{enumerate}
	Moreover, if $\mu_L\equiv 0$, then 
	\begin{enumerate}
		\item $\mm_0 = 0$;
		\item $\deg \mm_1 = -1$;
		\item There exists an element $\NN^\star\in \wh{C_1\Omega_\star L}$ such that $\mm_1 \NN^\star = \underline{\star}$.
	\end{enumerate}
\end{idealtheorem}

We briefly explain the geometric picture of the terms (see section \ref{Sec: SketchPf} for more details): 
\begin{itemize}
	\item The curvature term $\mm_0 = \sum_{\beta\in H_1(L;\Z)} \mm_{0,\beta} \in \wh{C_*\Omega_\star L}$ is given by the boundaries loops of the pseudo-holomorphic discs which pass through the chosen basepoint $\star \in L$. 
	See Figure \ref{fig:m0}.

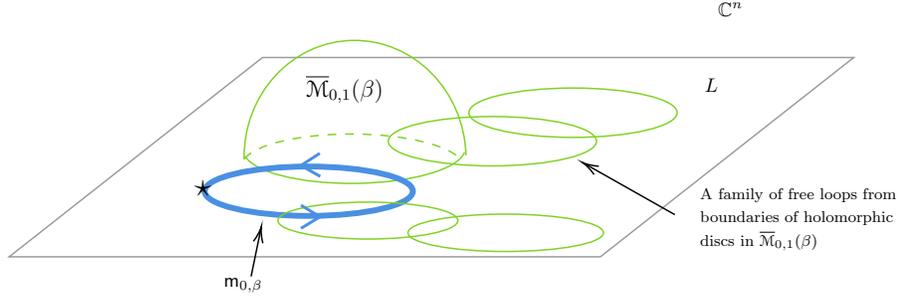
\begin{figure}
\centering
\resizebox{12cm}{!}{%
\tikzset{every picture/.style={line width=0.75pt}} 

\begin{tikzpicture}[x=0.75pt,y=0.75pt,yscale=-1,xscale=1]

\draw  [color={rgb, 255:red, 155; green, 155; blue, 155 }  ,draw opacity=1 ] (195.8,56.73) -- (599,56.73) -- (426.2,192.77) -- (23,192.77) -- cycle ;
\draw  [draw opacity=0][line width=0.75]  (183.21,126.28) .. controls (182.48,82.02) and (215.68,45.67) .. (257.43,45.06) .. controls (299.01,44.45) and (333.4,79.53) .. (334.54,123.55) -- (258.88,125.33) -- cycle ; \draw  [color={rgb, 255:red, 126; green, 211; blue, 33 }  ,draw opacity=1 ][line width=0.75]  (183.21,126.28) .. controls (182.48,82.02) and (215.68,45.67) .. (257.43,45.06) .. controls (299.01,44.45) and (333.4,79.53) .. (334.54,123.55) ;  
\draw  [draw opacity=0] (333.56,121.38) .. controls (329.31,132.98) and (297.47,142.27) .. (258.7,142.65) .. controls (221.27,143.03) and (190.08,134.97) .. (183.9,124.05) -- (258.41,119.4) -- cycle ; \draw  [color={rgb, 255:red, 126; green, 211; blue, 33 }  ,draw opacity=1 ] (333.56,121.38) .. controls (329.31,132.98) and (297.47,142.27) .. (258.7,142.65) .. controls (221.27,143.03) and (190.08,134.97) .. (183.9,124.05) ;  
\draw  [draw opacity=0][dash pattern={on 4.5pt off 4.5pt}] (183.56,125.82) .. controls (191.22,115.97) and (221.43,108.75) .. (257.47,108.96) .. controls (294.7,109.18) and (325.59,117.24) .. (331.8,127.66) -- (257.31,131.14) -- cycle ; \draw  [color={rgb, 255:red, 126; green, 211; blue, 33 }  ,draw opacity=1 ][dash pattern={on 4.5pt off 4.5pt}] (183.56,125.82) .. controls (191.22,115.97) and (221.43,108.75) .. (257.47,108.96) .. controls (294.7,109.18) and (325.59,117.24) .. (331.8,127.66) ;  
\draw  [color={rgb, 255:red, 74; green, 144; blue, 226 }  ,draw opacity=1 ][line width=3]  (156.31,147.99) .. controls (156.31,138.68) and (188.1,131.14) .. (227.31,131.14) .. controls (266.52,131.14) and (298.31,138.68) .. (298.31,147.99) .. controls (298.31,157.29) and (266.52,164.84) .. (227.31,164.84) .. controls (188.1,164.84) and (156.31,157.29) .. (156.31,147.99) -- cycle ;
\draw  [color={rgb, 255:red, 126; green, 211; blue, 33 }  ,draw opacity=1 ] (295,176.38) .. controls (295,169.33) and (324.84,163.62) .. (361.66,163.62) .. controls (398.47,163.62) and (428.31,169.33) .. (428.31,176.38) .. controls (428.31,183.42) and (398.47,189.13) .. (361.66,189.13) .. controls (324.84,189.13) and (295,183.42) .. (295,176.38) -- cycle ;
\draw  [color={rgb, 255:red, 126; green, 211; blue, 33 }  ,draw opacity=1 ] (336.53,94.41) .. controls (336.53,85.11) and (368.32,77.56) .. (407.53,77.56) .. controls (446.75,77.56) and (478.53,85.11) .. (478.53,94.41) .. controls (478.53,103.72) and (446.75,111.26) .. (407.53,111.26) .. controls (368.32,111.26) and (336.53,103.72) .. (336.53,94.41) -- cycle ;
\draw  [color={rgb, 255:red, 126; green, 211; blue, 33 }  ,draw opacity=1 ] (281.53,113.85) .. controls (281.53,104.54) and (313.32,97) .. (352.53,97) .. controls (391.75,97) and (423.53,104.54) .. (423.53,113.85) .. controls (423.53,123.15) and (391.75,130.7) .. (352.53,130.7) .. controls (313.32,130.7) and (281.53,123.15) .. (281.53,113.85) -- cycle ;
\draw  [color={rgb, 255:red, 126; green, 211; blue, 33 }  ,draw opacity=1 ] (206.31,168.03) .. controls (206.31,161.07) and (233.78,155.43) .. (267.66,155.43) .. controls (301.54,155.43) and (329,161.07) .. (329,168.03) .. controls (329,174.99) and (301.54,180.63) .. (267.66,180.63) .. controls (233.78,180.63) and (206.31,174.99) .. (206.31,168.03) -- cycle ;
\draw    (477,164) -- (415.74,129.38) ;
\draw [shift={(414,128.4)}, rotate = 29.47] [color={rgb, 255:red, 0; green, 0; blue, 0 }  ][line width=0.75]    (10.93,-3.29) .. controls (6.95,-1.4) and (3.31,-0.3) .. (0,0) .. controls (3.31,0.3) and (6.95,1.4) .. (10.93,3.29)   ;
\draw    (188,206.14) -- (194.6,174.08) ;
\draw [shift={(195,172.12)}, rotate = 101.63] [color={rgb, 255:red, 0; green, 0; blue, 0 }  ][line width=0.75]    (10.93,-3.29) .. controls (6.95,-1.4) and (3.31,-0.3) .. (0,0) .. controls (3.31,0.3) and (6.95,1.4) .. (10.93,3.29)   ;
\draw  [color={rgb, 255:red, 74; green, 144; blue, 226 }  ,draw opacity=1 ][line width=1.5]  (235.16,137.91) -- (222,130.42) -- (234.83,122.12) ;
\draw  [color={rgb, 255:red, 74; green, 144; blue, 226 }  ,draw opacity=1 ][line width=1.5]  (222,157.55) -- (235,165.44) -- (222,173.34) ;

\draw (496,69.59) node [anchor=north west][inner sep=0.75pt]   [align=left] {$\displaystyle L$};
\draw (505,17.58) node [anchor=north west][inner sep=0.75pt]   [align=left] {$\displaystyle \mathbb{C}^{n}$};
\draw (148,139.26) node [anchor=north west][inner sep=0.75pt]  [font=\huge] [align=left] {$\displaystyle \star $};
\draw (223.92,67.49) node [anchor=north west][inner sep=0.75pt]  [font=\Large] [align=left] {$\displaystyle \overline{\mathcal{M}}_{0,1}( \beta )$};
\draw (492.87,144.31) node [anchor=north west][inner sep=0.75pt]   [align=left] {{\footnotesize A family of free loops from }\\{\footnotesize boundaries of holomorphic }\\{\footnotesize discs in $\displaystyle \overline{\mathcal{M}}_{0,1}( \beta )$}};
\draw (168,205.64) node [anchor=north west][inner sep=0.75pt]   [align=left] {$\displaystyle \mathsf{m}_{0,\beta }$};

\end{tikzpicture}
}
\caption{Definition of $\mm_{0,\beta}$} \label{fig:m0}
\end{figure}

	\item The term $\mm_{1,+} = \sum_{\beta \in H_1(L;\Z)} \mm_{1,\beta}\colon \wh{C_*\Omega_\star L}\to \wh{C_*\Omega_\star L}$ is a string topology operation,  given by taking the intersection between the sweepout of a family of based loops and the geometric image of the boundary loops of pseudo-holomorphic curves, and concatenating them where they intersect. See Figure \ref{fig:m1}.

\begin{figure}
	\centering
\resizebox{12cm}{!}{%

\tikzset{every picture/.style={line width=0.75pt}} 

\begin{tikzpicture}[x=0.75pt,y=0.75pt,yscale=-1,xscale=1]

\draw  [color={rgb, 255:red, 155; green, 155; blue, 155 }  ,draw opacity=1 ] (221.8,45) -- (583,45) -- (428.2,157) -- (67,157) -- cycle ;
\draw  [color={rgb, 255:red, 74; green, 144; blue, 226 }  ,draw opacity=1 ] (203.06,97.99) .. controls (223.27,84.74) and (258.66,73.99) .. (282.1,73.99) .. controls (305.55,73.99) and (308.17,84.74) .. (287.96,97.99) .. controls (267.76,111.25) and (232.37,121.99) .. (208.92,121.99) .. controls (185.48,121.99) and (182.86,111.25) .. (203.06,97.99) -- cycle ;
\draw  [color={rgb, 255:red, 74; green, 144; blue, 226 }  ,draw opacity=1 ] (206.86,108.64) .. controls (231,107.46) and (266.95,116.13) .. (287.16,128.01) .. controls (307.37,139.89) and (304.19,150.48) .. (280.05,151.67) .. controls (255.91,152.85) and (219.96,144.18) .. (199.75,132.3) .. controls (179.53,120.42) and (182.72,109.83) .. (206.86,108.64) -- cycle ;
\draw  [color={rgb, 255:red, 74; green, 144; blue, 226 }  ,draw opacity=1 ] (205.32,105.42) .. controls (228.33,98.02) and (265.3,97.08) .. (287.9,103.31) .. controls (310.5,109.55) and (310.17,120.61) .. (287.16,128.01) .. controls (264.15,135.41) and (227.18,136.35) .. (204.58,130.11) .. controls (181.98,123.88) and (182.31,112.82) .. (205.32,105.42) -- cycle ;
\draw  [color={rgb, 255:red, 74; green, 144; blue, 226 }  ,draw opacity=1 ] (195.17,92.24) .. controls (212.96,75.88) and (246.14,59.55) .. (269.28,55.75) .. controls (292.41,51.96) and (296.74,62.13) .. (278.94,78.49) .. controls (261.15,94.84) and (227.97,111.17) .. (204.83,114.97) .. controls (181.7,118.77) and (177.37,108.59) .. (195.17,92.24) -- cycle ;
\draw  [draw opacity=0] (301.76,81.04) .. controls (301.22,44.59) and (325.79,14.66) .. (356.67,14.15) .. controls (387.45,13.66) and (412.91,42.57) .. (413.72,78.84) -- (357.75,80.24) -- cycle ; \draw  [color={rgb, 255:red, 126; green, 211; blue, 33 }  ,draw opacity=1 ] (301.76,81.04) .. controls (301.22,44.59) and (325.79,14.66) .. (356.67,14.15) .. controls (387.45,13.66) and (412.91,42.57) .. (413.72,78.84) ;  
\draw  [draw opacity=0] (413.72,78.84) .. controls (411.21,89.89) and (387.37,98.8) .. (358.23,99.12) .. controls (329.87,99.44) and (306.35,91.52) .. (302.56,80.96) -- (357.99,77.44) -- cycle ; \draw  [color={rgb, 255:red, 126; green, 211; blue, 33 }  ,draw opacity=1 ] (413.72,78.84) .. controls (411.21,89.89) and (387.37,98.8) .. (358.23,99.12) .. controls (329.87,99.44) and (306.35,91.52) .. (302.56,80.96) ;  
\draw  [draw opacity=0][dash pattern={on 4.5pt off 4.5pt}] (302.77,82.27) .. controls (307.99,73.96) and (330.63,67.82) .. (357.72,68) .. controls (385.6,68.18) and (408.68,74.98) .. (412.87,83.72) -- (357.6,86.26) -- cycle ; \draw  [color={rgb, 255:red, 126; green, 211; blue, 33 }  ,draw opacity=1 ][dash pattern={on 4.5pt off 4.5pt}] (302.77,82.27) .. controls (307.99,73.96) and (330.63,67.82) .. (357.72,68) .. controls (385.6,68.18) and (408.68,74.98) .. (412.87,83.72) ;  
\draw  [color={rgb, 255:red, 155; green, 155; blue, 155 }  ,draw opacity=1 ] (230.8,174) -- (592,174) -- (437.2,286) -- (76,286) -- cycle ;
\draw  [fill={rgb, 255:red, 0; green, 0; blue, 0 }  ,fill opacity=1 ] (329,174.8) -- (336.56,174.8) -- (336.56,140) -- (338.44,140) -- (338.44,174.8) -- (346,174.8) -- (337.5,198) -- cycle ;
\draw  [color={rgb, 255:red, 74; green, 144; blue, 226 }  ,draw opacity=1 ] (203.06,230.99) .. controls (223.27,217.74) and (258.66,206.99) .. (282.1,206.99) .. controls (305.55,206.99) and (308.17,217.74) .. (287.96,230.99) .. controls (267.76,244.25) and (232.37,254.99) .. (208.92,254.99) .. controls (185.48,254.99) and (182.86,244.25) .. (203.06,230.99) -- cycle ;
\draw  [draw opacity=0] (412.56,216.39) .. controls (406.61,225.89) and (384.67,233.14) .. (358.44,233.43) .. controls (330.08,233.75) and (306.56,225.83) .. (302.77,215.27) -- (358.2,211.74) -- cycle ; \draw  [color={rgb, 255:red, 126; green, 211; blue, 33 }  ,draw opacity=1 ] (412.56,216.39) .. controls (406.61,225.89) and (384.67,233.14) .. (358.44,233.43) .. controls (330.08,233.75) and (306.56,225.83) .. (302.77,215.27) ;  
\draw  [draw opacity=0] (302.77,215.27) .. controls (307.99,206.96) and (330.63,200.82) .. (357.72,201) .. controls (385.6,201.18) and (408.68,207.98) .. (412.87,216.72) -- (357.6,219.26) -- cycle ; \draw  [color={rgb, 255:red, 126; green, 211; blue, 33 }  ,draw opacity=1 ] (302.77,215.27) .. controls (307.99,206.96) and (330.63,200.82) .. (357.72,201) .. controls (385.6,201.18) and (408.68,207.98) .. (412.87,216.72) ;  
\draw  [color={rgb, 255:red, 74; green, 144; blue, 226 }  ,draw opacity=1 ][line width=3]  (203.06,230.99) .. controls (223.27,217.74) and (258.66,206.99) .. (282.1,206.99) .. controls (305.55,206.99) and (308.17,217.74) .. (287.96,230.99) .. controls (267.76,244.25) and (232.37,254.99) .. (208.92,254.99) .. controls (185.48,254.99) and (182.86,244.25) .. (203.06,230.99) -- cycle ;
\draw  [draw opacity=0][line width=3]  (412.56,216.39) .. controls (406.61,225.89) and (384.67,233.14) .. (358.44,233.43) .. controls (330.08,233.75) and (306.56,225.83) .. (302.77,215.27) -- (358.2,211.74) -- cycle ; \draw  [color={rgb, 255:red, 74; green, 144; blue, 226 }  ,draw opacity=1 ][line width=3]  (412.56,216.39) .. controls (406.61,225.89) and (384.67,233.14) .. (358.44,233.43) .. controls (330.08,233.75) and (306.56,225.83) .. (302.77,215.27) ;  
\draw  [draw opacity=0][line width=3]  (302.77,215.27) .. controls (307.99,206.96) and (330.63,200.82) .. (357.72,201) .. controls (385.6,201.18) and (408.68,207.98) .. (412.87,216.72) -- (357.6,219.26) -- cycle ; \draw  [color={rgb, 255:red, 74; green, 144; blue, 226 }  ,draw opacity=1 ][line width=3]  (302.77,215.27) .. controls (307.99,206.96) and (330.63,200.82) .. (357.72,201) .. controls (385.6,201.18) and (408.68,207.98) .. (412.87,216.72) ;  
\draw  [color={rgb, 255:red, 74; green, 144; blue, 226 }  ,draw opacity=1 ][line width=1.5]  (237.79,243.64) -- (250.59,247.21) -- (244.47,259) ;
\draw  [color={rgb, 255:red, 74; green, 144; blue, 226 }  ,draw opacity=1 ][line width=1.5]  (321.39,219.06) -- (329.93,229.25) -- (318.2,235.51) ;
\draw  [color={rgb, 255:red, 74; green, 144; blue, 226 }  ,draw opacity=1 ][line width=1.5]  (356.28,210.49) -- (347.8,200.26) -- (359.56,194.07) ;
\draw  [color={rgb, 255:red, 74; green, 144; blue, 226 }  ,draw opacity=1 ][line width=1.5]  (245.39,219.16) -- (232.33,216.73) -- (237.4,204.45) ;
\draw  [color={rgb, 255:red, 126; green, 211; blue, 33 }  ,draw opacity=1 ][line width=0.75]  (349.08,94.78) -- (356.25,99.11) -- (348.73,102.82) ;
\draw  [color={rgb, 255:red, 74; green, 144; blue, 226 }  ,draw opacity=1 ][line width=0.75]  (294.25,119.81) -- (302.57,118.82) -- (298.92,126.36) ;
\draw  [color={rgb, 255:red, 74; green, 144; blue, 226 }  ,draw opacity=1 ][line width=0.75]  (279.59,147.24) -- (287.24,150.66) -- (280.24,155.26) ;
\draw  [color={rgb, 255:red, 74; green, 144; blue, 226 }  ,draw opacity=1 ][line width=0.75]  (282.13,69.13) -- (289.97,66.17) -- (288.24,74.36) ;
\draw  [color={rgb, 255:red, 126; green, 211; blue, 33 }  ,draw opacity=1 ][line width=0.75]  (355.19,72.03) -- (347.9,67.9) -- (355.31,63.99) ;
\draw    (166,54.5) -- (211.1,69.37) ;
\draw [shift={(213,70)}, rotate = 198.25] [color={rgb, 255:red, 0; green, 0; blue, 0 }  ][line width=0.75]    (10.93,-3.29) .. controls (6.95,-1.4) and (3.31,-0.3) .. (0,0) .. controls (3.31,0.3) and (6.95,1.4) .. (10.93,3.29)   ;
\draw    (331,259.13) -- (307.44,236.39) ;
\draw [shift={(306,235)}, rotate = 43.99] [color={rgb, 255:red, 0; green, 0; blue, 0 }  ][line width=0.75]    (10.93,-3.29) .. controls (6.95,-1.4) and (3.31,-0.3) .. (0,0) .. controls (3.31,0.3) and (6.95,1.4) .. (10.93,3.29)   ;
\draw    (413,111.13) -- (380.92,101.57) ;
\draw [shift={(379,101)}, rotate = 16.6] [color={rgb, 255:red, 0; green, 0; blue, 0 }  ][line width=0.75]    (10.93,-3.29) .. controls (6.95,-1.4) and (3.31,-0.3) .. (0,0) .. controls (3.31,0.3) and (6.95,1.4) .. (10.93,3.29)   ;

\draw (516,54) node [anchor=north west][inner sep=0.75pt]   [align=left] {$\displaystyle L$};
\draw (525,11) node [anchor=north west][inner sep=0.75pt]   [align=left] {$\displaystyle \mathbb{C}^{n}$};
\draw (180,106) node [anchor=north west][inner sep=0.75pt]  [font=\huge] [align=left] {$\displaystyle \star $};
\draw (326,32) node [anchor=north west][inner sep=0.75pt]  [font=\Large] [align=left] {$\displaystyle \overline{\mathcal{M}}_{0,1}( \beta )$};
\draw (517,186) node [anchor=north west][inner sep=0.75pt]   [align=left] {$\displaystyle L$};
\draw (343,162) node [anchor=north west][inner sep=0.75pt]   [align=left] {$\displaystyle \mathsf{m_{1,\beta }}$};
\draw (183,241) node [anchor=north west][inner sep=0.75pt]  [font=\huge] [align=left] {$\displaystyle \star $};
\draw (90,22) node [anchor=north west][inner sep=0.75pt]  [font=\footnotesize] [align=left] {A family of based loops,\\viewed as a chain $\displaystyle c\in C_{*}^{\Omega _{\star }}$};
\draw (332,243) node [anchor=north west][inner sep=0.75pt]   [align=left] {{\footnotesize Resulting chain $\displaystyle \mm_{1,\beta}(c)$}\\{\footnotesize in $\displaystyle C_{*}^{\Omega _{\star}}$ by concatenation}};
\draw (414.87,92.72) node [anchor=north west][inner sep=0.75pt]   [align=left] {{\footnotesize Boundary of }\\{\footnotesize holomorphic disc}};

\end{tikzpicture}
}
\caption{Definition of $\mm_{1,\beta}$} \label{fig:m1}
\end{figure}
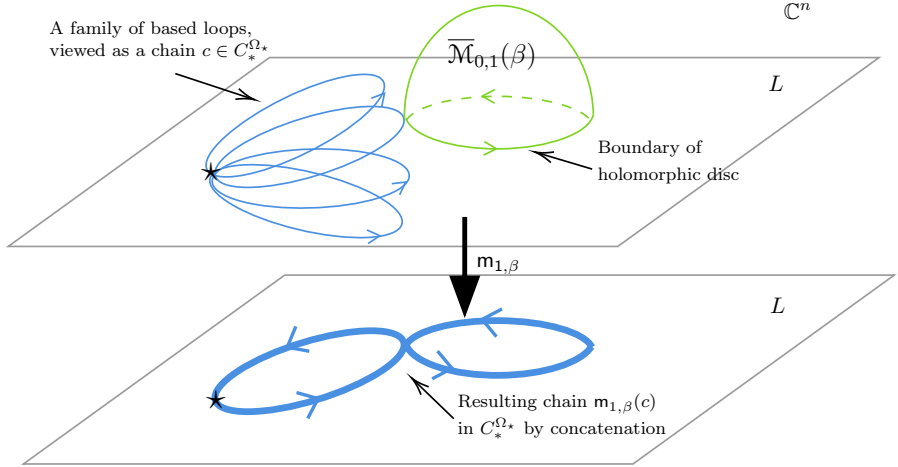
\end{itemize}

\begin{remark}
	In the context of homological mirror symmetry, especially in the case $L$ is a Lagrangian torus (e.g. a smooth fibre of an SYZ fibration), $\mm_0$ can be thought of as encoding the information of the \emph{superpotential} in the mirror local chart  (see e.g. \cite{AbouzaidLect}, section 3 of \cite{AurouxTDuality}, as well as e.g. \cite{TonkonogEnum,YuanHangThesis}).
	
	In the main text, we sometimes refer to $\mm_0$ as the \emph{anomaly}  in order to distinguish it with the \emph{curvature} in ordinary Lagrangian Floer theory of \cite{FOOO1}. The terminology of anomaly is also used in various similar contexts in string topology \cite{SullivanStringSurvey} and Lagrangian Floer theory \cite{FOOO1,FOOO2}.
\end{remark}

We expect that such a construction works for closed, (relatively) spin Lagrangian submanifold of any  symplectic manifold that is either closed or convex at infinity. We only constructed the structure for Lagrangians in $\C^n$ because of  technical simplifications.
If one works with suitable chain-level intersection theory and virtual techniques, one could expect that Theorem \ref{Thm: Main} to be upgraded to apply to Hamiltonian displaceable Lagrangians in geometrically bounded symplectic manifolds.

\begin{proof}
	[A heuristic proof of how the idealized Theorem \ref{IdealThm} implies Theorem \ref{Thm: Main}]
	Assume $L$ is a Lagrangian in $\C^n$ with vanishing Maslov class. Then $\mm_0=0$, and thus $\mm_1^2=0$.
	Using the energy filtration $\{\ms{F}^\lambda \wh{C_*\Omega_\star L}\}_{\lambda\in \R}$ and gappedness of the energy spectrum, we construct a $\Z$-filtration $\{\mf{F}^q \wh{C_*\Omega_\star L}\}_{q\in \Z}$ on $\wh{C_*\Omega_\star L}$ by choosing a sufficiently fine subdivision of $\R$, e.g. by choosing $\lambda_0>0$ satisfying a condition analogous to Condition 6.3.16 in \cite{FOOO1}, and set
	\begin{align*}
		\mf{F}^q\wh{C_*\Omega_\star L}:= \ms{F}^{q\lambda_0} \wh{C_*\Omega_\star L},\quad q\in\Z.
	\end{align*} 
	Then take the associated spectral sequence. The $E_1$-page computes the (completion of) the ordinary homology $\wh{H_*\Omega_\star L}$ of the based loop space, whereas the $E_\infty$-page computes the $\mm_1$-homology, which is 0 by the existence of a primitive $\NN^\star$ of the identity class $\underline{\star}$ under the differential $\mm_1$, according to Theorem \ref{IdealThm}.
	Now $0\neq \underline{\star}\in \wh{H_0\Omega_\star L}$ in $E_1$-page, and since $\deg  \mm_1 = -1$ there needs to be some non-zero elements in $\wh{H_1\Omega_\star L}$ which kills $\underline{\star}$ in the spectral sequence. 
	However by the assumption that $\pi_2 L=0$, it follows that $\wh{H_1\Omega_\star L}=0$, which gives a contradiction.

	One can also prove Theorem \ref{Thm: Main} using a filtered version of the homological perturbation lemma, similar to section 2 of \cite{Irie2} or section 5.4 of \cite{FOOO1}.
\end{proof}

\bigskip
The paper is organized as follows.
Section \ref{Sec: SketchPf} contains a sketch of the construction in Theorem \ref{IdealThm}, emphasizing  geometric motivations and related works.
Section \ref{Sec: CO map} is devoted to the construction of various structures in chain-level string topology  (in particular a closed-open map) needed for the proof, with detailed verifications of properties (especially signs) relegated to Appendix \ref{Sec: OCchainSigns}.
Section \ref{Sec: HolomorphicCurves} is the construction of the curved dg algebra by incorporating contributions from the virtual fundamental chains on moduli spaces of holomorphic curves, with constructions of the relevant virtual fundamental chains relegated to Appendix \ref{Sec: Kuranishi}. 
Section \ref{Sec: MainProof} contains the proof of the main theorem.

\bigskip

\Acknowledgements{I would like to thank my advisor Mark McLean for constant support and encouragement. I would also like to thank Mohammed Abouzaid, Jiaji Cai, Spencer Cattalani, Yash Deshmukh, Ceyhun Elmac{\i}o\u{g}lu, Kenji Fukaya, Sebastian Haney, Kei Irie, John Pardon, Dennis Sullivan, Chris Woodward, Guangbo Xu, and Frank Zheng for useful conversations or correspondences at various stages of the project.
This paper was partially supported by NSF award DMS-2203308 and also by Simons Foundation International, LTD.
\bigskip

\section{Heuristics of the  construction and related works}\label{Sec: SketchPf}

 In this section, we present a sketch of proof ignoring technical issues such as transversality (in chain-level string topology and moduli spaces of curves), with emphasis on geometric ideas and motivations. For example, we do not specify the chain model for the loop spaces (the reader may take $C_*\Omega_\star L$ to mean singular chain complex) and we assume that all the intersections are transversal. The string topology operations are presented in a style similar to the exposition in \cite{ChasSullivan}.

\subsection{Conventions and notations}
Throughout, all mentions of ``manifold'' mean manifold-without-boundary, unless otherwise specified.
We denote the de Rham complex of a manifold $M$ to be $\mathscr{A}^*(M)$ and the subcomplex of compactly-supported forms to be $\mathscr{A}^*_c(M)$.
We shall work with real coefficients $\K=\R$: for a space $X$, unless otherwise specified, $H_*(X)$ (\emph{resp.} $H^*(X)$) means the homology (\emph{resp.} cohomology) of $X$ with $\R$-coefficient.
We shall frequently identify $H_1(L;\Z)$ with $H_2(\C^n,L;\Z)$, and $\pi_1(L)$ with $\pi_2(\C^n,L)$, without further mention.

Let $L$ be a closed oriented manifold of dimension $n$, with a fixed basepoint $\star \in L$.  
We identify $S^1$ as the unit circle in $\R^2$, and fix a marked point $\star_{S^1}=1\in S^1$.

Fix a Lagrangian embedding $L\hookrightarrow \C^n$. Recall the two associated invariants:
\begin{enumerate}
	\item The symplectic energy $E\equiv E_L\colon \pi_1(L)\to \R$ (or $E\equiv E_L\colon H_1(L;\Z)\to \R$);
	\item The Maslov class $\mu\equiv \mu_L\colon \pi_1(L)\to \Z$ (or $\mu\equiv \mu_L\colon H_1(L;\Z)\to \Z$).
\end{enumerate}

\subsection{Holomorphic discs and string topology}
Our construction uses moduli spaces of holomorphic discs \cite{GromovPseudo}.
For each $k\in \Z_{\geq 0}$ and $\beta\in H_2(\C^n, L;\Z)$, let $\ms{M}_{k+1}(\beta)$ be the compactified moduli space of holomorphic discs (see e.g. section 7.2.2 of \cite{Irie2} for a complete definition). 
Its ``main stratum'' is the \emph{uncompactified moduli space of pseudo-holomorphic disc}, $\mathring{\ms{M}}_{k+1}(\beta)$, which we describe as follows (see also section 7.2.1 of \cite{Irie2}). When $\beta = 0$ and $k=0$ or $1$, define $\mathring{\ms{M}}_k(\beta)=\emptyset$. 
Otherwise, define $\mathring{\ms{M}}_{k+1}(\beta)$ as the space of equivalence classes of  tuples $(u, z_0,\dots, z_k)$ where 
\begin{itemize}
	\item $u\colon (\D^2, \partial \D^2)\to (\C^n,L)$ is a smooth map satisfying $\bar\partial u = 0$ and $u_*[\D^2] = \beta\in H_2(\C^n,L)$;
	\item $z_0,\dots, z_k\in \partial \D^2$ are distinct boundary marked points aligned in anti-clockwise order;
	\item such that for each automorphism $\rho \in \Aut(\D^2)$, we identify \[(u,z_0,\dots, z_k)\sim (u\circ \rho, \rho^{-1}(z_0),\dots, \rho^{-1}(z_k)).\]
\end{itemize}
For each $j=0,\dots, k$, define the \emph{evaluation map}
\[
	\ev_j\colon \mathring{\ms{M}}_{k+1}(\beta)\to L,\quad \quad \ev_j(u, z_0\dots, z_k):= u(z_j). 
\]
Then the ``codimension-1 stratum'' of the 1-marked moduli space  $\ms{M}_{1}(\beta)$ can be described as the fibre product \begin{align}\label{Eq: Codim1Boundary} \partial^1 \ms{M}_1(\beta)\cong  \bigsqcup_{\beta_1+\beta_2=\beta} \mathring{\ms{M}}_2(\beta_1) \fiberprod{\ev_1}{\ev_0} \mathring{\ms{M}}_1(\beta_2) \end{align}
(where in the above expression $\partial^1$  denotes the codimension-1 boundary).

In \cite{FukayaLagrangian}, Fukaya proposed the following perspective relating these moduli spaces with string topology:
\begin{enumerate}
	\item View $\ms{M}_{1}(\beta)$ as a chain in the free loop space $\ms{L}L := \Map(S^1, L)$. More precisely, one defines a map \[ \ev\colon \ms{M}_1(\beta)\to \ms{L}L \]and pushforward the ``virtual fundamental chain'' on $\ms{M}_1(\beta)$ to obtain a chain \[\MM(\beta):= \ev_* [\ms{M}_1(\beta)]\in C_{\dim L + \mu_L(\beta) -2}\ms{L}L;\]
	\item View the codimension-1 degeneration \eqref{Eq: Codim1Boundary} as a (chain-level) string topology operation --- specifically, let \[ [-,-]\colon C_*\ms{L}L \otimes C_*\ms{L}L \to C_{*+1-\dim L}\ms{L}L \] be a chain-level refinement of the loop bracket (see e.g. \cite{ChasSullivan}).
\end{enumerate}
Then \eqref{Eq: Codim1Boundary} reads \begin{align*}
		\partial \MM(\beta) - \frac{1}{2}\sum_{\beta_1+\beta_2=\beta}[\MM(\beta_1),\MM(\beta_2)] = 0 ,
	\end{align*}
	or, if we put together all the classes as an element $\MM:= \sum_\beta \MM(\beta) \in \wh{C_*\ms{L}L}$ in an appropriate completion of $C_*\ms{L}L$, the equation becomes a \emph{Maurer-Cartan equation} \begin{align*}
		\partial \MM - \frac{1}{2}[\MM,\MM] = 0.
	\end{align*}

This can then be interpreted as a deformation of the dg Lie algebra structure on $\wh{C_*\ms{L}L}$ provided by the loop bracket. This is used in \cite{FukayaLagrangian,Irie2} to prove a generalized version of Audin's conjecture as well as a classification of orientable, closed, prime 3-manifolds admitting a Lagrangian embedding into $\C^3$.

\subsection{Perturbed based holomorphic discs and based loops}\label{Subsc: PerturbedHolom}
In \cite{OhDisjEnergy}, Oh defined certain moduli spaces to study the displacing energy of a Lagrangian. 
We describe their codimension-0 strata as follows. 

Fix a basepoint $\star\in L$.
Let $H\in C_c^\infty(\C^n\times [0,1]_t)$ be a compactly supported time-dependent displacing Hamiltonian function satisfying Assumption \ref{Asmp: Disp Hamil}. 
Let $\chi\colon\R\to [0,1]$ be a smooth function such that 
\begin{itemize}
	\item $\chi(s) \equiv 0$ for $s\leq 0$;
	\item $\chi(s) \equiv 1$ for $s\geq 1$.
\end{itemize}
Define, for each $r\geq 0$,
\begin{align*}
	\chi_r(s):= \chi(r+s)\chi(r-s).
\end{align*}
By identifying $\D^2\setminus\{\pm 1\}$ with $\R_s\times [0,1]_t$, we obtain two coordinate functions $s\colon \D^2\setminus\{\pm 1\}\to \R$ and $t\colon \D^2\setminus\{\pm 1\}\to [0,1].$
For each $k\in \Z_{\geq 0}$, $\eta\in H_2(\C^n,L)$, we define the \emph{uncompactified moduli space of perturbed pseudo-holomorphic disc} $\mathring{\ms{N}}^\star_{k+1}(\eta)$ as the space of tuples $(r, u,z_0=1, z_1, \dots, z_k)$ where 
\begin{enumerate}
	\item $r\in \R_{\geq 0}$;
	\item $u\colon (\D^2, \partial \D^2)\to (\C^n,L)$ is a smooth map satisfying the perturbed pseudo-holomorphic equation \[ \bar\partial_r (u):= \left(du - X_{\chi_r(s)H_t}(u)\otimes dt\right)^{0,1} = 0 \] and $u_*[\D^2] = \eta\in H_2(\C^n,L)$;
	\item $1=z_0, z_1,\dots, z_k\in \partial \D^2$ are distinct boundary marked points aligned in anti-clockwise order;
	\item $u(z_0)=\star$.
\end{enumerate}

These spaces are equipped with evaluation maps
\begin{align*}
	\ev_j\colon \mathring{\ms{N}}^\star_{k+1}(a)\to L,\quad \textup{ given by } \ev_j(u,z_0,\dots, z_k):= u(z_j)
\end{align*}
for each $j=1,\dots, k$.

This construction is used in \cite{FukayaLagrangian} and is basically the same as in \cite{Irie2}, section 7.2.1 and 7.2.2, \emph{except  we require the marked point $z_0$ to pass through the basepoint $\star\in L$ (condition (4) above)}. Basically the same moduli space without condition (4) is also used by Abouzaid in \cite{AbouzaidFramedBordism}.
Also see \cite{BiranCieliebak, PolterovichMonotone}.

The ``codimension-1 stratum'' of the 1-marked such moduli spaces $\ms{N}^\star_1(\eta)$ consists of the following boundaries:
\begin{enumerate}
	\item \emph{Bubbling}: A pseudo-holomorphic disc can bubble off from these perturbed pseudo-holomorphic discs, and depending on where $z_0$ is positioned, we have the two parts in the codimension-1 boundary: \begin{align}\label{Eq: Pert Codim1 Part1}
		\bigsqcup_{\eta_1+\beta_2=\eta}\ms{N}_2^\star(\eta_1) \fiberprod{\ev_1}{\ev_0} \ms{M}_1(\beta_2)	\end{align}
		and \begin{align}\label{Eq: Pert Codim1 Part2}
			 \bigsqcup_{\beta_1+\eta_2=\eta} \ms{M}_2^\star(\beta_1) \fiberprod{\ev_1}{\ev_0} \ms{N}_1(\eta_2) 
		\end{align}
	where the moduli spaces in the above expressions with the $\star$ superscript consist of curves satisfying $u(z_0)=\star$, and ones without consist of curves not necessarily satisfying $u(z_0)=\star$. For example, $\ms{M}_2^\star(\beta_1):= \ev_0^{-1}(\ms{M}_2(\beta_1))$.
	\item \emph{$r=0$}: When the deformation parameter $r$ becomes 0, the perturbed Cauchy-Riemann operator $\ol{\partial}_0=\bar\partial$ becomes the unperturbed Cauchy-Riemann operator, so the curves in this boundary $\ms{N}_1^{0,\star}(\eta)$ are simply (unperturbed) pseudo-holomorphic curves. In particular, $\ms{N}_1^{0,\star}(\eta)=\emptyset$ unless (i) $E(\eta)>0$ or (ii) $\eta = 0$. Moreover, in case $\eta=0$, $\ms{N}_1^{0,\star}(0)$ is the singleton set consisting of the constant map to $\star\in L$.
\end{enumerate}
That is, 
\begin{align}\label{Eq: Pert Codim1Bdry}
	&\partial^1 \ms{N}_1^\star(\eta) \\ =\,\, &\ms{N}_1^{0,\star}(\eta) \sqcup 		\left(\bigsqcup_{\eta_1+\beta_2=\eta}\ms{N}_2^\star(\eta_1) \fiberprod{\ev_1}{\ev_0} \ms{M}_1(\beta_2)\right) \sqcup \left( \bigsqcup_{\beta_1+\eta_2=\eta} \ms{M}_2^\star(\beta_1) \fiberprod{\ev_1}{\ev_0} \ms{N}_1(\eta_2) \right).\nonumber
\end{align}

Following Fukaya's proposal outlined above, we interpret these moduli spaces in terms of string topology:
\begin{enumerate}
	\item View $\ms{N}_1^\star(\eta)$ as a chain in a loop space, but this time the based loop space $\Omega_\star L:= \Map((S^1,\star_{S^1}), (L, \star))$ where $\star_{S^1}$ is a fixed marked point in $S^1$: there is an evaluation map \[ \ev\colon \ms{N}_1^\star(\eta)\to \Omega_\star  L\]and we pushforward the virtual fundamental chain on  $\ms{N}_1^\star(\eta)$ to obtain a chain \[ \NN^\star(\eta):= \ev_*[\ms{N}_1^\star(\eta)] \in C_{\mu_L(\eta)+1} \Omega_\star L; \]
	Also, define \begin{align*}
		\NN^{0,\star}(\eta):= \ev_*[\ms{N}_1^{0,\star}(\eta)]\in C_{\mu_L(\eta)}\Omega_\star L.
	\end{align*}
	\item View the codimension-1 degeneration \eqref{Eq: Pert Codim1 Part1} as a (chain-level) string topology operation --- specifically, for each $\beta\in H_2(\C^n,L;\Z)$, define \begin{align*}
		\mm_{1,\beta}\colon C_*\Omega_\star L\to C_{*+\mu_L(\beta)-1}\Omega_\star L
	\end{align*}as following, assuming we are using singular chains and assuming transversality for now. Given a  $k$-chain $K_\alpha\xrightarrow{\alpha} \Omega_\star L$ where $K_\alpha$ is the underlying domain of the chain $\alpha$, define $\mm_{1,\beta}\alpha$ to be the chain $K_{\mm_{1,\beta}\alpha}\xrightarrow{\mm_{1,\beta}\alpha} \Omega_\star L$ where $K_{\mm_{1,\beta}}$ is the pre-image of the diagonal of $L\times L$ under \begin{align}\label{Eq: BasedLoopDegen1}
		K_\alpha \times [0,1]\times \ms{M}_1(\beta) &\to L\times L\\
		(k_x, s,u) &\mapsto (\alpha(k_x)(s), \ev_0(u)=u(1))\nonumber
	\end{align} 
	where we view $\alpha(k_x)$ as a map $S^1\xrightarrow{\alpha(k_x)} L$ and $u$ as a map $(\D^2,\partial \D^2)\xrightarrow{u} (\C^n,L)$ (where $\D^2$ is identified as the unit disc of $\R^2$). Then define $K_{\mm_{1,\beta}\alpha}\xrightarrow{\mm_{1,\beta}\alpha} \Omega_\star L$ by \begin{align}\label{Eq: BasedLoopDegen2}
		\big((\mm_{1,\beta}\alpha) (k_x,s,u)\big)(\tau):= \begin{cases}
			\alpha(k_x)(2\tau), & \tau \in [0,\frac{s}{2}) \\
			u(e^{2\pi i(2\tau -s)}), & \tau \in [\frac{s}{2}, \frac{s+1}{2})\\
			\alpha(k_x)(2\tau-1), & \tau \in [\frac{s+1}{2}, 1]
		\end{cases}.
	\end{align}
	This is analogous to the pre-Lie product $*$ in section 3 of \cite{ChasSullivan}. See Figure \ref{fig:m1}.
	\item The codimension-1 degeneration \eqref{Eq: Pert Codim1 Part2} doesn't have a good description in terms of chains in the based loop space because $\ms{N}_1(\eta_2)$ is not part of $C_*\Omega_\star L$; it is contributed by moduli spaces $\ms{M}_1^\star(\beta):= \ev_0^{-1}(\ms{M}_1(\beta))$ of holomorphic discs whose boundary pass through $\star\in L$. We record this data in the form of \[  \mm_{0,\beta}:= \ev_*[\ms{M}_1^\star(\beta)]\in C_{\mu_L(\beta)-2}\Omega_\star L, \] where $\ev\colon \ms{M}_1^\star(\beta)\to \Omega_\star L$ is the evaluation map. See Figure \ref{fig:m0}. 
\end{enumerate}

Recall that based loop concatenation endows the based loop space $\Omega_\star L$  with a product (assume for simplicity that we are using a model of based loop space where the product is strictly associative, e.g. the Moore loop space):
\begin{align*}
	\bullet\colon \Omega_\star L\times \Omega_\star L\to \Omega_\star L
\end{align*}
which makes the chains on $\Omega_\star L$ into a dg associative algebra $(C_*\Omega_\star L, \partial, \bullet)$ with the Pontryagin product \begin{align*}
	\bullet \colon C_*\Omega_\star L\otimes C_*\Omega_\star L\to C_*\Omega_\star L.
\end{align*}
This dg associative algebra has a unit $\underline{\star}\in C_0\Omega_\star L$, which is the 0-cycle given by the constant loop at the basepoint $\star\in L$.

Then, equation \eqref{Eq: Pert Codim1 Part1} is translated into the equation 
\begin{align}\label{Eq: Pert Codim1ChainLevel}
	\partial \NN^\star(\eta) + \sum_\beta \mm_{1,\beta} (\NN^\star(\eta -\beta)) +(\textup{terms involving }\mm_{0,\beta}) = \NN^{0,\star}(\eta),
\end{align}
or, if we define, in an appropriate completion $\wh{C_*\Omega_\star L}$,\begin{align*}
	\NN^\star:= \sum_\eta \NN^\star(\eta),\quad \mm_1:= \partial + \sum_\beta \mm_{1,\beta},\quad \mm_0:= \sum_\beta \mm_{0,\beta},
\end{align*} and notice that \begin{align*}
	\NN^{0,\star}(\eta)=\begin{cases}
		\underline{\star}, & \eta = 0\\
		0, & E(\eta)<0
	\end{cases},
\end{align*}
summing up equations \eqref{Eq: Pert Codim1ChainLevel} for all $\eta$ gives 
\begin{align}\label{Eq: unitIsPrimitive}
	\mm_1\NN^\star + (\textup{terms involving }\mm_0) = \underline{\star} + (\textup{terms of energy }E>\hbar)
\end{align}
where $\hbar>0$ is a constant (given by the minimal holomorphic disc energy).
The right hand side is invertible in the completion $\wh{C_*\Omega_\star L}$ so we might as  well pretend it is the unit $\underline{\star}$.

\bigskip
We have now defined the structures in our idealized Theorem \ref{IdealThm}. Most of the properties follow from the description (in an ideal situation, ignoring technical issues). For example, when $\mu_L=0$, the degree of $\mm_1$ follows directly from dimension-counting, and $\deg \mm_0 = -2$ so it should not contribute (homologically), as $H_*\Omega_\star L$ is concentrated in non-negative degrees. Property (3) follows from \eqref{Eq: unitIsPrimitive}, since $\mm_0$ vanishes in this situation and the right-hand side is invertible.
We now explain the heuristic for the proof of the identity $\mm_1^2(\alpha) = [\mm_0,\alpha]$ (where $\alpha \in C_*\Omega_\star L$), modulo signs.

Notice that the identity $\mm_1^2(\alpha)=[\mm_0,\alpha]$ is equivalent to \begin{align}\label{Eq: AinftyIdenEquiv}
	\partial\left(\mm_{1,\beta} \alpha\right) + \mm_{1,\beta}\left(\partial \alpha\right) + \sum_{\beta_1+\beta_2=\beta}\mm_{1,\beta_1}\left(\mm_{1,\beta_2} \alpha\right) = \mm_{0,\beta} \bullet \alpha  - \alpha \bullet \mm_{0,\beta}
\end{align}
for each $\beta\in H_2(\C^n,L;\Z)$.
By our explicit descriptions \eqref{Eq: BasedLoopDegen1} and \eqref{Eq: BasedLoopDegen2} of $\mm_{1,\beta}\alpha$, its boundary $\partial\mm_{1,\beta}\alpha$ consists of 
\begin{enumerate}
	\item The part of $K_{\mm_{1,\beta}}\subset K_\alpha \times [0,1] \times \ms{M}_1(\beta)$ contained in $\partial K_\alpha \times [0,1]\times \ms{M}_1(\beta)$, which is responsible for the term $\mm_{1,\beta}(\partial \alpha)$;
	\item The part contained in $K\times [0,1] \times \partial \ms{M}_1(\beta)$, where $\partial \ms{M}_1(\beta)$ is given in \eqref{Eq: Codim1Boundary}; this is responsible for the term $\sum_{\beta_1+\beta_2=\beta} \mm_{1,\beta_1}(\mm_{1,\beta_2}\alpha)$ (technically $\mm_{1,\beta_1}(\mm_{1,\beta_2}\alpha)$ consists also of compositions of loops where curves from $\ms{M}_1(\beta_1)$ and $\ms{M}_1(\beta_2)$ lands on two different points of curves in $\alpha$, but they cancel with the same term coming from $\mm_{1,\beta_2}(\mm_{1,\beta_1}\alpha)$ in the same way similar to Lemma 4.2 in \cite{ChasSullivan});
	\item The part contained in $K_\alpha \times \{0,1\} \times \ms{M}_1(\beta)$; for an element $(k_x, s, u)$ in this (where $s\in \{0,1\}$), we have $u(\star_{S^1}) = \star \in L$ and acccording to \eqref{Eq: BasedLoopDegen2}, the part of the boundary in  $K_\alpha \times \{0\}\times \ms{M}_1(\beta)$ gives the term $\mm_{0,\beta}\bullet \alpha$ and the part in $K_\alpha \times \{1\}\times \ms{M}_1(\beta)$  gives the term $\alpha \bullet \mm_{0,\beta}$.
\end{enumerate}

This verifies \eqref{Eq: AinftyIdenEquiv}.

\subsection{Open-closed string topology}\label{Subsc: OCStringSurvey}
The way we construct the curved dg algebra is via the construction of a closed-open map in chain-level string topology. We comment on the motivation and other works relevant to this construction.

A general theme in Floer-Fukaya type theories of Lagrangian submanifolds is that \emph{pseudo-holomorphic curves provide quantum deformations of classical algebraic topology.}
For example, the Lagrangian Floer (co)homology is an $A_\infty$-deformation of the ordinary (co)homology of the Lagrangian submanifold \cite{FOOO1}. 
The work of \cite{FukayaLagrangian,Irie1,Irie2} provided a conceptual explanation for Lagrangian Floer cohomology as a deformation of ordinary cohomology via the iterated integration map of Chen \cite{ChenIterIntegral}:
\begin{align*}
	C_*\ms{L}L \to \CC^*(C^*L,C^*L)
\end{align*} 
where $\CC^*(C^*L,C^*L)$ is the Hochschild cochain complex of the dg associative algebra $C^*L$ (with the cup product).

In e.g. \cite{ChasSullivan,ChasSullivan2,Sullivan,SullivanStringSurvey}, a rich collection of algebraic operations on various loop/path spaces are uncovered, including interactions between closed sector (e.g. free loop space) and the open sector (e.g. based loop space), under the name of \emph{open-closed string topology}. 
Algebraic frameworks and detailed constructions of various parts of this are done in e.g. \cite{MooreSegal,CostelloTCFT,Godin,BlumbergCohenTeleman,CohenGanatra}.
Combining ideas from open-closed string topology and work of \cite{FukayaLagrangian,Irie2}, we construct a homomorphism of dg Lie algebras, a closed-open string map
\begin{align}\label{Eq: StringTopOCMap}
	C_*\ms{L}L \to \CC^*(C_*\Omega_\star L, C_*\Omega_\star L)
\end{align}
where $\CC^*(C_*\Omega_\star L, C_*\Omega_\star L)$ is the Hochschild cochain complex of the dg associative algebra $C_*\Omega_\star L$ (with the Pontryagin product). See Theorem \ref{Thm: OCstringPackage}.

\begin{remark}
	The relation between homology of free loop spaces and Hochschild cohomology of chains on based loop spaces is first studied by \cite{Goodwillie,BurgheleaFiedorowicz}. Also see \cite{MalmThesis}.
\end{remark}

The Maurer-Cartan element $\MM\in \wh{C_*\ms{L}L}$ constructed in \cite{FukayaLagrangian,Irie2} out of moduli spaces of holomorphic discs is then pushed forward to a Maurer-Cartan element in $\wh{\CC^*}(C_*\Omega_\star L, C_*\Omega_\star L)$, a suitable completion of the Hochschild complex, and which is interpreted as a curved deformation of the dg associative algebra $C_*\Omega_\star L$ using deformation theory. See section \ref{Subsc: FiltrationGapped} for details.

\begin{remark}
	The map \eqref{Eq: StringTopOCMap} can be viewed as a string topology version of the closed-open map in e.g. \cite{AbouzaidGeneration,GanatraThesis}. In the context of symplectic Floer theory, the Maurer-Cartan element can be viewed in terms of a Borman-Sheridan class using (an extension of) the construction in \cite{BormanSheridan}. One could then possibly use this language to generalize our result to certain singular Lagrangian submanifolds.
\end{remark}

\bigskip

\section{A closed-open map in string topology}
\label{Sec: CO map}

\subsection{Summary of structures}
Let $L$ be a closed oriented manifold of dimension $n$, with a fixed basepoint $\star \in L$.  
The ingredients of the constructions in this section are purely topological and are intrinsic to $L$ itself, but we are going to fix a Lagrangian embedding $i_L\colon L\hookrightarrow \C^n$ (not assumed to have vanishing Maslov class), where $\C^n$ is endowed the standard symplectic structure, which we shall use to define the gradings of our constructions for later convenience.
Specifically we will use the Maslov index $\mu:=\mu_L\colon H_1(L;\Z)\to \Z$ of the Lagrangian embedding.

As before, denote by 
\[
	\mathscr{L}L:= \textup{Map}(S^1,L)
\]
the free loop space of $L$, and 
\[
	\Omega_{\star} L:= \textup{Map}((S^1,\star_{S^1}), (L,\star))
\]
the based loop space of $L$.
For each $a\in H_1(L;\Z)$, define $\ms{L}(a)$ (\emph{resp.} $\Omega_\star (a)$) to be the space of loops in $\ms{L}L$ (\emph{resp.} $\Omega_\star L$) with homology class $a$.
We are only concerned with the homotopy types of these spaces, which are independent of the regularity of the loops (see e.g. section 2.1 of \cite{ChataurOanceaBasic}), so we do not specify the regularity.

\begin{theorem}[Open-closed string topology package]
\label{Thm: OCstringPackage}
	Associate to a closed oriented manifold $L$ of dimension $n$ together with a Lagrangian embedding $i_L\colon L\hookrightarrow \C^n$ are the following structures: 
	\begin{enumerate}
		\item (Closed string state space) A dg Lie algebra $\CFr$ with a decomposition: \[ \CFr:= \bigoplus_{a\in H_1(L;\Z)} \CFr(a),\quad \partial \colon \Cfr{*} \to \Cfr{*}[-1],\quad   [-,-]\colon \CFr\otimes \CFr\to \CFr,\]which computes the free loop space homology $H_*(\mathscr{L}L)$\footnote{One could expect that the dg Lie algebra descends to the Chas-Sullivan Lie algebra on $H_*(\ms{L}L;\K)$ in Proposition 4.3 of \cite{ChasSullivan}, although this is not relevant for our purpose. See the second remark in Section 2.5 of \cite{WangThesis}.} up to grading shifts. More precisely, for each $a\in H_1(L;\Z)$, 
		\begin{align*}
			H_*(\CFr(a),\partial)\cong  H_{*+ n + \mu(a) - 1}(\ms{L}(a)).
		\end{align*}
		\item (Open string state space) A dg  associative algebra with a decomposition: \[C_*^{\Omega_\star} = \bigoplus_{a\in H_1(L;\Z)} C_*^{\Omega_\star}(a), \quad \partial\colon C_*^{\Omega_\star} \to C_*^{\Omega_\star}[-1], \quad \bullet\colon C_*^{\Omega_\star}\otimes C_*^{\Omega_\star}\to \CBa,\]which computes the based loop space homology $H_*(\Omega_\star L)$  together with the Pontryagin (concatenation) product, up to a grading shift. More precisely, for each $a\in H_1(L;\Z)$, 
		\begin{align*}
			H_*(\CBa(a),\partial) \cong H_{*+\mu(a)}(\Omega_\star (a)),
		\end{align*}
		and 
		\begin{align*}					\bigoplus_{a\in H_1(L;\Z)} \HBa(a) \cong \bigoplus_{a\in H_1(L;\Z)} H_{*+\mu(a)}(\Omega_\star(a))
		\end{align*}
		as graded algebras.
		\item (Closed-open string map) A homomorphism between dg Lie algebras preserving the decomposition into $H_1(L;\Z)$ classes: \[ \CO\colon \CFr\to \CC^*(\CBa, \CBa) \]
		where $\CC^*(\CBa,\CBa)$ is a model of Hochschild cochains together with the Hochschild differential and the Gerstenhaber Lie bracket (see section \ref{Subsubsc: Hochschild} for the precise definitions).
	\end{enumerate}
\end{theorem}

\begin{remarks}
	\begin{enumerate}
		\item For grading and sign conventions regarding dg algebras, see section \ref{Subsc: ConventionDGalg}.
		\item This structure is only a small part of the open-closed string topology package; see the discussion and references in section \ref{Subsc: OCStringSurvey}, as well as Remarks \ref{Rmk: RichStructure} and \ref{Rmk: RichStructure2}.
		\item An ``open-closed map'' in string topology is constructed in section 4 of \cite{AbouzaidCotangent}; our construction follows the same vein, as well as ideas from e.g. \cite{AbouzaidLect}. The technical framework of the proof of Theorem \ref{Thm: OCstringPackage} is heavily inspired by \cite{Irie2,WangThesis}.
	\end{enumerate}
\end{remarks}

For readability, the main text in this section is mainly focused on constructions of the relevant structures, with verifications of properties relegated to Appendix \ref{Sec: OCchainSigns}.

The organization of this section is as follows. In section \ref{Subsc: dR Chain}  we review the de Rham chain complex construction in \cite{Irie1} which we will use to build our chain models. In section \ref{Subsc: StateSpace} we define the chain complexes $(\CFr,\partial)$ and $(\CBa,\partial)$ in part (1) and (2) of Theorem \ref{Thm: OCstringPackage}, following the approach of \cite{WangThesis}. In section \ref{Subsc: StringOper} we define various string topology operations on de Rham chains, which we put together in section \ref{Subsc: StringOperState} to obtain the stated structures in Theorem \ref{Thm: OCstringPackage}.

\subsection{Irie's machinery of de Rham chains}\label{Subsc: dR Chain}
The machinery we shall use for solving chain-level transversality problems is the formalism of de Rham chains, first proposed by Fukaya (under the name ``approximate de Rham chains'' in \cite{FukayaLagrangian} ) and rigorously constructed by Irie \cite{Irie1,Irie2}. 
We now briefly review this machinery on free and based loop spaces,  heavily borrowing from \cite{Irie1}.
For conventions on orientations, see section \ref{Subsc: Ori}.

We are going to use ``the space of all manifolds'' as domains of the chains. To avoid set-theoretic issues, let $\mathscr{U}$ be the set of all oriented $m$-dimensional submanifolds of $\R^N$ for all $0\leq m \leq N$.

\begin{definition}[\cite{Irie1}, Section 4.2]
	A \emph{differentiable space} is a set $X$ together with a collection of maps, called \emph{plots}, from $U\in \mathscr{U}$ to $X$, satisfying the property that
\begin{itemize}
	\item If $\varphi\colon U\to X$ is a plot, $U'\in \mathscr{U}$ and $\theta\colon U'\to U$ is a submersion, then $\varphi\circ \theta\colon U'\to X$ for $X$ is a plot for $Y$.
\end{itemize}
A map $f\colon X\to Y$ between differentiable spaces is said to be \emph{smooth} if it is a map of sets such that composition of $f$ with a plot $\varphi\colon U\to X$ is a plot.
\end{definition}
\begin{example}[\cite{Irie1}, Example 4.2 (i)]\label{Ex: DiffStr on Manifolds}
	Let $L$ be a smooth manifold. 
	We define a differentiable space with underlying set $L$ itself by stipulating a map $\varphi\colon U\to L$ to be a plot if $\varphi$ is smooth.
\end{example}

\begin{definition}\label{Defn: dRchainCplx}[\cite{Irie1}, Section 4.3]
	Let $X$ be a differentiable space. Define the \emph{de Rham chain complex} of $X$ to be
	\[ C_j^\dR(X):= \left(\bigoplus_{(U,\varphi)} \mathscr{A}_c^{\dim U-j}(U)\right)\bigg/ \sim\]
where 
\begin{itemize}
	\item The direct sum is taken over all $(U,\varphi)$ where $U\in \mathscr{U}$ and $\varphi\colon U\to X$ is a plot;
	\item The notation $\mathscr{A}_c^{\dim U-j}(U)$ is the space of compactly supported smooth forms on $U$ of degree $\dim U-j$. Denote an element $\omega\in \mathscr{A}_c^{\dim U-j}(U)$ belonging to the the summand labeled by $(U,\varphi)$ to be $(U,\varphi,\omega)$; frequently, to signify the target space $X$, we will also use the notation $(U\xrightarrow{\varphi} X; \omega)$ instead to mean the same object as $(U,\varphi,\omega)$.
	\item The equivalence relation is taken by quotienting out the subspace generated by \begin{align}\label{Eq: dRChainRelation}(U,\varphi, \pi_!\omega)- (U',\varphi\circ \pi, \omega)\end{align}
	where $U,U'\in \mathscr{U}$, $\pi\colon U'\to U$ is a submersion, $\omega \in \mathscr{A}_c^*(U')$, and $\pi_!\omega$ denotes integration over the fibres (which is well-defined since $\omega$ is compactly supported).
	Denote the equivalence class of the element $(U,\varphi, \omega)$  to be $[(U,\varphi, \omega)]\in C_j^\dR(X)$, or $[(U\xrightarrow{\varphi}X;\omega)]$. In later computations we often drop the square bracket when the meaning is clear for notational simplicity.
	\item The differential $\partial^\dR \colon C_*^\dR(X)\to C_{*-1}^\dR(X)$ is defined by \begin{equation}\label{Eq: DefnPartialDR}\partial^\dR [(U,\varphi, \omega)]:= (-1)^{|\omega|+1}[(U,\varphi, d\omega)].\end{equation}
\end{itemize}
\end{definition}

Then it follows immediately that a smooth map $f\colon X\to Y$ between smooth manifolds induces a pushforward \[ f_*\colon C_*^\dR(X)\to C_*^\dR(Y).\]

\begin{notation}\label{Ntn: degx}
For $x\in C_j^\dR(X)$, we write $\deg x:= j$.	
\end{notation}

\begin{remarks}\label{Rmk: SignOri}
	Consider $U\in \mathscr{U}$ and $\ol{U}$ denoting $U$ with the opposite orientation. The identity map $\id\colon U\to \ol{U}$ is a submersion, but recall sign convention of integration along the fibre (Section 4.2.3 of \cite{Irie2}) is such that for a submersion $\pi\colon X\to Y$, \[\int_Y \pi_!\omega\wedge \eta = \int_X \omega\wedge \pi^* \eta\]for any $\omega\in\mathscr{A}_c^*(X)$ and $\eta\in \mathscr{A}^*(Y)$. Therefore for any $\omega\in \mathscr{A}_c^*(U)$, under the orientation-reversing $\pi\colon U\to \ol{U}$, the pushout $\pi_!\omega =-\omega$. In particular, \[[(U,\varphi, \omega) ]=-[(\ol{U},\varphi,\omega)].\]
\end{remarks}

\subsection{Definition of the state spaces}\label{Subsc: StateSpace}

Let $L$ be as before.
Following the construction of Wang in \cite{WangThesis} (which is a generalization of Adams's cobar construction \cite{AdamsCobar} to non-simply-connected cases), we will construct cosimplicial models of the free and based loop spaces, to which we shall apply the machinery of de Rham chains in section \ref{Subsc: dR Chain} to obtain the chain complexes for the state spaces.

Let $\Pi_1 L$ be the fundamental groupoid:
\[
	\Pi_1 L := \{(p,q,[\sigma])\mid p,q\in L, [\sigma]\in \mathcal{P}_{p,q}/\textup{homotopy}\}.
\]
Denote the source and sink maps by
\begin{align*}
	\textsf{s}&\colon \Pi_1 L\to L;\quad \textsf{s}(p,q,[\sigma]):= p\\
	\textsf{t}&\colon \Pi_1 L\to L;\quad \textsf{t}(p,q,[\sigma]):= q.
\end{align*}

\begin{definition}\label{Defn: PathFundGroupoid}
	\begin{enumerate}
		\item Given two composable elements $(p,q,[\sigma]), (q,r,[\tau])\in \Pi_1 L$, denote their concatenation (the groupoid multiplication) by \begin{align}\label{Eq: Concat} (p,q,[\sigma])*(q,r,[\tau]):= (p,r,[\sigma * \tau]).\end{align}
		\item Given a point $y\in L$, denote by $\underline{y}\in \mc{P}_{p,q}$ the constant path at $y$, by $[\underline{y}]\in \mc{P}_{p,q}/\textup{homotopy}$ the homotopy class, and also by $[\underline{y}]$ the element $(y,y,[\underline{y}])\in \Pi_1 L$.
	\end{enumerate}
\end{definition}

Define 
\begin{itemize}
	\item $\mathscr{L}^{k+1}L$ denote the space of elements $(c_0,\dots, c_k)\in (\Pi_1L)^{k+1}$ such that $\textsf{t}(c_i) = \textsf{s}(c_{i+1})$ for all $i=0,\dots, k-1$, and $\textsf{t}(c_k)=\textsf{s}(c_0)$. 
	Then we have the evaluation maps \[ \ev=(\ev_0,\cdots, \ev_{k})\colon \mathscr{L}^{k+1} L \to L^{\times (k+1)},\quad (c_0,\dots, c_k)\mapsto (\textsf{s}(c_0),\textsf{s}(c_1),\dots, \textsf{s}(c_k)) ;\]
	\item $\Omega^{k+1}_\star L$, where $\star\in L$ is the basepoint, denote the space of elements $(c_0,\dots, c_k)\in (\Pi_1L)^{k+1}$ such that $\textsf{t}(c_i)=\textsf{s}(c_{i+1})$ for all $i=0,\dots, k-1$, and $\textsf{t}(c_k)=\textsf{s}(c_0) = \star$. Then we have the evaluation maps \[ \ev= (\ev_1,\dots, \ev_{k})\colon \Omega^{k+1}_\star L \to L^{\times k} ,\quad (c_0,\dots, c_k) \mapsto (\textsf{s}(c_1),\dots, \textsf{s}(c_k)). \]
\end{itemize}

\begin{lemma}\label{Lm: WangLoopSmooth}
	The maps $\mathscr{L}^{k+1}L\xrightarrow{\ev} L^{\times (k+1)}$, $\Omega^{k+1}_\star L\xrightarrow{\ev} L^{\times k}$ are covering maps.
\end{lemma}

Pulling back the smooth manifold structure on $L^{\times (k+1)}$ (\emph{resp.} $L^{\times k}$), we obtain a smooth manifold structure on $\mathscr{L}^{k+1}L$ (\emph{resp.} $\Omega_\star^{k+1} L$)  of dimension $(k+1)(\dim L)$ (\emph{resp.} $k(\dim L)$), so that each of the evaluation maps \[ \ev_i\colon \mathscr{L}^{k+1}L\to L, \quad \ev_j\colon \Omega_\star^{k+1} L \to L ,\quad\textup{ where } i=0,\dots, k, \textup{ and } j=1,\dots  k  \]
is a smooth map.

We now specify the structures of $\mathscr{L}^{k+1} L$ and $\Omega_\star^{k+1} L$ as differentiable spaces:

\begin{definition}[Differentiable spaces structures on $\mathscr{L}^{k+1}L$ and $\Omega^{k+1}_\star L$]\label{Defn: DiffSpaceOnLoops}For each $k\in \Z_{\geq 0}$,
\begin{itemize}
	\item For $\Omega^{k+1}_\star L$, we regard it as a differentiable space simply using its smooth structure (i.e. a map $\varphi\colon U\to \Omega^{k+1}_\star L$ is a plot if $\varphi$ is smooth; see Example \ref{Ex: DiffStr on Manifolds}).
	\item  For $\mathscr{L}^{k+1} L$, we need an additional constraint: a map $\varphi\colon U\to \mathscr{L}^{k+1}L$ is a plot if 
\begin{enumerate}
	\item $\varphi$ is smooth;
	\item $\ev_0\circ \varphi\colon U\to L$ is a submersion.
\end{enumerate}
\end{itemize}
\end{definition}
This allows us to define, for each $k\in \Z_{\geq 0}$, the de Rham chain complexes $(C_*^\dR(\mathscr{L}^{k+1} L), \partial^\dR)$ and $(C_*^\dR(\Omega_\star^{k+1} L), \partial^\dR)$.

\begin{lemma}\label{Lm: dRchainComputesHomology}
	Let $\mathcal{X}$ denote either of the symbols $\mathscr{L}$ or $\Omega_\star$.
	Then, for each $k\in\Z_{\geq 0}$,
	$(C_*^\dR(\mathcal{X}^{k+1} L), \partial^\dR)$ computes the ordinary homology $H_*(\mathcal{X}^{k+1} L)$.
\end{lemma}
We will be a bit sketchy in the proof about the $\mc{X}=\ms{L}$ case, since strictly speaking we only need the $\mc{X}=\Omega_\star$ part of the lemma for this paper.
\begin{proof}
	For $\mc{X}=\Omega_\star$, this is Theorem 5.1 of \cite{Irie1} since $\Omega^{k+1}_\star L$ is an oriented smooth manifold and the plots we used are just smooth maps.
	For $\mc{X}=\ms{L}$, since the plots used to define the de Rham chain complex need to satisfy additionally that the compositions with $\ev_0\colon \ms{L}^{k+1} L\to L$ are submersions (condition (2) in Definition \ref{Defn: DiffSpaceOnLoops}), we also need to show that the chain complex is quasi-isomorphic to the one defined using all smooth maps $U\to \ms{L}^{k+1}L$ as plots. This is a finite-dimensional analogue of Lemma 7.7 in \cite{Irie1} and the proof is completely analogous.
\end{proof}

\bigskip

Still following \cite{WangThesis},
we then construct cosimplicial spaces $\mc{X}L$ as a collection of spaces $\{\mc{X}^{k+1} L\}_{k\in \Z_{\geq 0}}$ together with the following structure maps (where $\mathcal{X}$ denotes either $\mathscr{L}$ or $\Omega_\star$):
\begin{align*}
	\delta_i &\colon \mathcal{X}^{k}L\to \mathcal{X}^{k+1} L; \quad \delta_i(c_0,\dots, c_{k-1}):= \begin{cases}
		(c_0,\dots, c_{i-1},\underline{\textsf{s}(c_i)}, c_i,\dots, c_{k-1}), & 0 \leq i \leq k-1\\
		(c_0,\dots, c_{k-1}, \underline{\textsf{t}(c_{k-1})}), & i=k
	\end{cases}\\
	\sigma_i &\colon \mathcal{X}^{k}L\to \mathcal{X}^{k+1}L;\quad \sigma_i(c_0,\dots, c_{k+1}) := (c_0,\dots, c_i*c_{i+1},\dots, c_{k+1}), \quad i=0,\dots, k,
\end{align*}
where for a point $y\in L$, $\underline{y}$ denotes the constant path at $y$, and $*$ denotes composition (see Definition \ref{Defn: PathFundGroupoid}).

\bigskip

For each $a\in H_1(L;\Z)$, recall that we define $\mathscr{L}(a)$ to be the component in $\mathscr{L}L$ consisting of loops whose homology class is $a$, and $\Omega_\star(a)$ to be the component in $\Omega_\star L$ consisting of loops whose homology class is $a$. 

Similarly, for each element $(c_0,\dots, c_k)$ in $\mathscr{L}^{k+1} L$ or $\Omega_\star^{k+1} L$, there is a well-defined homology class $a\in H_1(L)$ associated to $c_0*\dots * c_k\in \Pi_1L$ based at $\textsf{s}(c_0)=\textsf{t}(c_k)$, and denote by $\mathscr{L}^{k+1}(a)$ and $\Omega^{k+1}_\star (a)$ accordingly.

\begin{definition}
For each $a\in H_1(L;\Z)$ and $k\in \Z_{\geq 0}$, we define
\begin{align*}
	C^\mathscr{L}(a,k)_*:= C^\dR_{*+\dim L + \mu(a) + k - 1}(\mathscr{L}^{k+1	}(a))
\end{align*}
and
\begin{align*}
	C^{\Omega_\star}(a,k)_*:= C^\dR_{*+\mu(a)+k}(\Omega_\star ^{k+1}(a)).
\end{align*}
The \emph{closed-string state space} is then defined as
\[ \CFr:= \bigoplus_{a\in H_1(L;\Z)} \underbrace{\prod_{k\in \Z_{\geq 0} } C^\mathscr{L}(a,k)_*}_{=: \CFr(a)}\]
and the \emph{open-string state space} is defined as
\[ \CBa:= \bigoplus_{a\in H_1(L;\Z)}\underbrace{\prod_{k\in \Z_{\geq 0} } C^{\Omega_\star}(a,k)_*}_{=: \CBa(a)}.\]
\end{definition}

\begin{notation}\label{Ntn: DegConvert}
	To clarify the notation on grading, recall that (see Notation \ref{Ntn: degx}) for an element $C_\ell^\dR(X)$, we use the notation $\deg x = \ell$; more concretely given a de Rham chain $x=[(U,\varphi,\omega)]$, its degree $\deg x := \dim U - |\omega|$ (where $|\omega|$ is the degree of $\omega$ as a differential form). In contrast, for an (homogeneous) element $x\in C_\ell^\mathcal{X}$, we use the notation $|x|=\ell$. Thus, in particular,
	\begin{itemize}
		\item For $x\in \CFr$, the $(a,k)$-component $x(a,k)$ has \begin{equation}\label{Eq: DegConvertFree}\deg x(a,k) = |x|+\dim L +\mu(a) + k -1; \end{equation}
		\item For $\alpha\in \CBa$, the $(a,k)$-component $\alpha(a,k)$ has \begin{equation}\label{Eq: DegConvertBased}
				 \deg \alpha(a,k) = |\alpha| + \mu(a) + k. 
			\end{equation}
	\end{itemize}
\end{notation}
\begin{remark}
	The grading on $\CFr$ is defined to be consistent with that in \cite{Irie2}. The same grading-shift convention (with a Riemann-Roch formula term appearing to homogenize the degrees of the operations) appears also in e.g. \cite{SullivanSigma}, in the context of closed pseudo-holomorphic curves.
\end{remark}

The differentials on both of the state spaces $\CFr,\CBa$ are defined by \begin{equation}\label{Eq: differential}\partial = \partial^0 + \partial^1\end{equation} where
\begin{equation}\label{Eq: differentialExplicit}
	(\partial^0 x)(a,k):= \partial^\dR(x(a,k)),\quad (\partial^1 x)(a,k):= (-1)^{\dim L + |x|}\sum_{i=0}^k (-1)^i(\delta_i)_*(x(a,k-1)).
\end{equation}

It is clear that $\partial$ preserves the decompositions \[ \CFr = \bigoplus_{a\in H_1(L;\Z)} \CFr(a),\quad \CBa = \bigoplus_{a\in H_1(L;\Z)} \CBa(a). \]

\bigskip

To state the result in \cite{WangThesis}  that $\CFr$ and $\CBa$ indeed compute the correct homology groups,
we first need to define the differentiable space structures (and thus de Rham chain complexes) on the $C^\infty$ free and based loop spaces as in Example 4.2 (ii) of \cite{Irie1} and show that $C^\dR_*(\mathscr{L} L)$, $C_*^\dR(\Omega_\star L)$ compute the ordinary homologies of the spaces (in a way similar to section 6 of \cite{Irie1}).
We then define the chain complexes, for each $a\in H_1(L;\Z)$,
\begin{gather*}
	C_*^\Delta(\ms{L}(a)):= \left( \prod_{k\geq 0} C^\dR_{*+\dim L + \mu(a) + k - 1}(\ms{L}(a)\times \Delta^k), \partial\right),\\
	C_*^\Delta (\Omega_\star(a)):= \left( \prod_{k\geq 0} C^\dR_{*+\mu(a)+k}(\Omega_\star (a)\times \Delta^k),\partial \right),
\end{gather*}
coming from the cosimplicial structure given by the standard $k$-simplices \[\Delta^k:= \{ (t_1,\dots, t_k)\in \R_{\geq 0}^{k}\mid 0\leq t_1\leq \dots \leq t_k \leq 1\} .\]
For each $k\in \Z_{\geq 0}$, $a\in H_1(L;\Z)$, and $\mc{X}\in \{\ms{L},\Omega_\star\}$, define
\begin{align*}
	e_k\colon \mathcal{X}(a)\times \Delta^k &\to \mathcal{X}^{k+1} (a),\\ \quad  (\gamma, t_1,\dots, t_k)&\mapsto \big((\gamma(t_i), \gamma(t_{i+1}), [\gamma|_{[t_i,t_{i+1}]}] \big)_{0\leq i \leq k} \in \mathcal{X}^{k+1} (a) \subset (\Pi_1 L)^{k+1}.
\end{align*}
This induces a cosimplicial map for each $a\in H_1(L;\Z)$:
\[
	e_*\colon C_*^\Delta(\mathcal{X}(a))\to C_*^\mathcal{X}(a).
\]

On the other hand, we have the chain maps
\begin{gather*}	
	\pr_0 \colon (C_*^\Delta(\ms{L}(a)), \partial)\to (C_{*+\dim L + \mu(a) - 1}^\dR(\ms{L} (a)),\partial^\dR) ,\quad (x_k)_{k\geq 0} \mapsto x_0; \\
	\pr_0 \colon (C_*^\Delta(\Omega_\star(a)), \partial)\to (C_{*+ \mu(a)}^\dR(\Omega_\star (a)),\partial^\dR) ,\quad (x_k)_{k\geq 0} \mapsto x_0.
\end{gather*}

The following theorem then follows from Theorem 2.2.1 and Lemma 2.2.3 in \cite{WangThesis} and section 6 of \cite{Irie1} (Section 6 of \cite{Irie1} shows that de Rham chains on the $C^\infty$-free loop spaces compute ordinary homology, but the argument can be adapted to based loop spaces in exactly the same way):
\begin{theorem}\label{Thm: cosimpHlgy}
	Both of the maps $e_*$ and $\pr_0$ are quasi-isomorphisms, for either $\mathcal{X} = \mathscr{L}$ or $\Omega_\star$.
	Therefore 
	\begin{gather*}
		H_*(\CFr(a), \partial) \cong H_{*+\dim L + \mu(a) -1}(\ms{L}(a));\\
		H_*(\CBa(a), \partial) \cong H_{*+ \mu(a)}(\Omega_\star(a)).
	\end{gather*}
\end{theorem}
Technically the $(\partial^1 x)$ term  in \eqref{Eq: differentialExplicit} may have the opposite sign as that in \cite{WangThesis} depending on $\dim L$; however in dimensions where this happen, the statement still holds because the chain map \[\prod_{k\in \Z_{\geq 0}} C_*^\mc{X}(a,k)_*\xrightarrow{(-1)^k}\prod_{k\in \Z_{\geq 0}} C_*^\mc{X}(a,k)_*\] is an isomorphism of chain complexes.

\subsection{String topology operations on de Rham chains}\label{Subsc: StringOper}
In this section we construct various string topology operations on de Rham chains on the components of the cosimplicial chain complexes $C_*^\dR(\ms{L}^{k+1}(a))$ and $C_*^\dR(\Omega_\star^{k+1}(a))$.
\subsubsection{Closed-string}
Define the (closed-string) concatenation map, for $k_1,k_2\in \Z_{\geq 0}$, $a_1,a_2\in H_1(L;\Z)$, and $i=1,\dots, k_1$, 
\[
	\conc_{k_1,i,k_2}\colon  \mathscr{L}^{k_1+1}(a_1) \fiberprod{\ev_i}{\ev_0} \mathscr{L}^{k_2+1}(a_2) \to \mathscr{L}^{k_1+k_2}(a_1+a_2)
\]
by 
	\begin{align*}
		&\conc_{k_1,i,k_2}\big((c_0,\dots, c_{k_1}), (c_0',\dots, c_{k_2}')\big) :=\\
		& \begin{cases}
			(c_0,\dots, c_{i-2}, c_{i-1}*c_0', c_1',\dots, c'_{k_2-1},c'_{k_2}* c_i, c_{i+1},\dots, c_{k_1}), & k_2\geq 1\\
			(c_0, \dots, c_{i-2}, c_{i-1}*c_0'*c_i,c_{i+1},\dots, c_{k_1}), & k_2=0
		\end{cases}.
	\end{align*}  
	
	Roughly speaking, $\conc_{k_1,i,k_2}$ looks for where the $i$-th marked point of the first loop coincides with the $0$-th marked point of the second loop, and concatenate the two at the coinciding point (compare with the construction of the loop product).
	Henceforth we will often write $\conc$ instead of $\conc_{k_1,i,k_2}$ whenever the context is clear.

Let $x\in C_{*}^\dR( \mathscr{L}^{k_1+1}(a_1))$ and $y\in C_{*}^\dR( \mathscr{L}^{k_2+1}(a_2))$.
For each $i=1,\dots, k_1$, we define a chain $x\circ_i^\mathscr{L} y\in  C^\dR_{\deg x+ \deg y+\dim L}(\mathscr{L}^{k_1+k_2}(a_1+a_2))$ as (up to signs) the composition of the chain-level fibre product and the chain-level pushforward along $\conc_{k_1,i,k_2}$. Explicitly, if $x=[(U_1,\varphi_1,\omega_1)]$, $y=[(U_2,\varphi_2, \omega_2)]$, then, using the abbreviation \[U_1\fiberprod{i}{0} U_2:= U_1\fiberprod{\ev_i\circ \varphi_1}{\ev_0\circ \varphi_2} U_2\] (this is a transverse fibre product under our assumption that $\ev_0\circ \varphi_2$ is a submersion; see Definition \ref{Defn: DiffSpaceOnLoops}), define
\begin{equation} \label{Eq: PartialLoopPre}
x\circ_i^\mathscr{L} y:= (-1)^{(\dim U_1-|\omega_1| -\dim L)|\omega_2|} [(U_1\fiberprod{i}{0} U_2, \varphi_1\circ_i^\mathscr{L} \varphi_2, \omega_1\times \omega_2|_{U_1{}_i\times_0 U_2})].\end{equation}
where $\varphi_1\circ_i^\mathscr{L} \varphi_2$ denotes the composition
\[
	U_1\fiberprod{i}{0} U_2 \xrightarrow{\varphi_1\times \varphi_2}\mathscr{L}^{k_1+1}(a_1)\fiberprod{\ev_i}{\ev_0} \mathscr{L}^{k_2+1}(a_2) \xrightarrow{\conc_{k_1,i,k_2}} \mathscr{L}^{k_1+k_2}(a_1+a_2).
\]
We shall frequently abuse notation by making the restriction of $\omega_1\times \omega_2$ to the correct domain $U_1\fiberprod{i}{0} U_2$ implicit.

Hence this induces a map 
\[
	\circ_i^\mathscr{L}\colon C_{i_1}^\dR(\mathscr{L}^{k_1+1}(a_1))\otimes C^\dR_{i_2}(\mathscr{L}^{k_2+1}(a_2))\to C^\dR_{i_1+i_2-\dim L}(\mathscr{L}^{k_1+k_2}(a_1+a_2))
\]
for $k_1,k_2\in \Z_{\geq 0}$, $a_1,a_2\in H_1(L;\Z)$, and $i=1,\dots, k_1$.

See Appendix \ref{Subsubsc: SignClosedString} for properties of $\circ_i^\ms{L}$.

\subsubsection{Open-string}
Define the (open-string) concatenation map, for $k_1,k_2\geq 0$ and $a_1,a_2\in H_1(L;\Z)$,
\[
	*\colon \Omega^{k_1+1}_\star (a_1)\times \Omega^{k_2+1}_\star (a_2)\to \Omega^{k_1+k_2+1}_\star (a_1+a_2)
\]
by
\[
	(c_0,\dots, c_{k_1}) * (c_0',\dots, c_{k_2}')\mapsto (c_0,\dots, c_{k_1-1}, c_{k_1}*c_{0}', c_1', \dots, c_{k_2}').
\]
This induces a map of de Rham chain complexes
\[
	\bullet\colon  C_{i_1}^\dR(\Omega^{k_1+1}_\star(a_1))\otimes C_{i_2}^\dR(\Omega_\star^{k_2+1}(a_2))\to C_{i_1+i_2}^\dR(\Omega_\star^{k_1+k_2+1}(a_1+a_2)).
\]
We spell this out more explicitly to fix the sign: for $\alpha\in C_{i_1}^\dR(\Omega^{k_1+1}_\star(a_1))$ given by $\alpha=[(V_1,\psi_1,\eta_1)]$, $\beta\in C^\dR_{i_2}(\Omega_\star^{k_2+1}(a_2))$ given by $\beta = [(V_2,\psi_2,\eta_2)]$, define
\begin{align}\label{Eq: PartialPontryagin}
	\alpha \bullet\beta := (-1)^{(\dim V_1- |\eta_1|)|\eta_2|}[(V_1\times V_2, \psi_1\bullet \psi_2, \eta_1\times \eta_2)]
\end{align}
where $\psi_1\bullet\psi_2$ is the abbreviation for the composition
\[
	V_1\times V_2 \xrightarrow{\psi_1\times \psi_2} \Omega^{k_1+1}_\star(a_1)\times \Omega_\star^{k_2+1}(a_2)\xrightarrow{*} \Omega_\star^{k_1+k_2+1}(a_1+a_2).
\]

See Appendix \ref{Subsubsc: SignOpenString} for properties of $\bullet$.

\subsubsection{Open-closed string}
Finally, define the (open-closed) concatenation map, for $k_1,k_2\in \Z_{\geq 0}$, $a_1,a_2\in H_1(L;\Z)$, and $i=1,\dots, k_1$, 
\[\concOmega_{k_1,i,k_2}\colon \Omega^{k_1+1}_\star (a_1) \fiberprod{\ev_i}{\ev_0}\mathscr{L}^{k_2+1}(a_2)\to \Omega^{k_1+k_2}_\star (a_1+a_2)
\] by exactly the same formula as before:
	\begin{align*}
		&\concOmega_{k_1,i,k_2}\big((c_0,\dots, c_{k_1}), (c_0',\dots, c_{k_2}')\big) :=\\
		& \begin{cases}
			(c_0,\dots, c_{i-2}, c_{i-1}*c_0', c_1',\dots, c'_{k_2-1},c'_{k_2}* c_i, c_{i+1},\dots, c_{k_1}), & k_2\geq 1\\
			(c_0, \dots, c_{i-2}, c_{i-1}*c_0'*c_i,c_{i+1},\dots, c_{k_1}), & k_2=0
		\end{cases}.
	\end{align*}
The same fibre product procedure therefore yields this operation on the chain level. 
Explicitly, if $\alpha=[(V,\psi, \eta)]\in C_*^\dR(\Omega_\star^{k_1+1}(a_1))$ and  $x=[(U,\varphi, \omega)] \in C_*^\dR(\mathscr{L}^{k_2+1}(a_2))$, define 
\begin{align}\label{Eq: PartialOCprod}
	\alpha\circ_i^{\Omega_\star} x:= (-1)^{(\dim V-|\eta|)|\omega|}[( V\fiberprod{i}{0}U, \psi \circ_i^{\Omega_\star} \varphi, \eta\times \omega )]
\end{align}
where $\psi\circ_i^{\Omega_\star} \varphi$ is short for
\[
	V\fiberprod{i}{0} U \xrightarrow{\psi\times \varphi}\Omega_\star^{k_1+1}(a_1)\fiberprod{\ev_i}{\ev_0}\mathscr{L}^{k_2+1}(a_2)\xrightarrow{\concOmega_{k_1,i,k_2}}\Omega_\star^{k_1+k_2}(a_1+a_2),
\]
and we follow the previous abbreviation of \[V\fiberprod{i}{0} U:= V\fiberprod{\ev_i\circ \psi}{\ev_0\circ \varphi} U,\] as well as the abuses of notation of making the restriction of $\eta\times \omega$ to the domain $V\fiberprod{i}{0}U$ implicit.

This induces the chain-level open-closed string topology operation 
\[
	\circ_i^{\Omega_\star} \colon C_{i_1}^\dR(\Omega_\star^{k_1+1}(a_1))\otimes C_{i_2}^\dR(\mathscr{L}^{k_2+1}(a_2))\to C_{i_1+i_2-\dim L}^\dR(\Omega_\star^{k_1+k_2}(a_1+a_2))
\]
for $k_1,k_2\in \Z_{\geq 0}$, $a_1,a_2\in H_1(L;\Z)$, and $i=1,\dots, k_1$.

See Appendix \ref{Subsubsc: OCSignAppendix} for properties of $\circ_i^{\Omega_\star}$.

\subsubsection{The anomaly map}
Define, for each $a\in H_1(L;\Z), k\in \Z_{\geq 0}$,
\[
	\mathfrak{o}\colon C_{*+\dim L}^\dR(\mathscr{L}^{k+1}(a))\to C_{*}^\dR(\Omega_{\star}^{k+1}(a))
\]as follows (here the notation $\mathfrak{o}$ is meant for ``obstruction'', similar to that in section 3.6.2 of \cite{FOOO1}).
Given a de Rham chain $x:=[(U\xrightarrow{\varphi} \mathscr{L}^{k+1}(a),\omega)]\in C_{*}^\dR(\mathscr{L}^{k+1}(a))$, 
define \begin{align}\label{Eq: AnomalyMap}
	\mathfrak{o}(x):=(-1)^{(\deg x)+1}[\big((\ev_0\circ\varphi)^{-1}(\star)\xrightarrow{\varphi}\Omega_\star^{k+1}(a), \omega\big)] \in C_{*}^\dR(\Omega_\star^{k+1}(a)).
\end{align}
Here, since we have assumed $\ev_0\circ \varphi\colon U\to L$ is a submersion (see Definition \ref{Defn: DiffSpaceOnLoops}), $(\ev_0\circ \varphi)^{-1}(\star)$ is transversely cut out. 
Also, $\varphi$ maps points in $(\ev_0\circ\varphi)^{-1}(\star)\subset U$ to loops in $\mathscr{L}^{k+1} (a)$ where the starting point is $\star$. Thus the image under $\varphi$ of $(\ev_0\circ \varphi)^{-1}(\star)$ does land in $\Omega^{k+1}_\star(a)$ as claimed.

See Appendix \ref{Subsubsc: SignAnomaly} for properties of $\mf{o}$.

\subsection{Resulting structures on the state spaces}\label{Subsc: StringOperState}
We now describe the algebraic structures on the state spaces $C_*^\mathscr{L}$ and $C_*^{\Omega_\star}$ by putting together the operations defined on components of the cosimplicial chain complexes in the previous section.

\subsubsection{The closed- and open-string dg algebras}

\begin{definition}
We define the following operations on $\CFr$ and $\CBa$ respectively:

\textbf{{(Loop bracket)}}
For $x,y\in C_*^\mathscr{L}$, define the pre-Lie product
	\[
		\circ\colon C_*^\mathscr{L}\otimes C_*^\mathscr{L}\to C_*^\mathscr{L}
	\] given by \begin{equation}\label{Eq: LoopPreBracket}
	(x\circ y)(a,k):= \sum_{\substack{k'+k''=k+1\\1\leq i\leq k'\\ a' + a''=a}}(-1)^{(i-1)(k''-1)+(k'-1)(|y|+1+k'')} x(a',k')\circ_i y(a'',k'').
\end{equation}
and the \emph{loop bracket} 
\[
	[\,,\,]\colon C_*^\mathscr{L}\otimes C_*^\mathscr{L}\to \CFr
\]
given by
\begin{equation}\label{Eq: LoopBracket}
	[x,y]:= x\circ y -(-1)^{|x||y|}y\circ x.
\end{equation} 

\textbf{{(Pontryagin product)}}
 For  $\alpha,\beta \in \CBa$, define the \emph{Pontryagin product} 
\[
	\bullet\colon \CBa \otimes \CBa \to \CBa
\]
given by
\begin{equation}\label{Eq: PontryaginProd}
	(\alpha \bullet \beta)(a,k) := \sum_{\substack{k'+k''=k\\a'+a''=a}} (-1)^{k'|\beta|}\alpha(a',k')\bullet \beta(a'',k'')
\end{equation}
\end{definition}

\begin{remark}
	The signs in the formulas here do not involve the Maslov classes, since in our case $L$ is orientable and $\mu_L\in 2\Z$, which do not contribute to the signs.
\end{remark}

\begin{definition}\label{Defn: CBaUnit}
	We define the element \[ [\underline{\star}] \in \Cba{0}(0) = \prod_{k\in \Z_{\geq 0}}C^\dR_k(\Omega_\star^{k+1}(0))\]as follows:
  \begin{itemize}
  	\item For $k=0$, consider the map $\pt\to \Omega^1_\star(0)$ where the single point is mapped to the constant based loop $[\underline{\star}]\in \Omega^1_\star(0)\subset \Pi_1 L$. The de Rham chain $[(\pt\to \Omega^1(0); 1\in \mathscr{A}^0_c(\pt))]$ defines a closed cycle in $C^\dR_0(\Omega^1(0))$, which we set to be the $k=0$ component of $[\underline{\star}]$;
  	\item For all $k>0$, set the $k$-the component of $[\underline{\star}]$ to be $0$.
  \end{itemize}
\end{definition}

\begin{lemma}
	[Structures on $C_*^\mathscr{L}$ and $C_*^\Omega$]\label{Lm: StructuresOnStates}
	\begin{enumerate}
		\item $(\CFr, \partial, [,])$ is a dg Lie algebra;
		\item $(\CBa, \partial, \bullet)$ is a dg associative algebra with strict unit $[\underline{\star}]$; 
		\item All the structures are compatible with the decomposition into $a\in H_1(L;\Z)$: more specifically, for $\mathcal{X},\mathcal{X}'$ denoting either $\mathscr{L}$ or $\Omega_\star$, and for any $a,b\in H_1(L;\Z)$, we  have $\partial\colon C_*^\mathcal{X}(a)\to C_{*-1}^\mathcal{X}(a)$, and all the binary operations have $C_*^\mathcal{X}(a)\otimes C_*^{\mathcal{X}'}(b)\to C_*^{\mathcal{X}'}(a+b)$.
	\end{enumerate}
\end{lemma}

See Appendix \ref{Subsc: ConventionDGalg} for conventions regarding signs of dg algebras, and Appendix \ref{Subsc: dgAlgebraSignApp} the proof of this Lemma. Specifically, see Lemma \ref{Lm: StructuresOnClosedStatesApp} for  (1), and Lemma  \ref{Lm: StructuresOnOpenStatesApp} for (2); the statement in (3) regarding compatibility with the decomposition into $H_1(L;\Z)$ is clear from construction.

\begin{remark}\label{Rmk: RichStructure}
	We remark that there should be rich algebraic structures on $\CFr$ (e.g. a version of Deligne's conjecture in \cite{Irie1}),  but the only part we will use is the Lie bracket.
\end{remark}

\begin{proposition}
The isomorphism \begin{align*}
		\bigoplus_{a\in H_1(L;\Z)} \HBa(a) \cong \bigoplus_{a\in H_1(L;\Z)} H_{*+\mu(a)}(\Omega_\star(a))
	\end{align*}
	in Theorem \ref{Thm: cosimpHlgy} is an isomorphism of graded algebras.
\end{proposition}

\begin{proof}
	This is in
	Section 2.5 (Proposition 2.5.1) in \cite{WangThesis}.
	Similar to the remark after Theorem \ref{Thm: cosimpHlgy}, in cases where the $(\partial^1 x)$ term  in \eqref{Eq: differentialExplicit} has the opposite sign as that in \cite{WangThesis}, the chain map \[\bigoplus_{a\in H_1(L;\Z)}\prod_{k\in \Z_{\geq 0}} \CBa(a,k)_*\xrightarrow{(-1)^k}\bigoplus_{a\in H_1(L;\Z)}\prod_{k\in \Z_{\geq 0}} \CBa(a,k)_*\] is an isomorphism of dg associative algebras so the statement still holds.
\end{proof}

These structures are parts (1) and (2) of Theorem \ref{Thm: OCstringPackage}.

\subsubsection{A model for Hochschild cochains}\label{Subsubsc: Hochschild}
We now start to describe the last component of Theorem \ref{Thm: OCstringPackage}, i.e. the construction of the dg Lie algebra homomorphism 
\[
	\CO\colon C_*^\mathscr{L} \to \CC^*(\CBa,\CBa).
\]
In this section we first construct a model of $\CC^*(\CBa,\CBa)$.

We have shown that $(\CBa,\partial, \bullet)$ is a dg associative algebra with unit the constant loop class $[\underline{\star}] \in \Cba{0}$. We now construct a chain model of Hochschild cohomology that is similar to the construction of $C_*^\mathscr{L}$ (in particular, the main difference with the usual construction of Hochschild cochains is that we need to make explicit the decomposition into $a\in H_1(L;\Z)$). 

\begin{remark}\label{Rmk: RichStructure2}
Similar to the situation of $\CFr$ (see Remark \ref{Rmk: RichStructure}), we remark that there are rich algebraic structures on Hochschild cochains of a dg associative algebra (e.g. Gerstenhaber structure on the cohomology in \cite{Gerstenhaber}, and many later work on Deligne's conjecture of a $E_2$-algebra structure), but the only part we will use is the dg Lie algebra on Hochschild cochains.
\end{remark}

For degree and sign conventions, we roughly follow section 2.2 of \cite{Irie1}.

\begin{definition}\label{Defn: HochschildCochains}
	For any $a\in H_1(L;\Z)$ and $\ell\in \Z_{\geq 0}$, let 
	\begin{align*}
			\CC^*(&\CBa,\CBa)(a,\ell)\\&:= \prod_{\substack{a_1,\dots, a_\ell\\ \,\,\,\,\,\,\,\,\,\,\,\, \in H_1(L;\Z)}}\Hom_{*+\ell - 1} \left( \CBa(a_1)\otimes \dots \otimes \CBa(a_\ell), \CBa\left(a_1+\dots + a_\ell+a\right)\right).
	\end{align*}
	
	Then define \[ \CC^*(\CBa,\CBa):= \bigoplus_{a\in H_1(L;\Z)} \underbrace{\prod_{\ell\in \Z_{\geq 0}} \CC^*(\CBa,\CBa)(a,\ell)}_{=:\CC^*(\CBa,\CBa)(a)}. \]	
\end{definition}

\begin{remark}
	Here, the degree is shifted by 1 from the usual Hochschild cochain complex degrees (in e.g. Definition 2.27 of \cite{GanatraThesis}), to keep consistency with that of $\CFr$ in \cite{Irie1,Irie2}. In particular, similar to the situation in $\CFr$, the advantage for this degree shift is that the Gerstenhaber bracket, which is the structure we shall use, has degree 0 instead of $+1$.
\end{remark}

\begin{notation}\label{Ntn: l-ary}
Later on, for an element $\Phi\in \CC^*(\CBa,\CBa)$, we will refer to its component in $\CC^*(\CBa,\CBa)(a)$ as $\Phi(a)$, and the component of $\Phi(a)$ in $\Hom((\CBa)^{\otimes \ell}, \CBa)$ as the \emph{$\ell$-ary part} of the element $\Phi$, and write it as $\Phi_\ell(a)\in \Hom((\CBa)^{\otimes \ell}, \CBa)$. We will also denote by $\Phi_\ell \in \Hom((\CBa)^{\otimes \ell},\CBa)$ by the image of $\Phi$ under $\CC^*(\CBa,\CBa)\to \bigoplus_{a\in H_1(L;\Z)} \CC^*(\CBa,\CBa)(a,\ell) \to \Hom((\CBa)^{\otimes \ell},\CBa)$.
\end{notation}
\begin{notation}
	To clarify the notation on grading (similar to Notation \ref{Ntn: DegConvert}), for a homomorphism $\Phi \in \Hom_k(V^{\otimes \ell},W)$ where $V$ and $W$ are graded vector spaces and $\ell\in \Z_{\geq 0}$, we use the notation $\deg \Phi = k$. More precisely, given $\alpha_1,\dots, \alpha_\ell \in V$, we have \[ \deg \Phi(\alpha_1,\dots, \alpha_\ell) = \sum_{i=1}^\ell  |\alpha_i|  + \deg \Phi.\]
	In contrast, for an (homogeneous) element $\Phi \in \CC^k(\CBa,\CBa)$, we write $|\Phi| = k$. Thus, in particular, for $\Phi\in \CC^*(\CBa,\CBa)$, the $(a,\ell)$-component $\Phi_\ell(a)$ has \begin{align}\label{Eq: DegConvertHoch} \deg\Phi_\ell(a) = |\Phi| + \ell -1. \end{align}
\end{notation}

\begin{definition}
We define the following operations on the Hochschild cochain complex: 

\textbf{{(Hochschild differential)}} $\CC^*(\CBa,\CBa)$ has a differential of degree $-1$ 
	\[
		\delta\colon \CC^*(\CBa,\CBa)\to \CC^{*-1}(\CBa,\CBa),
		\]  where given an element $\Phi:= (\Phi_\ell)_{\ell\geq 0}\in \CC^*(\CBa,\CBa)$, its differential $\big((\delta\Phi)_\ell\big)_{\ell\geq 0}$ is given by  the two components \[ (\delta \Phi)_\ell:= \delta^0 \Phi_\ell + \delta^1\Phi_{\ell -1}, \] where 
		\begin{itemize}
			\item The component $\delta^0\Phi_\ell$ is induced by the differential on $\CBa$, i.e. given $\alpha_1,\dots, \alpha_\ell\in \CBa$,  \begin{align*}
				(\delta^0\Phi_\ell)&(\alpha_1,\dots, \alpha_\ell) \\&= \partial(\Phi_\ell(\alpha_1,\dots, \alpha_\ell)) - (-1)^{\deg \Phi_\ell} \sum_{i=1}^\ell (-1)^{\sum_{j=1}^{i-1} |\alpha_j|}\Phi_\ell(\alpha_1,\cdots, \partial \alpha_i, \cdots, \alpha_\ell). 
			\end{align*} 	
			\item The component $\delta^1\Phi_{\ell-1}$ is given as follows. Define, for each $\ell\geq 1$ and $i=0,\dots, \ell$, a chain $\delta^1_{\ell,i}\colon \Hom_*((\CBa)^{\otimes (\ell-1)}, \CBa)\to \Hom_*((\CBa)^{\otimes \ell}, \CBa)$, by \[ (\delta^1_{\ell,i}\Phi_{\ell-1})(\alpha_1,\dots, \alpha_\ell) := \begin{cases}
				(-1)^{|\alpha_1|(\deg \Phi_{\ell-1})}\alpha_1\bullet \Phi_{\ell-1}(\alpha_2,\dots, \alpha_\ell) , & i = 0\\
				\Phi_{\ell-1}(\alpha_1,\dots, \alpha_i\bullet \alpha_{i+1}, \dots, \alpha_\ell) , & 1\leq i \leq \ell-1 \\
				\Phi_{\ell-1}(\alpha_1,\dots, \alpha_{\ell-1})\bullet \alpha_\ell, & i = \ell
			\end{cases}. \]
			Set \[ \delta^1(\Phi_{\ell-1}):= (-1)^{\deg \Phi_{\ell-1}+ \ell - 1} \sum_{i=0}^\ell (-1)^i \delta^1_{\ell,i}\Phi_{\ell-1}.\]
			\end{itemize}

\begin{remark}
	The two components of the differential, $\delta = \delta^0+\delta^1$ are  analogous to the two components $\partial = \partial^\dR + \delta$ in $\CFr$ and $\CBa$.
\end{remark}

\textbf{{(Gerstenhaber bracket)}} $\CC^*(\CBa,\CBa)$ has a Lie bracket of degree 0 \[
 [-,-]\colon \CC^*(\CBa,\CBa)\otimes \CC^*(\CBa,\CBa)\to \CC^*(\CBa,\CBa)
\] 
defined by \[ [\Phi,\Psi]:= \Phi \circ \Psi - (-1)^{|\Phi|\cdot |\Psi|} \Psi \circ \Phi, \]

where given two elements $\Phi:=(\Phi_\ell)_{\ell\geq 0}, \Psi:=(\Psi_\ell)_{\ell\geq 0} \in \CC^*(\CBa,\CBa)$, we define the pre-Lie product $\circ$ as follows: given $\alpha_1,\dots, \alpha_{\ell}\in \CBa$,
\begin{align*}
	(\Phi \circ \Psi)_\ell &(\alpha_1,\dots, \alpha_\ell) \\
	:= & \sum_{\substack{\ell' + \ell''=\ell+1\\ 1 \leq i \leq \ell'}} (-1)^\dagger \Phi_{\ell'}(\alpha_1,\dots, \alpha_{i-1}, \Psi_{\ell''}(\alpha_i,\dots, \alpha_{i+\ell''-1}),\alpha_{i+\ell''},\dots, \alpha_\ell),
\end{align*}
where the sign is given by \[ \dagger:= (i-1)(\ell''-1)+(|\Psi|+\ell''+1)(|\alpha_1|+\dots + |\alpha_{i-1}| + \ell' - 1). \]
\end{definition}

\begin{lemma}
	[Structures on $\CC^*$]\label{Lm: StructuresOnHochschild}
	\begin{enumerate}
		\item $(\CC^*(\CBa,\CBa), \delta, [-,-])$ is a dg Lie algebra;
		\item All the structures are compatible with the decomposition into $H_1(L;\Z)$: more precisely,  for any $a,b\in H_1(L;\Z)$, we  have \[\delta\colon \CC^*(\CBa,\CBa)(a)\to \CC^{*-1}(\CBa,\CBa)(a),\]and \[[-,-]\colon \CC^*(\CBa,\CBa)(a)\otimes \CC^*(\CBa,\CBa)(b) \to \CC^*(\CBa,\CBa)(a+b).\]
	\end{enumerate}
\end{lemma}
The proof of (1) is classical (see e.g. \cite{Irie1} Example 2.7 and Theorem 2.8); for conventions on degrees and signs of dg Lie algebras, see definition \ref{Defn: dgLa} in the appendix. Part (2) is straightforward.

\subsubsection{The closed-open string map}
We now define the closed-open map \[\CO\colon \CFr\to \CC^*(\CBa,\CBa).\]
We shall define, for each $a\in H_1(L;\Z)$ and $\ell \in \Z_{\geq 0}$,
\begin{align*}
	\CO_{\ell,a}\colon \CFr(a)\to \CC^*(\CBa,\CBa)(a,\ell),
\end{align*}
and accordingly, for each (fixed) $\ell$, the composition \[\CO_\ell\colon \CFr\xrightarrow{\bigoplus_a \CO_{\ell,a}} \bigoplus_a \CC^*(\CBa,\CBa)(a,\ell) \to  \Hom\big(\left(\CBa\right)^{\otimes \ell},\CBa\big),\] 
so that for $x\in \CFr$, under Notation \ref{Ntn: l-ary},
\begin{align*}
	\big(\CO(x)\big)_\ell(a):= \CO_{\ell, a}(x), \quad \big(\CO(x)\big)_\ell := \CO_\ell(x).
\end{align*}

\begin{definition}
We define the closed-open map map $\CO$ in arity/components as follows:

\textbf{(0-ary part)}
For each $a\in H_1(L;\Z), k\in \Z_{\geq 0}$, 
define the \emph{anomaly map} $\mathfrak{o}\colon \CFr\to \Cba{*-1}$ by
\begin{align}\label{Eq: AnomalyMapState}
	 \big(\mathfrak{o}(x)\big)(a,k):=  \mathfrak{o}\big(x(a,k)\big).
\end{align}
Then define
\begin{align}\label{Eq: CO0 OpenClosed}
\CO_{0,a}(x) := {(-1)^{|x|}}\bigg(\big(\mathfrak{o}(x)\big)(a,k)\bigg)_{k\in \Z_{\geq 0}} \in \Cba{|x|-1} .\end{align}

\textbf{(Unary part)}
Given $\alpha\in \CBa$ and $x\in \CFr$, define
\begin{align}\label{Eq: CO1 openclosed}\big\{\CO_{1}(x)\big\}(\alpha)=:{(-1)^{|\alpha||x|+1}}\alpha\circ x  \in \CBa,\end{align}
where for each $a\in H_1(L;\Z)$ and $k\in\Z_{\geq 0}$, we define
\begin{align}\label{Eq: CO1 Expand}(\alpha\circ x)(a, k):= \sum_{\substack{k'+k'' = k+1\\ 1\leq i \leq k'\\a'+ a'' = a}} {(-1)^{(i-1)(k''-1)+k'(|x|+1+k'')}} \alpha(a',k')\circ_i x(a'',k''). \end{align}

Then for each $a,a'\in H_1(L;\Z)$, the map $\CO_1$ restricts to
\[
	\CO_{1,a}(x)\colon \CBa(a')\to \Cba{*+|x|}(a+a')
\]
because in formula \eqref{Eq: CO1 Expand}, for $\alpha(a',k') \in \CBa(a')$ and $x(a'',k'')\in \CFr(a'')$, the resulting $\alpha(a',k')\circ_i x(a'',k'')\in \CBa(a'+a'')$.

\textbf{(Higher arity parts)} All higher arity parts are set to 0. That is, for each $\ell\geq 2$, we set $\CO_\ell\colon \CFr\to \Hom((\CBa)^{\otimes \ell}, \CBa)$ to be 0.

\bigskip
Then these operations $\{\CO_{k,a}\}_{\substack{k\in \Z_{\geq 0}, a\in H_1(L;\Z)}}$ are packaged into 
\[
	\CO \colon \CFr\to \CC^* (\CBa,\CBa).
\]
This is a degree-0 map.
\end{definition}

We claim that this is a homomorphism of dg Lie algebra:
\begin{lemma}[Lemma \ref{Lm: dgLieMorphismApp}]\label{Lm: dgLieMorphism}
	$\CO$ is a homomorphism of dg Lie algebra. Moreover, $\CO$ respects the decompositions of $\CFr$ and $\CC^*(\CBa,\CBa)$ into $a\in H_1(L;\Z)$; that is, if $x\in \CFr(a)$ for some $a\in H_1(L;\Z)$, then $\CO(x)\in \CC^*(\CBa,\CBa)(a)$.
\end{lemma}
We verify this in Lemma \ref{Lm: dgLieMorphismApp} in the appendix.

This finishes our construction of the open-closed string topology package, i.e. Theorem \ref{Thm: OCstringPackage}.

\bigskip

\section{Holomorphic curves as deformation}
\label{Sec: HolomorphicCurves}

We now use moduli spaces of pseudo-holomorphic discs with boundary on the Lagrangian $L\subset \C^n$ to produce a curved $A_\infty$-deformation of the open string dg algebra $\CBa$ of $L$.
As mentioned in the introduction, the basic idea comes from Fukaya (\cite{FukayaLagrangian}) in the context of free loop spaces (which is realized by Irie in \cite{Irie1,Irie2}; for our purpose we work in the framework of \cite{WangThesis}) and from the proposal of Abouzaid (\cite{AbouzaidLect}) in the context of based loop spaces.

The main result of this section is Theorem \ref{Cor: BasedAinfty}, which is a rigorous version of the idealized Theorem \ref{IdealThm} in the introduction, constructing a gapped (curved) associative algebra deforming the Pontryagin algebra structure on the open string state space $\CBa$. 
This follows from the construction of the closed-open map in section \ref{Sec: CO map}, as well as the construction of a Maurer-Cartan element in the closed-string state space $\CFr$ coming from moduli spaces of pseudo-holomorphic discs (Theorem \ref{Thm: VFCmain}). In section \ref{Subsc: FiltrationGapped} we discuss the energy filtration and completion of the state space, as well as gappedness of Maurer-Cartan elements and $A_\infty$-structures,  to deal with various convergence issues.

\subsection{Preliminaries on energy filtration and gappedness}\label{Subsc: FiltrationGapped}

To deal with convergence of operations coming from pseudo-holomorphic discs, we need to consider the energy filtration on $\CFr$ and $\CBa$, introduced below. The issue of convergence is dealt with often by introducing the Novikov  field or ring (e.g. in ordinary Lagrangian Floer theory, see \cite{FOOO1,FOOO2}). In our setting, it is more convenient to avoid introducing the Novikov field and simply work with the completion of the energy filtration on $\CFr$ and $\CBa$, as is done in \cite{Irie2}.

\subsubsection{Energy filtration and completion}
Suppose $L\subset \C^n$ is a closed Lagrangian. Recall that we have the energy homomorphism $E\colon H_1(L;\Z)\to \R$, given by e.g. integrating the Liouville 1-form (the primitive to the standard symplectic form).

\begin{definition}
Let $\textbf{C}$ denote one of the three chain complexes: $\CFr,\CBa$ or $\CC^*(\CBa,\CBa)$. In each case we have, by construction, a decomposition \[ \textbf{C}= \bigoplus_{a\in H_1(L;\Z)} \textbf{C}(a). \] 

\begin{enumerate}
	\item 	The \emph{energy filtration} on $\textbf{C}$ is given by $\{\mathscr{F}^\lambda \textbf{C}\}_{\lambda\in \R}$ where for each $\lambda \in \R$,
	\begin{align*}
		\ms{F}^\lambda \textbf{C}:= \bigoplus_{E(a)>\lambda} \textbf{C}(a).
	\end{align*}
	\item 	The \emph{completion} of $\textbf{C}$ with respect to the energy filtration is denoted
\[
	\wh{\textbf{C}}:= \varprojlim_{\lambda \to+\infty} \textbf{C}\big /\ms{F}^\lambda \textbf{C}.
\]
\end{enumerate}
\end{definition}

An element $x\in \wh{\textbf{C}}$ can be identified as an (possibly) infinite sum \[ x = \sum_{i=1}^\infty x(a_i), \quad  \textup{ where }a_i \in H_1(L;\Z), \,\, x(a_i) \in \textbf{C}(a_i), \textup{ and } E(a_i)\to +\infty.\]

Since the differential in each of the three cases ($\textbf{C}= \CFr,\CBa,$ or $\CC^*(\CBa,\CBa)$) preserves the splitting of $\textbf{C}$ into $H_1(L;\Z)$ classes (Lemma \ref{Lm: StructuresOnStates} (3) and Lemma \ref{Lm: StructuresOnHochschild} (2)), it preserves the energy filtration by definition, and thus we obtain, on the homology $\textbf{H}$ of $\textbf{C}$, a filtration $\{\ms{F}^\lambda \textbf{H}\}_{\lambda \in \R}$ . More precisely, $\textbf{H}$ splits into
\begin{align*}
	\textbf{H} = \bigoplus_{a \in H_1(L;\Z)} \textbf{H}(a),
\end{align*}
and the energy filtration on $\textbf{H}$ is given by, for each $\lambda \in \R$,
\begin{align*}
	\ms{F}^\lambda \textbf{H} = \bigoplus_{E(a)>E} \textbf{H}(a).
\end{align*}
We also define its completion to be 
\begin{align*}
	\wh{\textbf{H}}:=  \varprojlim_{E\to\infty} \textbf{H}/\ms{F}^\lambda \textbf{H}.
\end{align*}

By Lemma \ref{Lm: StructuresOnStates} (3), all the string topology operations defined in the previous section respect the splitting into $a\in H_1(L;\Z)$, and therefore extends to the completions $\wh{\CFr}$ and $\wh{\CBa}$.
Similarly by Lemma \ref{Lm: StructuresOnHochschild} (2), the Gerstenhaber bracket extends to the completion $\wh{\CC^*}(\CBa,\CBa)$.
In particular, $\wh{\CFr}$ and $\wh{\CC^*} (\CBa,\CBa)$ are dg Lie algebras and $\wh{\CBa}$ is a dg associative algebra.
Moreover by Lemma \ref{Lm: dgLieMorphism}, the closed-open map $\CO\colon \CFr\to \CC^*(\CBa,\CBa)$ preserves the splitting into $a\in H_1(L;\Z)$, so $\CO$ extends to $\CO\colon \wh{\CFr}\to \wh{\CC^*}(\CBa,\CBa)$.

\subsubsection{Gappedness}
One essential point for many of our arguments is that the energy levels of pseudo-holomorphic discs as well as the perturbed pseudo-holomorphic discs are distributed according to Gromov compactness theorem. Following \cite{FOOO1} (see Condition 3.1.6 and Definition 3.2.26; also see \cite{Irie2,YuanHangThesis}), we say
\begin{definition}\label{Defn: monoid}
A subset $\mathfrak{G}\subset H_1(L;\Z)$ is \emph{a monoid of curve classes} if 
\begin{enumerate}
	\item $\mathfrak{G}$ is a submonoid of $H_1(L;\Z)$, i.e. $0\in\mathfrak{G}$, and if $\beta_1,\beta_2\in \mathfrak{G}$ then $\beta_1+\beta_2\in \mathfrak{G}$;
	\item The image of $\mathfrak{G}$ under the energy homomorphism $E\colon H_1(L;\Z)\to \R$ is discrete;
	\item For any $\beta\in \mathfrak{G}$, its energy $E(\beta)\geq 0$, and the only $\beta\in \mathfrak{G}$ with $E(\beta)=0$ is $\beta=0$;
	\item For each energy level $\lambda\in \R$, there are only finitely many $\beta\in \mathfrak{G}$ with $E(\beta)=\lambda$.
\end{enumerate}
\end{definition}

Let $H\in C_c^\infty(\C^n\times [0,1]_t)$ be a compactly supported time-dependent Hamiltonian function. 
The \emph{Hofer norm} of the Hamiltonian $H$ is
\begin{equation}\label{Eq: HoferNorm}
	\norm{H} := \int_0^1 \left(\max H_t - \min H_t\right)\,dt.
\end{equation}

\begin{definition}\label{Defn: module}
	We say a subset $\mathfrak{N}\subset H_1(L;\Z)$ is a \emph{module of $H$-perturbed curve classes over $\mathfrak{G}$} if
\begin{enumerate}
	\item $\mathfrak{N}$ is a module over $\mathfrak{G}$, i.e.  $0\in \mf{N}$, and if $\eta\in \mathfrak{N}$ and $\beta\in\mathfrak{G}$, then $\eta+\beta\in \mathfrak{N}$ (in particular $\mf{G}\subset \mf{N}$);
	\item The image of $\mathfrak{N}$ under the energy homomorphism $E\colon H_1(L;\Z)\to \R$ is discrete;
	\item For any $\eta\in \mathfrak{N}$, its energy $E(\eta)\geq -\norm{H}$;
	\item For each energy level $\lambda\in \R$, there are only finitely many $\eta\in \mathfrak{N}$ with $E(\eta)=\lambda$. 
\end{enumerate}
\end{definition}

\bigskip

As before, let $\textbf{C}$ denote one of the three: $\CFr,\CBa$, or $\CC^*(\CBa,\CBa)$.

\begin{definition}\label{Defn: G-gapped}
	An element $x\in \wh{\textbf{C}}$ is \emph{$\mathfrak{G}$-gapped} if we can decompose $x$ as 
	\[ x = \sum_{\beta\in\mathfrak{G}} x(\beta), \quad x(\beta)\in \textbf{C}(\beta). \]

	This expression for $x$ makes sense as an element in $\wh{\textbf{C}}$ since we can sort elements in $\mathfrak{G}$ so that $\mathfrak{G} = \{\beta_0,\beta_1,\beta_2,\cdots \}$ where $0\leq E(\beta_0) \leq E(\beta_1)\leq E(\beta_2)\leq \cdots\to  +\infty$ by Definition \ref{Defn: monoid}; so $x=\sum_{i=1}^\infty x(\beta_i)$ is an element in $\wh{\textbf{C}}$.
	
	Define $\wh{\textbf{C}}_\mathfrak{G}$ to be the subspace of $\wh{\textbf{C}}$ consisting of $\mathfrak{G}$-gapped elements. 
	
	Similarly, an element $y\in \wh{\textbf{C}}$  is \emph{$\mathfrak{N}$-gapped} if we can decompose $y$ as \[ y = \sum_{\eta\in\mathfrak{N}} y(\eta),\quad y(\eta)\in \textbf{C}(\eta). \]
	Define $\wh{\textbf{C}}_{\mathfrak{N}}$  to be the subspace of $\wh{\textbf{C}}$ consisting of $\mathfrak{N}$-gapped elements.
\end{definition}

An element $\Phi$ in $\wh{\CC^*}_\mf{G}(\CBa,\CBa)$ is a collection of operations 
\[\Phi_{\ell,\beta}:=\Phi_\ell(\beta)\colon \CBa(a_1)\otimes \dots \otimes \CBa(a_\ell) \to \CBa(a_1+\dots +a_\ell + \beta) \]
for each $\ell\in \Z_{\geq 0}$, $\beta\in \mf{G}$, and $a_1,\dots, a_\ell \in H_1(L;\Z)$. We sometimes refer to this as a \emph{$\mf{G}$-gapped operator system}
(this is similar to \cite{YuanHangThesis}, Definition 2.1; later on the $\mf{G}$-gapped operator systems we use will eventually need to satisfy the requirement in \emph{ibid.} that $\Phi_0(0)=0$ which we don't require for now).

\bigskip

Recall that the closed-open string map \[ \CO \colon \CFr \to \CC^*(\CBa,\CBa) \]
is a homomorphism of dg Lie algebras which preserves the decomposition into $H_1(L;\Z)$ classes (Lemma \ref{Lm: dgLieMorphism}).
Thus it induces maps
\begin{align*}
	\CO\colon \wh{\CFr}_\mf{G}\to \wh{\CC^*}_\mf{G}(\CBa,\CBa),\quad \CO \colon \wh{\CFr}_\mf{N}\to \wh{\CC^*}_\mf{N}(\CBa,\CBa).
\end{align*}
In particular, given a $\mf{G}$-gapped element $x\in \wh{\CFr}_\mf{G}$, applying the closed-open map gives a $\mf{G}$-gapped operator system $\CO(x)\in \wh{\CC^*}_{\mf{G}}(\CBa,\CBa).$

\begin{notation}
	\label{Ntn: PositiveG}
	For a monoid of curve classes $\mf{G}\subset H_1(L;\Z)$, 
	we write \[\mf{G}^+:= \mf{G} \setminus \{0\}.\]
	We denote by $\wh{\mathbf{C}}_{\mf{G}^+}$ the collection of $\mf{G}$-gapped elements with no zero-energy term, i.e. $x\in \wh{\mathbf{C}}_{\mf{G}^+}\subset \wh{\mathbf{C}}_{\mf{G}}$  if \[ x = \sum_{\beta\in \mf{G}^+} x(\beta),\quad x(\beta)\in \mathbf{C}(\beta). \]
\end{notation}
\begin{definition}
	An element $x\in \wh{\Cfr{-1}}_{\mathfrak{G}^+}$ is an \emph{Maurer-Cartan element} if it satisfies the \emph{Maurer-Cartan equation}  \[ \partial  x - \frac{1}{2}[x,x ] = 0. \]
\end{definition}

The following definition is a working definition tailored for our situation:
\begin{definition}\label{Defn: G-gapped assdefo}
	A \emph{$\mathfrak{G}$-gapped curved dg associative deformation} of $\CBa$ is a $\mf{G}$-gapped operator system $\mm\in \wh{\CC^*}_{\mf{G}}(\CBa,\CBa)$ with \[ \mm_{\ell,\beta}\colon \CBa(a_1)\otimes \dots \otimes \CBa(a_\ell)\to \CBa(a_1+\dots + a_\ell + \beta) \] for all $\ell\in \Z_{\geq 0}$, $\beta\in \mf{G}$, and $a_1,\dots, a_\ell \in H_1(L;\Z)$, such that
	\begin{enumerate}
		\item $\mm\in \wh{\CC^{-1}}_\mf{G}(\CBa,\CBa)$. This is equivalent to: for each $\ell\in \Z_{\geq 0}$ and $\beta\in \mf{G}$, \[ \deg \mm_{\ell,\beta} = \ell - 2;\]
		\item If $\beta = 0$, we require \[\mm_{0,0}=0, \quad \mm_{1,0}=\partial, \quad \mm_{2,0} = \bullet;\]
		\item If either of the following holds:
		\begin{enumerate}
			\item $\beta=0$ and $\ell > 2$;
			\item $\beta \neq 0$ (so that $E(\beta)>0$) and $\ell \geq 2$,
		\end{enumerate}
		we require $\mm_{\ell,\beta} = 0$.
		\item The following identities hold:
		\begin{enumerate}
			\item $\mm_1(\mm_0)= 0$.
			That is, for any $\beta\in \mf{G}$, \[ \sum_{\substack{\beta_1,\beta_2\in \mf{G}\\ \beta_1+\beta_2=\beta}} \mm_{1,\beta_1}(\mm_{0,\beta_2})=0; \]
			\item $\mm_1(\mm_1 \alpha) = \mm_0 \bullet \alpha - \alpha \bullet \mm_0$. That is, for any $\beta \in \mf{G}$ and $\alpha \in \CBa$, \[ \sum_{\substack{\beta_1,\beta_2\in \mf{G}\\ \beta_1+\beta_2=\beta}} \mm_{1,\beta_1}(\mm_{1,\beta_2} \alpha) = \mm_{0,\beta} \bullet \alpha - \alpha \bullet \mm_{0,\beta};\]
			\item $\mm_1$ is a derivation with respect to $\bullet$. That is, for any $\beta \in \mf{G}$ and $\alpha_1,\alpha_2 \in \CBa$, \[ \mm_{1,\beta}(\alpha_1\bullet \alpha_2) = \mm_{1,\beta}(\alpha_1)\bullet \alpha_2 + (-1)^{|\alpha_1|} \alpha_1 \bullet \mm_1(\alpha_2). \]
		\end{enumerate}
	\end{enumerate}
\end{definition}

\begin{lemma}\label{Lm: MC Ainfty}
	Suppose that $x\in \wh{\Cfr{-1}}_{\mathfrak{G}^+}$ satisfies the Maurer-Cartan equation. 
	Then the $\mf{G}$-gapped operator system given by \[ \mm_{0,\beta}:= \begin{cases}
			0, & \beta = 0 \\
			-\CO_{0,\beta}(x), & \beta \in \mf{G}^+
		\end{cases},
		\quad \mm_{1,\beta}:= \begin{cases}
			\partial, & \beta =0\\
			-\CO_{1,\beta}, & \beta \in \mf{G}^+
		\end{cases}, \quad \mm_{2,\beta}:= \begin{cases}
			\bullet, & \beta = 0\\
			0, & \beta \in \mf{G}^+
		\end{cases}  \]
		for any $\beta\in \mf{G}$,
		and $\mm_{\ell,\beta}=0$ for $\ell \geq 3$, is a $\mf{G}$-gapped curved dg associative deformation of $\CBa$.
\end{lemma}

\begin{proof}
	$\mm$ is a $\mf{G}$-gapped operator system by construction, and properties (1) to (3) in Definition \ref{Defn: G-gapped assdefo} is clear.

	For (4): for brevity write $\Phi:= \CO(x)$. Since $\CO$ is a homomorphism of dg Lie algebras and by assumption $x$ satisfies the Maurer-Cartan equation, we have \[ \delta \Phi - \frac{1}{2}[\Phi,\Phi] = 0 \quad \textup{in } \wh{\CC^{-1}}(\CBa,\CBa). \]
	By inspecting different arities of this equation, we get, for any $\alpha,\alpha_1,\alpha_2\in \CBa$, 
	\begin{align*}
		\begin{cases}
			\partial \Phi_0 - \Phi_1(\Phi_0) = 0 \\
			\Phi_1(\Phi_1(\alpha)) = \partial(\Phi_1(\alpha)) + \Phi_1(\partial \alpha) + (\alpha \bullet \Phi_0 - \Phi_0\bullet \alpha) \\
			\Phi_1(\alpha_1\bullet \alpha_2) = \Phi_1(\alpha_1) \bullet \alpha_2 + (-1)^{|\alpha_1|} \alpha_1\bullet \Phi_1(\alpha_2)
		\end{cases}.
	\end{align*}
	Then the three identities in (4) follow:
	\begin{itemize}
		\item For (a): \begin{align*}
			\mm_1(\mm_0) = -(\partial - \Phi_1)(\Phi_0) = 0.
		\end{align*}
		\item For (b): \begin{align*}
			\mm_1^2\alpha &=  (\partial - \Phi_1)^2\alpha = -\partial (\Phi_1(\alpha)) -\Phi_1(\partial \alpha) +\Phi_1\circ \Phi_1(\alpha)  \\
			&= \alpha \bullet \Phi_0 - \Phi_0\bullet \alpha = \mm_0\bullet \alpha - \alpha \bullet \mm_0.
		\end{align*} 
		\item For (c): 
		\begin{align*}
			\mm_1(\alpha_1\bullet \alpha_2) = \partial(\alpha_1\bullet\alpha_2) - \Phi(\alpha_1\bullet\alpha_2) = \mm_1(\alpha_1)\bullet \alpha_2 +(-1)^{|\alpha_1|} \alpha_1\bullet \mm_1(\alpha_2).
		\end{align*}
	\end{itemize}
\end{proof}

\subsection{Deformation in closed string}\label{Subsc: ClosedStringDefo}

In this section we discuss the deformation of the closed string state space coming from moduli spaces of pseudo-holomorphic discs, in the form of a Maurer-Cartan element in $\CFr$. 

Let $L$ be a closed, oriented, spin manifold of dimension $n$.
Assume we are given a Lagrangian embedding of $L$ into $\C^n$ (equipped with the standard symplectic structure $\omega=\omega_{\C^n}$).

Let $H\in C_c^\infty(\C^n\times [0,1]_t)$ be a compactly supported time-dependent Hamiltonian function. 
For each $t\in [0,1]$, let $X_{H_t}$ be the Hamiltonian vector field on $\C^n$ associated to $H_t:=H(-,t)\in C_c^\infty(\C^n)$, i.e. satisfying \[
	dH_t(-) = -\omega_{\C^n}(X_{H_t},-).
\] 
Let $(\varphi_H^t)_{t\in[0,1]}$ be the time-$t$ flow of the 1-parameter family of vector fields $X_{H_t}$.
We further assume that 
\begin{assumption}\label{Asmp: Disp Hamil}The Hamiltonian $H\in C_c^\infty(\C^n\times [0,1]_t)$ satisfies
	\begin{itemize}
	\item $H$ is a \emph{displacing Hamiltonian function}, i.e. $\varphi^1_H(L) \cap L = \emptyset$;
	\item $H_t\equiv 0 $ when $t\in [0,1/3]\cup [2/3,1]$. 
\end{itemize}
\end{assumption}
Such a displacing Hamiltonian function exists for any compact $L\subset \C^n$. Fix such a choice.

\begin{definition}\label{Defn: LL^0}
We define the element \[
	\LL^0\in \Cfr{1}(0) = \prod_{k\in \Z_{\geq 0}} C^\dR_{\dim L + k} (\mathscr{L}^{k+1}(0))\]
 as follows: 
 \begin{itemize}
 	\item For $k=0$,  consider the map $L\to \mathscr{L}^1(0)$ defined by $y\mapsto [\underline{y}]\in \mathscr{L}^1(0)\subset \Pi_1 L$. The de Rham chain $(-1)^{\dim L + 1}[(L\to \mathscr{L}^1(0); 1\in \mathscr{A}_c^0(L))]$ defines a closed cycle in $C^\dR_{\dim L}(\mathscr{L}^1(0))$, which we set to be the $k=0$ component of $\LL^0$;
 	\item For all $k> 0$, set the $k$-th component of $\LL^0$ to be 0.
  \end{itemize}
\end{definition}

\bigskip
\begin{theorem}\label{Thm: VFCmain}
	Under the setup above, there exists the following data:
	\begin{itemize}
	\item A monoid of curve classes $\mathfrak{G}\subset H_1(L;\Z)$ (Definition \ref{Defn: monoid})  and a module $\mathfrak{N}\subset H_1(L;\Z)$ of $H$-perturbed curve classes over $\mathfrak{G}$ (Definition \ref{Defn: module});
	\item For each $\beta\in\mathfrak{G}^+$ a chain $\MM(\beta)\in \Cfr{-1}(\beta)$, and for each $\eta\in \mathfrak{N}$ a chain $\NN^{\geq 0}(\eta)\in \Cfr{2}(\eta)$ and a chain $\NN^0(\eta)\in \Cfr{1}(\eta)$;
\end{itemize}
	such that 
	\begin{enumerate}
		\item The element \[\MM := \sum_{\beta\in\mathfrak{G}^+} \MM(\beta) \in \wh{\Cfr{-1}}_{\mathfrak{G}^+}, \]where $\MM(\beta)\in \Cfr{-1}(\beta)$,
	satisfies the Maurer-Cartan equation  \[ \partial \MM + \frac{1}{2} \big[\MM,\MM\big]  = 0. \]
		\item The elements \[\NN^{\geq 0} := \sum_{\eta\in\mathfrak{N}} \NN^{\geq 0}(\eta) \in \wh{\Cfr{2}}_\mathfrak{N},\quad 
	\NN^{ 0} := \sum_{\eta\in\mathfrak{N}} \NN^{ 0}(\eta)\in \wh{\Cfr{1}}_\mathfrak{N}, \]where $\NN^{\geq 0}(\eta)\in \Cfr{2}(\eta)$, $\NN^0(\eta) \in \Cfr{1}(\eta)$,
	satisfy \[ \partial  \NN^{\geq 0}- \big[\MM, \NN^{\geq 0}\big] = \NN^{0}.   \]
	\item $\NN^0(\eta)\neq 0$ only if $\eta \in \mf{G}$. Moreover, in case $\eta = 0$, $\NN^0(0) \in \Cfr{1}(0)$ is a cycle which is homologous  to $\LL^0$. 
	\end{enumerate}
\end{theorem}

This theorem is analogous to Theorem 5.1 (and section 5 in general) of \cite{Irie2}. The elements $\MM, \NN, \NN^0$ are constructed using virtual fundamental chains of moduli spaces of pseudo-holomorphic and perturbed pseudo-holomorphic discs (see section \ref{Sec: SketchPf} for the geometric motivations of these elements). 
The  proof of this theorem requires the theory of Kuranishi structures, similar as those in \cite{Irie2} (but simpler since we are using a simplified chain model by \cite{WangThesis}), and is contained in Appendix \ref{Sec: Kuranishi}.

\subsection{Deformation in open string}\label{Subsc: Ainfty}
In this section we construct a $\mf{G}$-gapped curved dg associative algebra deforming the Pontryagin algebra structure on $\CBa$.

\subsubsection{Statement}
Recall our setup in section \ref{Subsc: ClosedStringDefo}: $L$ is a closed, oriented, spin manifold of dimension $n$, together with a Lagrangian embedding into $\C^n$; $H\in  C_c^\infty(\C^n\times [0,1]_t)$ is a Hamiltonian satisfying Assumption \ref{Asmp: Disp Hamil}.

\begin{theorem}\label{Cor: BasedAinfty}
	Under the setup above, there exists the following data:
	\begin{itemize}
		\item A monoid of curve classes $\mathfrak{G}\subset H_1(L;\Z)$ (Definition \ref{Defn: monoid})  and a module $\mathfrak{N}\subset H_1(L;\Z)$ of $H$-perturbed curve classes over $\mathfrak{G}$ (Definition \ref{Defn: module});
		\item A $\mf{G}$-gapped curved dg associative deformation of $\CBa$, which we denote as  $\mm \in \wh{\CC^{-1}}_\mf{G}(\CBa,\CBa)$;
		\item For each $\eta\in\mathfrak{N}$, a chain $\NNN^{\geq 0}(\eta)\in \Cba{1}(\eta)$ and a chain $\NNN^0(\eta)\in \Cba{0}(\eta)$,
	\end{itemize}
	such that 
	\begin{enumerate}
		\item $\NNN^0(\eta)\neq 0$ only if $\eta \in \mf{G}$. Moreover, in case $\eta = 0$, $\NNN^0(0) \in \Cba{0}(0)$ is a cycle which is homologous to $(-1)^{\dim L} [\underline{\star}]$ (here $[\underline{\star}]\in \Cba{0}(0)$ is the unit of $\CBa$; see Definition \ref{Defn: CBaUnit}).
		\item When the Maslov class $\mu_L$ vanishes, (under a deformation of the structures and data keeping all the above properties) \[ \mm_0=0\] and  the elements \[ \NNN^{\geq 0}:= \sum_{\eta\in\mathfrak{N}} \NNN^{\geq 0}(\eta)\in \wh{\Cba{1}}_{\mathfrak{N}} , \quad \NNN^0:= \sum_{\eta\in\mathfrak{N}} \NNN^0(\eta) \in \wh{\Cba{0}}_{\mathfrak{N}}\]
		 satisfy \[ \mm_1(\NNN^{\geq 0}) = \NNN^0.  \]
	\end{enumerate}
\end{theorem}

\begin{remark}\label{Rmk: PfBasedAinfty}
Pushing forward the Maurer-Cartan element $\MM\in \wh{\Cba{-1}}_{\mf{G}^+}$ and using Lemma \ref{Lm: MC Ainfty} to cook up the $\mf{G}$-gapped curved dg associative deformation $\mm$ of $\CBa$, and defining $\NNN^{\geq 0}:= \CO_0(\NN^{\geq 0})$, $\NNN^0:= \CO_0(\NN^0)$, would  satisfy everything in the Theorem but property (2).

Ideally, the term $\mm_0$ should be 0 in case $L$ has vanishing Maslov class, morally because the moduli spaces of Maslov-zero curves with one marked point have virtual dimension $\dim L -2$ and the curves do not generically intersect the basepoint $\star\in L$. However, a technical issue that appears is that the moduli spaces of curves with a large number of boundary marked points still have positive virtual dimensions after intersecting with $\star\in L$, and therefore has non-zero contribution to the curvature $\mm_0$ at the chain level. 
To deal with this, we use the machinery of bounding chains following \cite{FOOO1}.
For completeness, we sketch the obstruction theory of bounding chains developed in e.g. section 3.6 of \cite{FOOO1} in our very specific setting of curved dg algebras (instead of curved $A_\infty$-algebras in general) in section \ref{Subsc: BoundingChain}, before going into the proof of Theorem \ref{Cor: BasedAinfty}.
\end{remark}

\subsubsection{Obstruction theory for bounding chains}\label{Subsc: BoundingChain}
In this section, we temporarily denote by $\mm \in \wh{\CC^{-1}}_\mf{G}(\CBa,\CBa)$ the $\mf{G}$-gapped curved dg associative deformation of $\CBa$ given by $\MM\in \wh{\Cba{-1}}_{\mf{G}^+}$ and Lemma \ref{Lm: MC Ainfty}. 
	\begin{definition}[\cite{FOOO1} Definition 3.6.4, 3.6.16]
	A \emph{bounding chain over $\mathfrak{G}$} is an element $b\in \wh{\Cba{-1}}_{\mathfrak{G}^+}$ such that the Maurer-Cartan equation\begin{align}\label{Eq: BoundingChain}
		\mm_0 + \mm_1 b + b\bullet b = 0
	\end{align}holds. We say the curved algebra $\CBa$ is \emph{unobstructed over $\mf{G}$} if it admits a bounding chain $b$ over $\mf{G}$.
\end{definition}

We now suppose $L$ has vanishing Maslov class $\mu_L$.
\begin{lemma}\label{Lm: BoundingChain}
	If the Maslov class $\mu_L= 0 \in H^1(L;\Z)$, then the deformed algebra $(\CBa, \mm)$ is unobstructed over $\mf{G}$.
\end{lemma}
\begin{proof}
	This follows from the obstruction theory of bounding (co)chains developed in e.g. Theorem 3.6.18 of \cite{FOOO1}. We repeat the argument here because of different grading conventions.
	
	The bounding chain $b$ is to be constructed by induction on energy levels.
	We sort the elements in the monoid $\mf{G}$ by energy:
	\begin{align*}
		\mf{G} = \{0=\beta_0,\beta_1,\beta_2,\dots \}\textup{ such that } 0=E(\beta_0)\leq E(\beta_1)\leq E(\beta_2)\leq \cdots
	\end{align*}
	Also, sort the set  $E(\mathfrak{G})\subset \R$ in order:
	\[ E(\mf{G})= \{0=E^{(0)}<E^{(1)}<E^{(2)}<\cdots\}.\] 
	These are possible by the gappedness assumption (Definition \ref{Defn: monoid}).
	Denote by $K^{(i)}\in \Z_{\geq 1}$ the largest integer such that $E(\beta_{K^{(i)}})=E^{(i)}$.
	In the following proof, for a given energy level $\lambda$, an expression holds ``mod $\lambda$'' is taken to mean that it holds in \[\CBa/\ms{F}^{\lambda}\CBa\cong \bigoplus_{E(a)\leq \lambda} \CBa(a).\]

	For the base case, at energy level $E^{(1)}$,  we have $E(\beta_1)=\dots =E(\beta_{K^{(1)}})=E^{(1)}$. 
	Define 
	\begin{align*}
		o_i:= \mm_{0,\beta_i}\in \Cba{-2},\quad i=1,\dots, K^{(1)}.
	\end{align*}
	We claim that $\partial o_i=0$ for $i=1,\dots, K^{(1)}$. 
	Indeed, by Definition \ref{Defn: G-gapped assdefo} (4a), $\mm_1(\mm_0) = 0$, and since $\mm_{0,0}=0$, we have that for 
	\begin{align*}
		0=\mm_1(\mm_0) \equiv  \sum_{i=1}^{K^{(1)}}\partial \mm_{0,\beta_i}  \mod E^{(1)}.
	\end{align*}
	Then $\partial o_i=\partial \mm_{0,\beta_i} =0$ because they are in different summands $\Cba{*}(\beta_i)\subset \CBa$.
Moreover, by Theorem \ref{Thm: cosimpHlgy} and the Maslov-zero assumption (so that the grading works out), $\Hba{-2}(\beta_i)\cong H_{-2}(\Omega_\star (\beta_i))=0$. Therefore \[\partial b_i+o_i = 0\] for some $b_i\in \Cba{-1}(\beta_i)$. Define
\begin{align*}
	b^{(1)}:= \sum_{i=1}^K b_i\in \Cba{-1}.
\end{align*}
Then we have \begin{align*}
	\mm_0 + \mm_1\left(b^{(1)}\right) + b^{(1)}\bullet b^{(1)} \equiv 0 \mod E^{(1)}.
\end{align*}
	
	Now suppose that for some $j>1$, we have found $b^{(j-1)}= \sum_{i=1}^{K^{(j-1)}} b_i$ where $b_i\in \Cba{-1}(\beta_i)$, such that 
	\begin{align}\label{Eq: BoundingChain (j-1)}
		\mm_0 + \mm_1\left(b^{(j-1)}\right) + b^{(j-1)}\bullet b^{(j-1)} \equiv 0 \mod E^{(j-1)}.
	\end{align}
	For $i=K^{(j-1)}+1,K^{(j-1)}+2, \dots , K^{(j)}$, define the component of the left-hand side in $\Cba{-2}(\beta_i)$ as $o_i\in \Cba{-2}(\beta_i)$.
	We now claim that $\partial o_i=0$ for all such $i$. This follows from applying $\mm_1$ to the left-hand side of \eqref{Eq: BoundingChain (j-1)} in two different ways:
	\begin{itemize}
		\item On the one hand, since $\mm_1= \partial + \sum_{i=1}^\infty \mm_{1,\beta_i}$,\begin{align*}
			\mm_1\left(\mm_0 + \mm_1(b^{(j-1)}) + b^{(j-1)}\bullet b^{(j-1)}\right) \equiv \partial \sum_{i=K^{(j-1)}+1}^{K^{(j)}} o_i \mod E^{(j)};
		\end{align*}
		\item On the other hand, by the $A_\infty$-identities,
		\begin{align*}
			&\mm_1\left(\mm_0 + \mm_1(b^{(j-1)}) + b^{(j-1)}\bullet b^{(j-1)}\right)\\=& \, \mm_1(\mm_0) + \mm_1^2(b^{(j-1)}) + \mm_1(b^{(j-1)}\bullet b^{(j-1)}) \\
			= & \left( \mm_0+\mm_1(b^{(j-1)}) \right)\bullet b^{(j-1)} - b^{(j-1)} \bullet \left( \mm_0+\mm_1(b^{(j-1)}) \right)\\
			\equiv & - b^{(j-1)}\bullet b^{(j-1)}\bullet b^{(j-1)} + b^{(j-1)} \bullet b^{(j-1)}\bullet b^{(j-1)}\mod E^{(j)}\\
			\equiv & \,\, 0 \mod E^{(j)}.
		\end{align*}
	\end{itemize}
	This shows that $\partial o_i=0$ in each $\CBa(\beta_i)$. Again by Theorem \ref{Thm: cosimpHlgy} and the Maslov-zero assumption, $[o_i]\in \Hba{-2}(\beta_i)=0$ is null-homologous and therefore \begin{align*}
		\partial b_i + o_i =0
	\end{align*}
	for some $b_i\in \Cba{-1}(\beta_i)$ for each $i=K^{(j-1)}+1,\dots K^{(j)}$. Then it follows that \begin{align*}
\mm_0 + \mm_1\left(b^{(j)}\right) + b^{(j)}\bullet b^{(j)} \equiv 0 \mod E^{(j)}.	\end{align*}
Since this induction can be carried out for every $j$, we can take the limit
\begin{align*}
	b:= \lim_{j\to\infty}b^{(j)} \in \Cba{-1}
\end{align*}
to get a bounding chain over $\mf{G}$.
\end{proof}

\begin{definition}\label{Defn: m1BoundingChainDeformed}
	Suppose $b\in \wh{\Cba{-1}}_\mf{G}$ is a bounding chain over $\mf{G}$. 
	Then define \begin{align*}
		\mm_1^b\colon \wh{\CBa}\to \wh{\CBa}; \quad \mm_1^b(\alpha):= \mm_1(\alpha) +b\bullet \alpha - (-1)^{|\alpha|} \alpha \bullet b.
	\end{align*}
\end{definition}
By \eqref{Eq: BoundingChain}, it can be verified that $(\mm_1^b)^2=0$.

For notational simplicity, we temporarily write  
\begin{align*}
	\Phi:= \CO(\MM), \quad \Psi^{\geq 0} := \CO(\NN^{\geq 0}),\quad \Psi^{0}:= \CO(\NN^{0}).
\end{align*}

\begin{lemma}\label{Lm: ThmPfIdKilled}
	$\mm_1^b(\Psi_0^{\geq 0} + \Psi_1^{\geq 0}(b)) = \Psi_0^{0} + \Psi_1^{0}(b),$ where
	\begin{enumerate}
		\item $\Psi_1^{\geq 0}(b)\in \wh{\Cba{1}}_\mf{N}$;
		\item $\Psi_1^{0}(b) \in \wh{\Cba{0}}_\mf{N}$ and the component of $\Psi_1^{0}(b)$ in $\Cba{0}(\eta)$ is non-zero only if $\eta\in \mf{G}^+$.
	\end{enumerate}
\end{lemma}
\begin{proof}
By construction of $\CO$, each of $\Phi, \Psi^{\geq 0}$ and $\Psi^{0}$ only have 0- and 1-ary components in $\CC^*(\CBa,\CBa)$.
Since $\CO$ is a homomorphism of dg Lie algebras, we have \[ \delta \Psi^{\geq 0} - \left[\Phi, \Psi^{\geq 0}\right] = \Psi^{ 0}. \]
This equation translates to 
\begin{align}\label{Eq: BoundChainDisp}
	\begin{cases}
		\mm_1\Psi_0^{\geq 0} - \Psi_1^{\geq 0}(\mm_0) = \Psi_0^{0}\\
		\mm_1(\Psi_1^{\geq 0}(\alpha)) - \Psi_1^{\geq 0}(\mm_1(\alpha)) + \Psi_0^{\geq 0} \bullet \alpha - (-1)^{|\alpha|} \alpha\bullet \Psi_0^{\geq 0} = \Psi_1^{0}(\alpha)\\
		\Psi_1^{\geq 0}(\alpha_1\bullet\alpha_2) = \alpha_1\bullet\Psi_1^{\geq 0}(\alpha_2) + \Psi_1^{\geq 0}(\alpha_1)\bullet \alpha_2
	\end{cases}
\end{align}
for all $\alpha,\alpha_1,\alpha_2\in \CBa$.

Then 
\begin{align*}
	\mm_1^b(\Psi_0^{\geq 0} + \Psi_1^{\geq 0}(b)) &= \mm_1\Psi_0^{\geq 0} + \mm_1(\Psi_1^{\geq 0}(b)) + b\bullet \Psi_0^{\geq 0} + b\bullet \Psi_1^{\geq 0}(b) + \Psi_0^{\geq 0}\bullet b + \Psi_1^{\geq 0}(b)\bullet b\\
	&= \Psi_0^{0}+\Psi_1^{\geq 0}(\mm_0) + \Psi_1^{0}(b) + \Psi_1^{\geq 0}(\mm_1(b))+ b \bullet \Psi_1^{\geq 0}(b) + \Psi_1^{\geq 0}(b)\bullet b\\
	&= \Psi_0^{0} + \Psi_1^{0}(b),
\end{align*}
by the the properties in Definition \ref{Defn: G-gapped assdefo} (4) as well as \eqref{Eq: BoundChainDisp}.

The claims that $\Psi_1^{\geq 0}(b), \Psi_1^{0}(b) \in \wh{\CBa}_\mf{N}$ follow because $\mf{N}$ is a module over $\mf{G}$ (Definition \ref{Defn: module}), $\Psi^{\geq 0}, \Psi^{0}\in \wh{\CC^*}_{\mf{N}}(\CBa,\CBa)$, and $b\in \wh{\CBa}_{\mf{G}^+}$.
Finally the component of $\Psi_1^{0}(b)$ in $\CBa(\eta)$  is non-zero only if $\eta\in \mf{G}^+$, since $\NN^{0}(\eta) \neq 0$ unless $\eta \in \mf{G}$ (Theorem \ref{Thm: VFCmain} (3)), and $b \in \wh{\Cba{-1}}_{\mf{G}^+}$ (so in particular the energy of $\Psi_1^{0}(b)=(\CO_1(\NN^{0}))(b)$ has to be strictly positive).
\end{proof}

\subsubsection{Proof of Theorem \texorpdfstring{\ref{Cor: BasedAinfty}}{4.12}}

\begin{lemma}\label{Lm: COtoStar}
	The element $\LL^0 \in \Cfr{1}(0)$ defined in Definition \ref{Defn: LL^0} pushes forward under $\CO$ to an element $\CO(\LL^0)\in \CC^{1}(\CBa,\CBa)(0)$ whose 0-nary component $\CO_0(\LL^0)\in \Cba{0}(0)$ corresponds to 
	\begin{align*}
		\CO_0(\LL^0) = (-1)^{\dim L}[\underline{\star}]
	\end{align*}
	where $[\underline{\star}]\in \Cba{0}(0)$ is the unit of $\CBa$ (Definition \ref{Defn: CBaUnit}).
\end{lemma}
\begin{proof}
	By Definition \ref{Defn: LL^0}, the element $\LL^0$ has components
	\begin{align*}
		\LL^0(0,k) = \begin{cases}
			(-1)^{\dim L + 1}[(L\xrightarrow{\iota} \mathscr{L}^1(0); 1\in \mathscr{A}_c^0(L))], & k = 0\\
			0, & k> 0
		\end{cases}
	\end{align*}
	where $\iota\colon y \mapsto [\underline{y}]\in \ms{L}^1(0)$.
	By \eqref{Eq: AnomalyMap}, \[ \mf{o}\left[\big(L\to \mathscr{L}^1(0); 1\in \mathscr{A}_c^0(L)\big)\right] =\left[\big( (\ev_0\circ \iota)^{-1}(\star) \to \Omega^1_\star(0); 1 \big)\right]  \]
	where $\ev_0\circ\iota \colon L\to L$ is the identity, and therefore 
	\begin{align*}
		\mf{o}(\LL^0(0,k)) &= \begin{cases}
		(-1)^{\dim L + 1} [(\pt \to \Omega^1_\star (0); 1\in \ms{A}^0_c(\pt))], & k = 0\\
		0, & k > 0
	\end{cases} \\
	&= (-1)^{\dim L + 1}[\underline{\star}](0,k),
	\end{align*}
	where $\pt \to \Omega^1_\star(0)$ sends the point to the constant based loop $[\underline{\star}]\in \Omega^1_\star (0)$.
	
	Finally by \eqref{Eq: AnomalyMapState} and \eqref{Eq: CO0 OpenClosed}, $\CO_0(\LL^0)= (-1)^{|\LL^0|}\mf{o}(\LL^0) = (-1)^{\dim L }[\underline{\star}]$.
\end{proof}

\begin{proof}[Proof of Theorem \ref{Cor: BasedAinfty}]
	It remains to deal with the case where the Maslov class $\mu_L$ vanishes (see Remark \ref{Rmk: PfBasedAinfty}). 
	Under the setting of Lemma \ref{Lm: BoundingChain}  we choose a bounding chain $b$, and set consider the curved dg associative deformation with  unary term $\mm_1^b$ (Definition \ref{Defn: m1BoundingChainDeformed}), binary term $\bullet$, and all the other operations set to 0.
	Set $\NNN^{\geq 0} := \Psi_0^{\geq 0} + \Psi_1^{\geq 0}(b)$ and $\NNN^{0} := \Psi_0^{0} + \Psi_1^{0}(b)$
	in the notation of Lemma \ref{Lm: ThmPfIdKilled}.
	Finally, to show $\NNN^0(0)$ is a cycle which is homologous to $(-1)^{\dim L} [\underline{\star}]$, we have $\NNN^0(0)= \Psi_0^{0}(0)= \CO_0(\NN^{0})$ since by Lemma \ref{Lm: ThmPfIdKilled} the 0-ary component of $\Psi_1^{0}(b)$ vanishes, and $\NN^{0}(0) \in \Cfr{0}(0)$ is a cycle which is homologous to $\LL^0$; the conclusion then follows from Lemma \ref{Lm: COtoStar}.
\end{proof}

\bigskip

\section{Lagrangians in \texorpdfstring{$\C^n$}{Cn} with vanishing Maslov classes}
\label{Sec: MainProof}

\subsection{Proof of Theorem \ref{Thm: Main}}\label{Subsc: PfMainThm}
For convenience, we restate Theorem \ref{Thm: Main} here:
\begin{theorem}\label{Thm: MainRestate}(Theorem \ref{Thm: Main})
	If $L$ is a closed spin manifold and $\pi_2(L)=0$, then $L$ does not admit a Lagrangian embedding into $\C^n$ with vanishing Maslov class.
\end{theorem}

The proof goes through the homology of the based loop space:
\begin{lemma}\label{Lm: BasedHlgy}
For any path connected space $L$,
	\[H_1(\Omega_\star L;\Z) = \bigoplus_{a \in \pi_1 L} \pi_2(L).\]
\end{lemma}
\begin{proof}
	Since $\Omega_\star L$ is an $H$-space, its fundamental group is abelian and all of its components $\Omega_\star(a)$ (labeled by $a\in \pi_1 L\cong \pi_0\Omega_\star L$) are homotopy equivalent to each other. Therefore by Hurewicz theorem,
	\begin{align*}
		H_1(\Omega_\star L)\cong \bigoplus_{a\in \pi_1 L\cong \pi_0 \Omega_\star L} H_1(\Omega_\star (a)) \cong \bigoplus_{a\in \pi_1L} \pi_1(\Omega_\star (a)) \cong  \bigoplus_{a\in \pi_1 L} \pi_2 L.
	\end{align*}
\end{proof}

\begin{proof}
	[Proof of Theorem \ref{Thm: Main}]	
	As in the proof of Lemma \ref{Lm: BoundingChain}, for a given energy level $\lambda$, we say an expression holds ``mod $\lambda$'' is taken to mean that it holds in \[\CBa/\ms{F}^{\lambda}\CBa\cong \bigoplus_{E(a)\leq \lambda} \CBa(a).\]

	By Theorem \ref{Cor: BasedAinfty}, assuming $L$ admits a Maslov-zero Lagrangian embedding into $\C^n$,  there exist elements $\NNN^{\geq 0}\in \wh{\Cba{1}}_\mf{N}$ and $\NNN^0\in \wh{\Cba{0}}_\mf{N}$ with $\mm_1(\NNN^{\geq 0}) = \NNN^0$, where 
	\[
	[\NNN^0] = (-1)^{\dim L} [\underline{\star}]\quad \textup{ in }\HBa/\ms{F}^\hbar \HBa
	\]
	for a sufficiently small $\hbar>0$ so that $\hbar < E(\beta)$ for all $\beta \in \mf{G}$. Thus
	\begin{equation}
		\label{Eq: MainPfEq}
		[\mm_1\NNN^{\geq 0}] = (-1)^{\dim L} [\underline{\star}] \quad \textup{in } \HBa/\ms{F}^\hbar\HBa.
	\end{equation} 
	We are going to perform a sequence of modifications on $\NNN^{\geq 0}$ which increases its energy but still keeps the equation \eqref{Eq: MainPfEq}.

	Similar to proof of Lemma \ref{Lm: BoundingChain}, we sort $\mf{N}$ into $\mf{N}=\{\eta_1,\eta_2,\cdots\}$ where $E(\eta_1)<E(\eta_2)<\cdots$, and sort $E(\mf{N})\subset \R$ into $E(\mf{N}) = \{E^{(1)}<E^{(2)}<\cdots \}$ so that $E^{(K)}=0$ for some $K\in\Z_{>0}$ (such $K$ exists by Definition \ref{Defn: module}; also $K$ is necessarily larger than 1 since otherwise $\mm_1\NNN^{\geq 0}$ has energy $>0$ whereas $\LL^0$ doesn't).
	Decompose $\NNN^{\geq 0}$ into \begin{align*}
		\NNN^{\geq 0} = \sum_{j=1}^\infty (\NNN^{\geq 0})^{(j)},\quad 
		\textup{ where } (\NNN^{\geq 0})^{(j)}= \sum_{\substack{\eta\in \mf{N} \\ E(\eta) = E^{(j)}}}\NNN^{\geq 0}(\eta),\quad \NNN^{\geq 0}(\eta)\in \CBa(\eta). 
	\end{align*}
	Equation \eqref{Eq: MainPfEq} becomes $\partial (\NNN^{\geq 0})^{(1)}=0$ in $\Cba{1}/\ms{F}^{E^{(1)}} \Cba{1}$, 
	and therefore each $\NNN^{\geq 0}(\eta)$ where $E(\eta)=E^{(1)}$ represents a homology class in $\Hba{1}(\eta)\cong H_1(\Omega_\star L(\eta))$, which is 0 by Lemma \ref{Lm: BasedHlgy}. 
	Thus there exist $\chi_\eta\in \Cba{2}(\eta)$ such that $\partial \chi_\eta + \NNN^{\geq 0}(\eta)=0$ in $\Cba{1}(\eta)$. 
	Then the element
	\[ (\NNN^{\geq 0})_{1}:=\NNN^{\geq 0} + \sum_{\substack{\eta \in \mf{N}\\ E(\eta) = E^{(1)}}} \mm_1(\chi_\eta) \in \ms{F}^{E^{(1)}}\Cba{1}  \]
	satisfies $[\mm_1(\NNN^{\geq 0})_1] = (-1)^{\dim L} [\underline{\star}]$ in $\HBa/\ms{F}^\hbar\HBa$ still and has energy strictly larger than $E^{(1)}$.  
	Continue this procedure until we obtain $(\NNN^{\geq 0})_{K-1}$, which satisfies $[\mm_1(\NNN^{\geq 0})_{K-1}] = (-1)^{\dim L} [\underline{\star}]$ in $\HBa/\ms{F}^\hbar\HBa$ with $(\NNN^{\geq 0})_{K-1}$ consisting only of elements of energy $\geq 0$. But then the energy of $\mm_1(\NNN^{\geq 0})_{K-1}$ is strictly larger than $\hbar$ whereas $(-1)^{\dim L} [\underline{\star}]$ has energy 0.
	This is a contradiction because $(-1)^{\dim L} [\underline{\star}]\in \HBa/\ms{F}^\hbar \HBa$ is non-zero in homology.
\end{proof}

\begin{remark}
	Alternatively, using homological perturbation lemma (in the filtered context, as done in e.g. Chapter 5 in \cite{FOOO1} or Section 2 in \cite{Irie2}), we can also show that there is an (uncurved) $\mathfrak{G}$-gapped, filtered $A_\infty$-structure on $\wh{H_*^\Omega}$, where $\mm_1$ has degree $-1$, which is quasi-isomorphic to $\wh{\CBa}$ and thus acyclic. This can also be used to conclude Theorem \ref{Thm: Main}. We do not spell out the detail here.
	As an another alternative, one can also use a spectral sequence argument after choosing a bounding chain.
\end{remark}

\subsection{Examples}\label{Subsc: Examples}
\begin{lemma}
	The condition $\pi_2 L=0$ is preserved under taking (finite) connected sums if $\dim L\geq 4$, or products in any dimension.
\end{lemma}
\begin{proof}
	The statement for products is straightforward. For connected sums, we use 
	\begin{align*}
		\pi_2(L_1\# L_2) \cong \pi_2(\wt{L_1 \# L_2}) \cong H_2(\wt{L_1\# L_2};\Z)
	\end{align*}
	and apply Mayer-Vietoris to the universal cover $\wt{L_1\# L_2}$.
\end{proof}

\begin{corollary}
	If $\dim L\geq 4$ and $L$ is an arbitrary (finite) connected sums or products of
	\begin{enumerate}
		\item Aspherical manifolds, i.e. $K(\pi,1)$'s;
		\item Spherical manifolds, i.e. $S^n/\Gamma$ where $\Gamma\subset SO_{n+1}$ is a finite subgroup acting freely on $S^n$ by rotations;
		\item Compact Lie groups,
	\end{enumerate}
	then, if $L$ is spin, any Lagrangian embedding of $L$ into $\C^n$ has non-vanishing Maslov class.
\end{corollary}

\bigskip

\appendix

\section{Verification of various identities in open-closed string topology}
\label{Sec: OCchainSigns}

In section \ref{Subsc: Ori}, we fix conventions on orientations of direct and fibre products of manifolds which we use in section \ref{Subsc: SignOper} to verify properties of various string topology operations on de Rham chains.
In section, \ref{Subsc: ConventionDGalg} we fix signs for various dg algebras, and use them in section \ref{Subsc: dgAlgebraSignApp} to verify properties of various string topology dg algebras and homomorphisms between them.

\subsection{Conventions on orientations}
\label{Subsc: Ori}
We follow the signs from \cite{Irie2} section 4.2 by default throughout the paper.
Here we briefly review the conventions on orientations of direct and fibre products of manifolds, and point out some consequence on signs that we shall use afterwards.

For direct products, let $M_1$ and $M_2$ be oriented manifolds; then we orient $M_1\times M_2$ by \[ T(M_1\times M_2) \cong TM_1\oplus TM_2. \]

For fibre products, let $M_1,M_2, X$ be oriented manifolds and $\pi_i\colon M_i\to X$ $(i=1,2)$ be smooth maps, where $\pi_2\colon M_2\to X$ is a submersion. Then the fibre product $M_1\fiberprod{\pi_1}{\pi_2}M_2$ can be formed. 
We orient $\ker(d\pi_2)\subset TM_2$ such that 
\[
	TM_2\cong TX\oplus\ker (d\pi_2)
\]as oriented vector bundles,
and orient $T(M_1\fiberprod{\pi_1}{\pi_2}M_2)$ as
\[
	T(M_1\fiberprod{\pi_1}{\pi_2} M_2) \cong TM_1\oplus \ker(d\pi_2).
\]

\begin{remarks}\label{Rmk: Ori}
In verifying various signs below, we will need to compare orientations in the following three situations:	
\begin{enumerate}
	\item If we have $M_1,M_2,M_3$ oriented smooth manifolds and 
\[ \pi_1^1,\pi_1^2\colon M_1\to X, \quad \pi_2\colon M_2\to X, \quad \pi_3\colon M_3\to X\]
 smooth maps, with $\pi_2,\pi_3$ submersions, then
\[
	T((M_1\fiberprod{\pi^1_1}{\pi_2} M_2)\fiberprod{\pi^2_1}{\pi_3} M_3)\cong T(M_1\fiberprod{\pi_1^1}{\pi_2}M_2)\oplus\ker(d\pi_3)\cong TM_1\oplus \ker (d\pi_2)\oplus \ker (d\pi_3)
\]
and
\[
	T((M_1\fiberprod{\pi_1^2}{\pi_3}M_3)\fiberprod{\pi_1^1}{\pi_2}M_2)\cong T(M_1\fiberprod{\pi_1^2}{\pi_3}M_3)\oplus \ker(d\pi_2)\cong TM_1\oplus\ker(d\pi_3)\oplus\ker(d\pi_2).
\]
Thus the orientation of these two fibre products differ by $(-1)^{(\dim M_2-\dim X)(\dim M_3-\dim X)}$.

\item 
Similarly, given 
\[
	\pi_1\colon M_1\to X, \quad \pi_2^1,\pi_2^2\colon M_2\to X, \quad \pi_3\colon M_3\to X
\]
 with $\pi_2^1,\pi_3$ submersions, 
then \begin{align*}
 	T((M_1\fiberprod{\pi_1}{\pi_2^1}M_2)\fiberprod{\pi_2^2}{\pi_3} M_3)\cong T(M_1\fiberprod{\pi_1}{\pi_2^1}M_2)\oplus \ker (d\pi_3) \cong TM_1\oplus \ker (d\pi_2^1)\oplus \ker (d\pi_3)
 \end{align*}
 and 
 \begin{align*}
 	T(M_1\fiberprod{\pi_1}{\pi_2^1}(M_2\fiberprod{\pi_2^2}{\pi_3}M_3))&\cong TM_1 \oplus\ker(d\pi_2^1\colon T(M_2\fiberprod{\pi_2^2}{\pi_3}M_3)\to TX)\\
 	&\cong TM_1\oplus \ker (TM_2\oplus \ker (d\pi_3)\xrightarrow{d\pi_2\oplus 0} TX)\\
 	&\cong TM_1\oplus\ker(d\pi_2^1)\oplus \ker(d\pi_3).
 \end{align*}
 Therefore these two fibre products have the same orientation.
 
 \item If we have \[  \pi_1^1,\pi_1^2\colon M_1\to X,\quad \pi_2\colon M_2\to X \]where $\pi_1^1,\pi_2$ are submersions and $\star\in X$ is a point, then 
 \begin{align*}
 	T((\pi_1^1)^{-1}(\star)\fiberprod{\pi_1^2}{\pi_2} M_2) \cong T((\pi_1^1)^{-1}(\star)) \oplus \ker (d\pi_2) \cong \ker (d\pi_1^1) \oplus \ker (d\pi_2)
 \end{align*}
 and 
 \begin{align*}
 	(M_1\fiberprod{\pi_1^2}{\pi_2} M_2\xrightarrow{\pi_1^1\circ \pr_1} X)^{-1}(\star) \cong \ker( TM_1\oplus \ker (d\pi_2) \xrightarrow{d\pi_1^1 \oplus 0} TX)\cong \ker (d\pi_1^1) \oplus d\pi_2.
 \end{align*}
 Therefore these two fibre products have the same orientation.
 \end{enumerate}
\end{remarks}

\subsection{Signs for string topology operations}\label{Subsc: SignOper}

\subsubsection{Closed string}\label{Subsubsc: SignClosedString}
Recall that given $x\in C_{\dim L +d_1}^\dR( \mathscr{L}^{k_1+1}(a_1))$, $y\in C_{\dim L + d_2}^\dR(\mathscr{L}^{k_2+1} (a_2))$, with $x=[(U_1,\varphi_1,\omega_1)], y=[(U_2,\varphi_2,\omega_2)],$
\[
	x\circ_i^\mathscr{L} y:= (-1)^{d_1|\omega_2|}[(U_1\fiberprod{i}{0}U_2, \varphi_1\circ_i^\mathscr{L} \varphi_2, \omega_1\times \omega_2)].
\]
\begin{lemma}\label{Lm: dRchainClosedStringApp}
	\begin{enumerate}
		\item (Leibniz rule)  If $x\in C_{\dim L +d}^\dR( \mathscr{L}^{k_1+1}(a_1))$, $y\in C_{\dim L + d'}^\dR(\mathscr{L}^{k_2+1} (a_2))$,\[\partial^\dR(x\circ_i y)  = (\partial^\dR x) \circ_i y + (-1)^{d} x\circ_i (\partial^\dR y);\]
		\item (Associativity) If $x_i\in C^\dR_{\dim L+d_i}(\mathscr{L}^{k_i+1}(a_i))$ $(i=1,2,3)$,
		\begin{align*}
			(x_1\circ_{i_1} x_2)\circ_{k_2+i_2-1} x_3&= (-1)^{d_2d_3}(x_1\circ_{i_2} x_3)\circ_{i_1}x_2 \quad (1\leq i_1<i_2\leq k_1);\\
			(x_1\circ_{i_1} x_2)\circ_{i_1+i_2-1}x_3 &= x_1\circ_{i_1} (x_2\circ_{i_2} x_3) \quad (1\leq i_1\leq k_1, 1\leq i_2\leq k_2).
		\end{align*}
	\end{enumerate}
\end{lemma}
\begin{proof}
		Recall that given a de Rham chain $[(U,\varphi,\omega)]$, the de Rham differential on it is given by
\[
	\partial^\dR [(U,\varphi,\omega)]=(-1)^{|\omega|+1}[(U,\varphi,d\omega)],
\]
	\begin{enumerate}
		\item Note that $d = \dim U_1-|\omega_1|-\dim L$, so \begin{align*}
			&\partial^\dR(x\circ_iy)\\= &\partial^\dR\left((-1)^{d|\omega_2|} (U_1\fiberprod{i}{0} U_2,\conc\circ (\varphi_1\times \varphi_2), \omega_1\times \omega_2) \right)\\
			=& (-1)^{|\omega_1|+|\omega_2|+1}(-1)^{d|\omega_2|}(U_1\fiberprod{i}{0}U_2, \conc\circ (\varphi_1\times \varphi_2), d\omega_1\times \omega_2)\\
			&+ (-1)^{|\omega_1|+|\omega_2|+1}(-1)^{d|\omega_2|}(-1)^{|\omega_1|}(U_1\fiberprod{i}{0}U_2, \conc \circ (\varphi_1\times \varphi_2), \omega_1\times d\omega_2)
		\end{align*}
		Now 
		\begin{align*}
			(\partial^\dR x)\circ_i y =(-1)^{(d+1)|\omega_2|} (-1)^{|\omega_1|+1}(U_1\fiberprod{i}{0}U_2, \conc\circ (\varphi_1\times \varphi_2), d\omega_1\times \omega_2)
		\end{align*}
		and \begin{align*}
			x\circ_i (\partial^\dR y)= (-1)^{d(|\omega_2|+1)}(-1)^{|\omega_2|+1}(U_1\fiberprod{i}{0}U_2,\conc\circ (\varphi_1\times \varphi_2), \omega_1\times d\omega_2).
		\end{align*}
		Then comparing the exponents checks the signs of (1).
	\item Write $x_i=[(U_i, \varphi_i,\omega_i)]$ for $i=1,2,3$. We have $d_i = \dim U_i -|\omega_i| -\dim L$.
	Notice that $\deg(x\circ_i y) -\dim L=(\deg x-\dim L)+(\deg y - \dim L)$.
	
	For the first identity, 
	\begin{align*}
		&(x_1\circ_{i_1}x_2)\circ_{k_2+i_2-1} x_3 \\ =& (-1)^{d_1|\omega_2|}(U_1\fiberprod{i_1}{0} U_2, \conc \circ(\varphi_1\times \varphi_2), \omega_1\times \omega_2 )\circ_{k_2+i_2-1}x_3\\
		=&(-1)^{d_1|\omega_2|}(-1)^{(d_1+d_2)|\omega_3|}((U_1\fiberprod{i_1}{0}U_2)\fiberprod{k_2+i_2-1}{0}U_3 , \conc\circ((\conc\circ(\varphi_1\times \varphi_2))\times \varphi_3), \omega_1\times \omega_2\times \omega_3 )
	\end{align*}
	and 
	\begin{align*}
		&(x_1\circ_{i_2}x_3)\circ_{i_1} x_2\\
		=& (-1)^{d_1|\omega_3|}(U_1\fiberprod{i_2}{0}U_3,\conc \circ( \varphi_1\times \varphi_3),  \omega_1\times \omega_3)\circ_{i_1} x_2\\
		=& (-1)^{d_1|\omega_3|}(-1)^{(d_1+d_3)|\omega_2|}((U_1\fiberprod{i_2}{0}U_3)\fiberprod{i_1}{0}U_2, \conc\circ((\conc\circ(\varphi_1\times\varphi_3))\times \varphi_2), \omega_1\times \omega_3\times \omega_2)\\
		=& (-1)^{d_1|\omega_3|}(-1)^{(d_1+d_3)|\omega_2|}(-1)^{|\omega_2||\omega_3|}((U_1\fiberprod{i_2}{0}U_3)\fiberprod{i_1}{0}U_2, \conc\circ((\conc\circ(\varphi_1\times\varphi_3))\times \varphi_2), \omega_1\times \omega_2\times \omega_3).
	\end{align*}
	Notice that the difference between the exponents of $(-1)$ in the front of the two expressions is 
	\begin{align*}
		&(d_1|\omega_2|+(d_1+d_2)|\omega_3|)- (d_1|\omega_3|+(d_1+d_3)|\omega_2|+|\omega_2||\omega_3| ) \\
		\equiv & d_2|\omega_3|+ d_3|\omega_2|+|\omega_2||\omega_3| \equiv (d_2+|\omega_2|)(d_3+|\omega_3|)+d_2d_3\\
		\equiv &  (\dim U_2 -\dim L)(\dim U_3-\dim L) + d_2d_3\mod 2.
	\end{align*}
	Moreover from Section \ref{Subsc: Ori} we see that the orientation of $(U_1\fiberprod{i_2}{0}U_3)\fiberprod{i_1}{0}U_2$ and $(U_1\fiberprod{i_2}{0}U_3)\fiberprod{i_1}{0}U_2$ differs by $(-1)^{(\dim U_2-\dim L)(\dim U_3-\dim L)}$, so from Remark \ref{Rmk: SignOri} we see that the first associativity identity is correct.
	
	For the second identity, 
	\begin{align*}
		&(x_1\circ_{i_1} x_2)\circ_{i_1+i_2-1}x_3 \\
		=& (-1)^{d_1|\omega_2|}(-1)^{(d_1+d_2)|\omega_3|}((U_1\fiberprod{i_1}{0}U_2 )\fiberprod{i_1+i_2-1}{0}U_3,\conc\circ((\conc\circ (\varphi_1\times \varphi_2))\times \varphi_3) , \omega_1\times \omega_2\times \omega_3)
	\end{align*}
	and 
	\begin{align*}
		&x_1\circ_{i_1} (x_2\circ_{i_2}x_3) \\
		=& (-1)^{d_2|\omega_3|}x_1\circ_{i_1}(U_2\fiberprod{i_2}{0}U_3, \conc\circ(\varphi_2\times \varphi_3), \omega_2\times \omega_3)\\
		=& (-1)^{d_2|\omega_3|}(-1)^{d_1(|\omega_2|+|\omega_3|)}(U_1\fiberprod{i_1}{0}(U_2\fiberprod{i_2}{0} U_3), \conc\circ((\conc\circ(\varphi_2\times \varphi_3))\times \varphi_1, \omega_1\times \omega_2\times \omega_3).
	\end{align*}
	By the discussion in Section \ref{Subsc: Ori},  the two domains of these two de Rham chains have the same orientation, and comparing the exponents of $(-1)$ shows that these two de Rham chains are the same. This proves the second associativity identity.
	\end{enumerate}
\end{proof}

\subsubsection{Open string}\label{Subsubsc: SignOpenString}  
Recall that given $\alpha \in C_{d_1}^\dR(\Omega^{k_1+1}_\star(a_1))$, $\beta\in C_{d_2}^\dR(\Omega^{k_2+1}_\star(a_2))$, with $\alpha=[(V_1,\psi_1,\eta_1)]$, $\beta = [(V_2,\psi_2,\eta_2)]$ (so that $d_1=\deg\alpha=\dim V_1-|\eta_1|$ and similarly for $\beta$),
\[
	\alpha\bullet \beta:= (-1)^{d_1|\eta_2|} [(V_1\times V_2, \psi_1\bullet \psi_2, \eta_1\times \eta_2)].
\]

\begin{lemma}\label{Lm: dRchainBulletProdApp}
	\begin{enumerate}
		\item (Leibniz rules) If $\alpha\in C^\dR_{d_1}(\Omega^{k_1+1}_\star (a_1)), \beta \in C^\dR_{d_2}(\Omega_\star^{k_2+1}(a_2))$, \[\partial^\dR( \alpha\bullet \beta )= (\partial^\dR \alpha)\bullet \beta + (-1)^{d_1}\alpha \bullet (\partial^\dR \beta);\]
		\item (Associativity of $\bullet$) If $\alpha_i\in C^\dR_*(\Omega_\star^{k_i+1}(a_i))$ $(i=1,2,3)$, \[ (\alpha_1 \bullet \alpha_2)\bullet \alpha_3 = \alpha_1\bullet (\alpha_2\bullet \alpha_3).\]
		\end{enumerate}
\end{lemma}
\begin{proof}
	\begin{enumerate}
		\item Given $\alpha = [(V_1,\psi_1,\eta_1)]$ and $\beta=[(V_2,\psi_2,\eta_2)]$, 
		\begin{align*}
			\partial^\dR(\alpha\bullet \beta)&= \partial^\dR \left((-1)^{(\deg \alpha_1) |\eta_2|} [( V_1\times V_2, \psi_1 \bullet \psi_2 , \eta_1\times \eta_2 )]\right)\\
			&= (-1)^{|\eta_1|+|\eta_2|+1}(-1)^{(\deg \alpha)|\eta_2|}[(V_1\times V_2, \psi_1 \bullet \psi_2, d\eta_1\times \eta_2)] \\
				&\quad + (-1)^{|\eta_1|+|\eta_2|+1}(-1)^{(\deg \alpha) |\eta_2|}(-1)^{|\eta_1|}[(V_1\times V_2, \psi_1 \bullet \psi_2, \eta_1\times d\eta _2)].
		\end{align*} 
		On the other hand
		\begin{align*}
			(\partial^\dR \alpha)\bullet \beta 
			= (-1)^{|\eta_1|+1}(-1)^{(\deg \alpha - 1)|\eta_2|}[(V_1\times V_2, \psi_1 \bullet \psi_2, d\eta_1 \times \eta_2)]
		\end{align*}
		and 
		\begin{align*}
			 \alpha \bullet (\partial^\dR \beta) 
			= (-1)^{|\eta_2|+1} (-1)^{ (\deg \alpha)(|\eta_2|+1) }[(V_1\times V_2 , \psi_1 \bullet \psi_2, \eta_1\times d\eta_2)].
		\end{align*}

		\item This follows from the fact that the concatenation product $*$ is associative, i.e. 
		\[
\begin{tikzcd}
\Omega^{k_1+1}_\star(a_1)\times \Omega^{k_2+1}_\star (a_2) \times \Omega^{k_3+1}_\star(a_3) \arrow[r] \arrow[d] & \Omega^{k_1+k_2+1}_\star(a_1+a_2)\times \Omega^{k_3+1}_\star(a_3) \arrow[d] \\
\Omega_\star^{k_1+1}(a_1) \times \Omega_\star^{k_2+k_3+1} (a_2+a_3) \arrow[r]                             & \Omega_\star^{k_1+k_2+k_3+1}(a_1+a_2+a_3)                            
\end{tikzcd}
		\]
		strictly commutes.
		Also the sign that appears on the left-hand side of the expression in (2)  is $(-1)^{(\deg\alpha_1)|\eta_2|}(-1)^{(\deg\alpha_1+\deg\alpha_2)|\eta_3|}$ and the sign on the right hand side is $(-1)^{(\deg \alpha_1)(|\eta_2|+|\eta_3|)}(-1)^{(\deg \alpha_2) |\eta_3|}$.
	\end{enumerate}
\end{proof}

\subsubsection{Open-closed string}\label{Subsubsc: OCSignAppendix}
Recall that given $\alpha \in C_{d_1}^\dR(\Omega^{k_1+1}_\star(a_1))$, $x\in C^\dR_{\dim L+d_2}(\ms{L}^{k_2+1}(a_2))$, with $\alpha = [(V,\psi,\eta)], x= [(U,\varphi,\omega)]$, 
\[
	\alpha \circ_i^\Omega x:= (-1)^{d_1|\omega|}[(V\fiberprod{i}{0}U, \psi \circ_i^\Omega \varphi, \eta \times \omega)].
\]

\begin{lemma}\label{Lm: dRchainOpenClosedApp}
	\begin{enumerate}
		\item (Leibniz rule) If $\alpha\in C^\dR_{d_1}(\Omega_\star^{k_1+1}(a_1))$, $x\in C^\dR_{\dim L + d_2}(\mathscr{L}^{k_2+1}(a_2))$, 
		\[\partial^\dR(\alpha\circ_i^\Omega x) =(\partial^\dR \alpha)\circ_i^\Omega x + (-1)^{d_1} \alpha \circ_i^\Omega (\partial^\dR x);\]
		\item (Associativity) If $\alpha\in C_{d_1}^\dR(\Omega_\star^{k_1+1}(a_1))$, $x\in C_{\dim L + d_2}^\dR (\mathscr{L}^{k_2+1}(a_2))$, $y\in C^\dR_{\dim L + d_3}(\mathscr{L}^{k_3+1}(a_3))$,\begin{align*}
			(\alpha\circ_{i_1}^\Omega x)\circ^\Omega_{k_2+i_2-1} y &= (-1)^{d_2d_3} (\alpha\circ_{i_2}^\Omega y)\circ_{i_1}^\Omega x \quad (1 \leq i_1 < i_2 \leq k_1);\\
			(\alpha \circ_{i_1}^\Omega x)\circ_{i_1+i_2-1}^\Omega y&= \alpha\circ_{i_1}^\Omega(x\circ_{i_2}^\mathscr{L} y)  \quad  (1 \leq i_1 \leq k_1, 1 \leq i_2 \leq k_2).
		\end{align*}
		\item (Compatibility with $\bullet$) If $\alpha\in C_{d_1}^\dR(\Omega_\star^{k_1+1}(a_1))$, $\beta \in C_{d_2}^\dR(\Omega_\star^{k_2+1}(a_2))$, $x\in C_{\dim L+d_3}^\dR(\mathscr{L}^{k_3+1}(a_3))$, \begin{align*}
			(\alpha \bullet \beta) \circ^\Omega_i x&= (-1)^{d_2d_3} (\alpha\circ^\Omega_i x) \bullet \beta \quad (1\leq i \leq k_1) \\
			(\alpha \bullet \beta) \circ_i^\Omega x&=  \alpha \bullet (\beta \circ_{i-k_1}^\Omega x)\quad (k_1+1\leq i \leq k_1+k_2)
		\end{align*}
	\end{enumerate}
\end{lemma}

\begin{proof}
	\begin{enumerate}
		\item Suppose that $\alpha=[(V,\psi, \eta)], x= [(U,\varphi,\omega)]$, then
		\begin{align*}
			\partial^\dR(\alpha\circ_i x)&= \partial ^\dR\left((-1)^{(\deg \alpha) |\omega|} [(V\fiberprod{i}{0} U, \concOmega\circ (\psi \times \varphi), \eta\times \omega)]\right)\\
			&= (-1)^{|\eta|+|\omega|+1}(-1)^{(\deg\alpha ) |\omega|}[(V\fiberprod{i}{0} U ,\concOmega\circ(\psi \times \varphi), d\eta\times \omega)]\\
			&\quad + (-1)^{|\eta|+|\omega|+1}(-1)^{(\deg\alpha)|\omega|}(-1)^{|\eta|}[(V\fiberprod{i}{0} U,\concOmega\circ(\psi\times \varphi) , \eta\times d\omega)].
		\end{align*}
		On the other hand,
		\begin{align*}
			(\partial^\dR \alpha)\circ_i x &= (-1)^{((\deg \alpha) +1)|\omega|}(-1)^{|\eta|+1}[ (V\fiberprod{i}{0} U, \concOmega\circ (\psi\times \varphi), d\eta \times \omega ]
		\end{align*}
		and 
		\begin{align*}
			\alpha\circ_i (\partial^\dR x)= (-1)^{(\deg \alpha)(|\omega|+1)}(-1)^{|\omega|+1}[ ( V\fiberprod{i}{0} U, \concOmega\circ(\psi \times \varphi), \eta \times d\omega)].
		\end{align*}
		\item Write 
		\begin{align*}
			\alpha =[(V,\psi,\eta)], \quad x=[(U_1,\varphi_1,\omega_1)], \quad y =[(U_2,\varphi_2,\omega_2)].
		\end{align*}
		We have that $d_1=\deg \alpha$ and $d_2=\deg x - \dim L, d_3=\deg y - \dim L$.
		
		For the first identity, 
		\begin{align*}
			&(\alpha\circ_{i_1}^\Omega x)\circ^\Omega_{k_2+i_2-1} y\\
			=& \left( (-1)^{d_1 |\omega_1|} [(V\fiberprod{i_1}{0} U_1, \psi \circ_{i_1}^\Omega \varphi_1 , \eta\times \omega_1)] \right) \circ_{k_2+i_2-1} y\\
			=& (-1)^{d_1|\omega_1|}(-1)^{(d_1+d_2) |\omega_2|}[((V\fiberprod{i_1}{0} U_1)\fiberprod{k_2+i_2-1}{0} U_2, (\psi\circ_{i_1}^\Omega \varphi_1)\circ_{k_2+k_2-1}^\Omega \varphi_2 , \eta\times \omega_1\times \omega_2)]
		\end{align*}
		and 
		\begin{align*}
			& (\alpha\circ_{i_2}^\Omega y)\circ^\Omega_{i_1} x\\
			=& (-1)^{d_1|\omega_2|}[( V\fiberprod{i_2}{0}U_2, \psi \circ_{i_2}^\Omega \varphi_2, \eta \times \omega_2 )]\circ_{i_1}^\Omega x\\
			=& (-1)^{d_1|\omega_2|}(-1)^{(d_1+d_3)|\omega_1|}[ ((V\fiberprod{i_2}{0} U_2)\fiberprod{i_1}{0}U_1, (\psi\circ_{i_2}^\Omega \varphi_2)\circ_{i_1}^\Omega \varphi_1, \eta\times \omega_2\times \omega_1) ].
		\end{align*}
		The orientation of the domains $(V\fiberprod{i}{0} U_1)\fiberprod{k_2+i_2-1}{0} U_2$ and $(V\fiberprod{i_2}{0} U_2)\fiberprod{i_1}{0}U_1$ of these two de Rham chains differ by $(-1)^{(\dim U_1 -\dim L) (\dim U_2 - \dim L)}$ and to turn the differential forms $\eta\times \omega_2\times \omega_1$ into $\eta\times \omega_1\times \omega_2$ requires a sign of $(-1)^{|\omega_1||\omega_2|}$, so the overall sign difference between the two de Rham chains is \begin{align*} &d_1|\omega_1|+(d_1+d_2)|\omega_2| - d_1 |\omega_2|-(d_1+d_3)|\omega_1|- (\dim U_1-\dim L)(\dim U_2 - \dim L) -|\omega_1||\omega_2| \\
		\equiv& d_2|\omega_2|  -d_3|\omega_1| -(\dim U_1-\dim L)(\dim U_2-\dim L) - |\omega_1||\omega_2|\\\equiv& d_2|\omega_2|-d_3|\omega_1| -(d_2+|\omega_1|)(d_3+|\omega_2|)-|\omega_1||\omega_2| \equiv d_2d_3\mod 2,
		\end{align*}using that $d_2=\dim U_1 -|\omega_1|-\dim L, d_3=\dim U_2-|\omega_2|-\dim L$.

		For the second identity, 
		\begin{align*}
			&(\alpha \circ_{i_1}^\Omega x)\circ_{i_1+i_2-1}^\Omega y\\
			=& (-1)^{d_1|\omega_1|}(-1)^{(d_1+d_2)|\omega_2|}[((V\fiberprod{i_1}{0}U_1)\fiberprod{i_1+i_2-1}{0}U_2, (\psi \circ^\Omega_{i_1} \varphi_1)\circ_{i_1+i_2+1}^\Omega \varphi_2, \eta\times \omega_1\times \omega_2)]
		\end{align*}
		and 
		\begin{align*}
			&\alpha\circ_{i_1}^\Omega(x\circ_{i_2}^\mathscr{L} y) \\
			=& (-1)^{d_2|\omega_2|} \alpha\circ_{i_1}^\Omega [( V_1\fiberprod{i_2}{0}V_2, \varphi_1\circ_{i_2}^\mathscr{L} \varphi_2, \omega_1\times \omega_2 )]\\
			=& (-1)^{d_2|\omega_2|}(-1)^{d_1(|\omega_1|+|\omega_2|)}[(U\fiberprod{i_1}{0}(V_1\fiberprod{i_2}{0}V_2), \psi\circ_{i_1}^\Omega (\varphi_1\circ_{i_2}^\mathscr{L} \varphi_2),\eta\times \omega_1\times \omega_2)].
		\end{align*}
		The signs cancel.
		
		\item Write \begin{align*}
			\alpha = [(V_1, \psi_1,\eta_1)], \quad \beta = [(V_2,\psi_2,\eta_2)], \quad x=[(U,\varphi,\omega)].
		\end{align*}
		For the first identity, 
		\begin{align*}
			&(\alpha\bullet \beta)\circ_i^\Omega x \\
			=& (-1)^{d_\alpha|\eta_2|} (-1)^{(d_\alpha+d_\beta)|\omega|}[((V_1\times V_2)\fiberprod{i}{0}U, (\psi_1\bullet \psi_2)\circ_i^\Omega \varphi, \eta_1\times \eta_2 \times \omega)]
		\end{align*}
		and 
		\begin{align*}
			&(\alpha\circ_i^\Omega x)\bullet \beta \\
			=&(-1)^{d_\alpha|\omega|}(-1)^{(d_\alpha +d_x)|\eta_2|}[((V_1\fiberprod{i}{0}U)\times V_2, (\psi_1\circ_i^\Omega \varphi)\bullet \psi_2, \eta_1\times \omega\times \eta_2)].
		\end{align*}
		The orientation difference between the two domains is $(-1)^{\dim V_2(\dim U-\dim L)}$ and the sign difference between the two differential forms is $|\eta_2||\omega|$, so the overall sign difference is 
		\begin{align*}
			&d_\alpha |\eta_2|+(d_\alpha+d_\beta)|\omega| - d_\alpha|\omega| -(d_\alpha + d_x)|\eta_2| -\dim V_2(\dim U - \dim L) - |\eta_2||\omega|\\
			\equiv & d_\beta|\omega| - d_x |\eta_2| - \dim V_2(\dim U -\dim L)-|\eta_2||\omega|\\
			\equiv & d_\beta |\omega| -d_x|\eta_2| - (d_\beta - |\eta_2|)(d_x - |\omega|)-|\eta_2||\omega| \equiv d_\beta d_x.
		\end{align*}
		For the second identity, 
		\begin{align*}
			& \alpha \bullet (\beta\circ_{i-k_1}^\Omega x)\\
			\equiv & (-1)^{d_\alpha(|\eta_2|+|\omega|)}(-1)^{d_\beta |\omega|}[(V_1\times (V_2\fiberprod{i-k_1}{0} U), \psi_1\bullet (\psi_2\circ_{i-k_1}\varphi), \eta_1\times \eta_2\times \omega)].
		\end{align*}
		The sign cancels with the sign in front of $(\alpha\bullet \beta)\circ_i^\Omega x$.
	\end{enumerate}
\end{proof}

\subsubsection{The anomaly map}\label{Subsubsc: SignAnomaly}
Recall that given $x\in C^\dR_*(\mathscr{L}^{k+1}(a))$ represented by the de Rham chain $[(U\xrightarrow{\varphi} \mathscr{L}^{k+1}(a), \omega)]$, we defined 
\[
	\mathfrak{o}(x) := (-1)^{(\deg x)+1} [((\ev_0\circ \varphi)^{-1}(\star)\xrightarrow{\varphi} \Omega_\star^{k+1}(a);\omega)]\in C_*^\dR(\Omega^{k+1}_\star(a)).
\]

\begin{lemma}\label{Lm: AnomalyApp}
	\begin{enumerate}
		\item (Compatibility with $\circ^\mathscr{L}$ and $\circ^\Omega$) If $x\in C_*^\dR(\mathscr{L}^{k_1+1}(a_1))$ and $y\in C^\dR_*(\mathscr{L}^{k_2+1}(a_2))$, then \[ \mathfrak{o}(x\circ_i^\mathscr{L} y) = (-1)^{\deg(y)-\dim L}\mathfrak{o}(x) \circ_i^\Omega y; \]
		\item (Compatibility with $\partial^\dR$)  If $x\in C_*^\dR(\mathscr{L}^{k+1}(a))$, then \[ \partial^\dR \mathfrak{o}(x) =-\mathfrak{o}(\partial^\dR x). \]
		\item (Compatibility with $\delta_i$) If $x\in C^\dR_*(\mathscr{L}^{k+1}(a))$, then for $i=0,\dots, k+1$, \[ \mathfrak{o}((\delta_i)_*(x)) = (\delta_i)_*(\mathfrak{o}(x)). \](See section \ref{Subsc: StateSpace}  for the definition of $\delta_i$.)
	\end{enumerate}
\end{lemma}

\begin{proof}
	\begin{enumerate}
		\item Write \[ x=[(U_1\xrightarrow{\varphi_1}\mathscr{L}^{k_1+1}(a_1);\omega_1)] \in C_{d_1+\dim L}^\dR(\mathscr{L}^{k_1+1}(a_1))\] and \[ y = [(U_2\xrightarrow{\varphi_2}\mathscr{L}^{k_2+1}(a_2);\omega_2)] \in C^\dR_{d_2+\dim L}(\mathscr{L}^{k_2+1}(a_2)) . \]
		Then 
		\begin{align*}
			&\mathfrak{o}(x\circ_i^\mathscr{L} y) \\=& (-1)^{d_1 |\omega_2|}\mathfrak{o}([U_1\fiberprod{i}{0} U_2\xrightarrow{\varphi_1\circ_i \varphi_2} \mathscr{L}^{k_1+k_2}(a_1+a_2);  \omega_1\times \omega_2])\\
			=& (-1)^{d_1|\omega_2|+\deg(x\circ_i^\mathscr{L} y)+1}[( (\ev_0\circ(\varphi_1\circ_i\varphi_2))^{-1}(\star) \xrightarrow{\varphi_1\circ_i \varphi_2} \Omega_\star^{k_1+k_2}(a_1+a_2); \omega_1\times \omega_2 ) ],
		\end{align*}
		and
		\begin{align*}
			&\mathfrak{o}(x)\circ_i^\Omega y\\=& (-1)^{(\deg x)+1}[( (\ev_0\circ\varphi_1)^{-1}(\star)\xrightarrow{\varphi_1} \Omega_\star^{k_1+1}(a_1) ;\omega_1  )]\circ_i^\Omega y\\
			=& (-1)^{(\deg x)+1+d_1|\omega_2|}[( \ev_0\circ\varphi_1)^{-1}(\star) \fiberprod{i}{0} U_2 \xrightarrow{\varphi_1\circ_i \varphi_2} \Omega_\star^{k_1+k_2}(a_1+a_2)  ; \omega_1\times \omega_2)]
		\end{align*}
		By Remark \ref{Rmk: Ori} (3) in section \ref{Subsc: Ori}, the two domains $(\ev_0\circ(\varphi_1\circ_i\varphi_2))^{-1}(\star)$ and  $( \ev_0\circ\varphi_1)^{-1}(\star) \fiberprod{i}{0} U_2$ have the same orientation. The claim  now follows from the fact that $\deg(x\circ_i y) = \deg(x)+\deg (y)-\dim L$.
		\item Write $x=[(U\xrightarrow{\varphi} \mathscr{L}^{k+1}(a);\omega)]$. We have that 
		\begin{align*}
			\partial^\dR \mathfrak{o}(x) = (-1)^{|\omega|+1+\deg x + 1}[((\ev_0\circ \varphi)^{-1}(\star)\xrightarrow{\varphi} \Omega_\star^{k+1}(a);d\omega)]
		\end{align*}
		and  \begin{align*}
			\mathfrak{o}(\partial^\dR x)& = \mathfrak{o}\big((-1)^{|\omega|+1} [(U\xrightarrow{\varphi} \mathscr{L}^{k+1}(a); d\omega)]\big) \\&= (-1)^{|\omega|+1+(\deg x+1)+1}[((\ev_0\circ \varphi)^{-1}(\star)\xrightarrow{\varphi} \Omega_\star^{k+1}(a); d\omega)].
		\end{align*}
		\item Write $x=[(U\xrightarrow{\varphi} \mathscr{L}^{k+1}(a);\omega)]$. To distinguish $\delta_i$ in the context of free and based loop spaces, we temporarily use the notations \[ \mathscr{L}^{k+1}(a)\xrightarrow{\delta_i^\mathscr{L}} \mathscr{L}^{k+2}(a) \quad \textup{and}\quad \Omega_\star^{k+1}(a) \xrightarrow{\delta_i^{\Omega_\star}} \Omega_\star^{k+2}(a). \] 
		We have that \begin{align*}
			\mathfrak{o}((\delta^\mathscr{L}_i)_*(x)) &= \mathfrak{o}[(U\xrightarrow{\delta^\mathscr{L}_i\circ\varphi} \mathscr{L}^{k+2}(a);\omega)]\\
			&= (-1)^{\deg x+1} [((\ev_0\circ \delta^\mathscr{L}_i \circ \varphi)^{-1}(\star) \xrightarrow{\delta^\mathscr{L}_i\circ \varphi} \Omega_\star^{k+2}(a);\omega)]
		\end{align*}
		and 
		\begin{align*}
			(\delta^{\Omega_\star}_i)_*(\mathfrak{o}(x)) = (-1)^{\deg x+1}[(\ev_0\circ \varphi)^{-1}(\star) \xrightarrow{ \delta^{\Omega_\star}_i\circ\varphi}\Omega_\star^{k+2}(a);\omega)].
		\end{align*}
		Now the two maps $\mathscr{L}^{k+1}\xrightarrow{\delta^{\mathscr{L}}_i} \mathscr{L}^{k+2} \xrightarrow{\ev_0} L$ and $\mathscr{L}^{k+1}\xrightarrow{\ev_0} L$ are identical, so the domains of the two de Rham chains are the same. The two maps from the domains to $\Omega_\star^{k+2}(a)$ are the same because the two maps $\Omega_\star^{k+1}(a)\xrightarrow{\delta_i^{\Omega_\star}} \Omega_\star^{k+2}(a)$ and $\mathscr{L}^{k+1}(a)\xrightarrow{\delta_i^\mathscr{L}}\mathscr{L}^{k+2}(a)$ are compatible under $\Omega_\star^{k+1}(a)\subset \mathscr{L}^{k+1}(a)$.
	\end{enumerate}
\end{proof}

\subsection{Conventions on dg algebras}\label{Subsc: ConventionDGalg}
We fix signs for various dg algebras used in this paper, following Chapter 2 of \cite{FukayaDefo} (adjusting for the difference that the degrees we use is homological whereas they are cohomological in \cite{FukayaDefo}) and are consistent with \cite{Irie2}.

Given a graded $\K$-module $C$ and a homogeneous element $x\in C_k$, we write $|x|=k$.

\begin{definition}\label{Defn: dgAa}
	A \emph{differential graded associative algebra (dg associative algebra)} is a graded $\K$-module $C$ together with maps 
	\[
		\partial \colon C_*\to C_{*-1}
	\]
	of degree $-1$ and 
	\[
		\bullet\colon C_*\otimes C_*\to C_*
	\]
	of degree $0$, such that 
	\begin{align}
		& \partial(x\bullet y) = (\partial x) \bullet y + (-1)^{|x|} x \bullet (\partial y); \label{Eq: dgaLeibniz}\\
		& (x\bullet y)\bullet z = x\bullet (y\bullet z).\label{Eq: dgaAssoc}
	\end{align}
\end{definition}

\begin{definition}\label{Defn: dgLa}
	A \emph{differential graded Lie algebra (dg Lie algebra)} is a graded $\K$-module $C$ together with maps \[\partial\colon C_*\to C_{*-1}\] of degree $-1$ and \[[,]\colon C_*\otimes C_*\to C_*\]of degree $0$, such that 
	\begin{align}
		& \partial [x,y]= [\partial x, y]+(-1)^{|x|}[x,\partial y]; \label{Eq: dgLa Leib}\\
		& [x,y] = (-1)^{|x||y|+1}[y,x];\label{Eq: dgLa antisym} \\
		& [[x,y],z] + (-1)^{(|x|+|y|)|z|}[[z,x],y]+ (-1)^{(|y|+|z|)|x|}[[y,z],x]=0. \label{Eq: dgLa Jacobi}
	\end{align}
\end{definition}

\begin{definition}
	A \emph{homomorphism of dg Lie algebras} $(C_*,\partial,[,])$ and $(C'_*,\partial,[,])$  is a degree-0 linear map $f\colon C_*\to C'_*$ that is a chain map preserving the Lie bracket, i.e.\begin{align*}
		 & \partial f(x) = f(\partial x) \quad \textup{for all }x\in C_*; 	\\&[f(x),f(y)] = f([x,y])\quad \textup{ for all }x,y\in C_*. 
	\end{align*} 
\end{definition}

For later purposes, we record the following well-known facts (modulo some unconventional signs) about deformation theories of dg algebras in the forms we need. For general theory of deformations of dg algebras, see e.g. \cite{FukayaDefo} and \cite{FOOO1}.
\begin{lemma}\label{Lm: dgDefo}
	\begin{enumerate}
		\item Let $(C_*,\partial,[,])$ be a dg Lie algebra, and $x\in C_{-1}$. Then if $x$ is a Maurer-Cartan element, i.e. \[ \partial x - \frac{1}{2}[x,x] = 0, \]then the following deformation of $\partial$,\[ \partial^x\colon C_*\to C_{*-1};\quad \partial^x(y):= \partial y - [x,y],\]is a differential. Moreover, $(C_*,\partial^x,[,])$ is still a dg Lie algebra.
		\item Let $(C_*, \partial, \bullet)$ be a dg associative algebra, and $\alpha \in C_{-1}$. Then $\alpha$ is a Maurer-Cartan element (``bounding chain'' in \cite{FOOO1} terminology), i.e. \[ \partial \alpha - \alpha \bullet \alpha = 0, \]
		then the following deformation of $\partial$,
		\[
			\partial^\alpha\colon C_*\to C_{*-1},\quad \partial^\alpha(\beta):= \partial \alpha - \alpha \bullet \beta + (-1)^{|\beta|} \beta \bullet \alpha,
		\]
		is a differential. Moreover, $(C_*, \partial^\alpha, \bullet)$ is still a dg Lie algebra;
		\item Let $(C_*, \partial,\bullet)$ be a dg associative algebra. Consider the Hochschild cochain complex $\CC^*(C_*,C_*) = \bigoplus_{\ell\geq 0} \Hom_{*+\ell-1}(C_*^{\otimes \ell}, C_*)$. Suppose $\Phi\in \CC^{-1}(C_*,C_*)$ is only non-zero in the $1$-ary component and write that component as $\Phi_1\in \Hom_{-1}(C_*,C_*)$, and that it satisfies the Maurer-Cartan equation in $\CC^*(C_*,C_*)$: \[ \delta \Phi - \frac{1}{2}[\Phi,\Phi]=0. \]Then the following deformation of $\partial$, 
		\[ \partial^\Phi\colon C_*\to C_{*-1}, \quad \partial^\Phi(\beta):=  \partial \beta -\Phi_1(\beta)  \]is a differential, with $(C_*, \partial^\Phi, \bullet)$ still a dg associative algebra. Denote this deformed dg associative algebra by $C_*^\Phi:= (C_*,\partial^\Phi,\bullet)$.
		\item Under the same setting, the Hochschild cochain complex differential of $C_*$ deformed by $\Phi\in \CC^{-1}(C_*,C_*)$ (as a dg Lie algebra) \[ \delta_{C_*}^\Phi\colon \CC^*(C_*,C_*)\to \CC^*(C_*,C_*); \quad \Psi \mapsto  \delta\Psi - [\Phi,\Psi]  \]agrees with the Hochschild cochain complex differential of $C_*^\Phi$ \[ \delta_{C_*^\Phi}\colon \CC^*(C_*^\Phi,C_*^\Phi)\to \CC^*(C_*^\Phi,C_*^\Phi). \]Here we have used that the underlying graded vector spaces of $C_*$ and $C_*^\Phi$ are identical.
	\end{enumerate}
\end{lemma}

\subsection{String topology dg algebras}\label{Subsc: dgAlgebraSignApp}

The goal of this section is to verify Lemma \ref{Lm: StructuresOnStates} and Lemma \ref{Lm: dgLieMorphism} regarding the dg algebra structures on $\CFr,\CBa$ and the homomorphism $\CFr\xrightarrow{\CO} \CC^*(\CBa,\CBa)$.

\subsubsection{Closed string state space}\label{Subsubsc: AppClosedState}

\begin{definition}\label{Defn: SpecialEltLL}
	Define an element
	\[ \LL \in \Cfr{-1}(0) = \prod_{k\in \Z_{\geq 0}} C^\dR_{\dim L + k - 2} (\mathscr{L}^{k+1}(0))\]
where 
\begin{itemize}
	\item For $k=2$, consider the map $L\xrightarrow{\iota} \mathscr{L}^3 (0)$ defined by $y \mapsto (\underline{y},\underline{y},\underline{y})$  (recall that for a point $y \in L$, we use $\underline{y}$ to denote the constant path at $y$; see Section \ref{Subsc: StateSpace}). The de Rham chain $[(L\xrightarrow{\iota} \mathscr{L}^3 (0); 1\in \mathscr{A}_c^0(L))]$ defines a closed cycle in $C^\dR_{\dim L}(\mathscr{L}^3(0))$, which we set to be the $k=2$ component of $\LL$; 
	\item For all $k\neq 2$, set the $k$-th component of $\LL$ to be 0.
\end{itemize}
	Define \[ \wt{\LL}:= (-1)^{\dim L + 1} \LL. \]
\end{definition}

The following proof imitates the construction in section 5 in \cite{Irie2}, where energy-zero moduli spaces of pseudo-holomorphic curves (whose role is played by $\LL$ in our proof) are used to deform the dg Lie algebra $(\CFr,\partial^\dR, [,])$.

\begin{lemma}[Lemma \ref{Lm: StructuresOnStates} (1)]\label{Lm: StructuresOnClosedStatesApp}
		$(\CFr, \partial, [,])$ is a dg Lie algebra.
\end{lemma}

Instead of checking directly, we will first use Lemma \ref{Lm: dRchainClosedStringApp} to directly verify that $(\CFr, \partial^\dR, [,])$ is a dg Lie algebra, and show that the term $\partial^1 = \partial - \partial^\dR=\partial - \partial^0$ can be given by deforming $(\CFr, \partial^\dR, [,])$ by a certain Maurer-Cartan element, and thus $(\CFr, \partial,[,])$ is a dg Lie algebra.

\begin{lemma}\label{Lm: StructureOnClosedStatesDRApp}
	$(\CFr,\partial^\dR, [,])$ is a dg Lie algebra.
\end{lemma}
\begin{proof}
The Leibniz identity \eqref{Eq: dgLa Leib} follows from the Leibniz rule of $\circ_i^\mathscr{L}$ (Lemma \ref{Lm: dRchainClosedStringApp} (1)), anti-symmetry \eqref{Eq: dgLa antisym} follows from the definition \eqref{Eq: LoopBracket} directly, and the Jacobi identity follows from the fact that $\circ$ is a pre-Lie product:
\[
	(x\circ y)\circ z - x\circ (y\circ z)  = (-1)^{|y||z|} \big((x\circ z)\circ y - x\circ(z\circ y)\big),
\]
which in turn follows from associativity of $\circ_i^\mathscr{L}$ (Lemma \ref{Lm: dRchainClosedStringApp} (2)).
\end{proof}

\begin{lemma}\label{Lm: E=0 MC elt}
	$\wt{\LL}:= (-1)^{\dim L + 1}\LL$ is a Maurer-Cartan element in $(\CFr, \partial^\dR, [,])$, i.e. 
\begin{align*}
	\partial^\dR \wt{\LL} - \frac{1}{2}\left[\wt{\LL},\wt{\LL}\right]  = 0.
\end{align*}
\end{lemma}
\begin{proof}
Now we show this by first verifying that $\LL$ is  a Maurer-Cartan element.
The term $\partial^\dR \LL$ vanishes since the de Rham chain representing only non-vanishing component ($a=0,k=2$) of $\LL$ is $[(L\xrightarrow{\iota} \mathscr{L}^3(0); 1\in \mathscr{A}^0_c(L))]$. By arity reasons, $\LL\circ \LL$ is only non-zero in the component $a=0$, $k=3$, and moreover by definition \eqref{Eq: LoopPreBracket}, \[ \LL \circ \LL (0,3) = \LL(0,2) \circ_1 \LL(0,2) - \LL(0,2)\circ_2 \LL(0,2). \]
We claim that $\LL(0,2)\circ_1\LL(0,2) = \LL(0,2)\circ_2\LL(0,2)$. 
By definition of $\circ_i^\mathscr{L}$ (see \eqref{Eq: PartialLoopPre}), we have, for $i=1$ or $2$,
\[
	\LL(0,2)\circ_i \LL(0,2) = [(L\fiberprod{i}{0} L, \iota\circ_i \iota, 1\times 1)];
\]
here the two evaluation maps $\ev_i\circ \iota$ and $\ev_0\circ \iota$ from $L$ to $L$ coincide, so $L\fiberprod{i}{0} L\cong L\subset L\times L$ is simply the diagonal. Also, $\iota\circ_i\iota$ denotes the composition
\[
	L\cong L\fiberprod{i}{0}L \xrightarrow{\iota\times \iota}\mathscr{L}^3(0)\fiberprod{\ev_i}{\ev_0} \mathscr{L}^3(0) \xrightarrow{\conc_{2,i,2}} \mathscr{L}^4(0).
\]
For any $y\in L\cong L\fiberprod{i}{0}L$, one directly checks that $\iota\circ_i\iota (y) = (\underline{y},\underline{y},\underline{y},\underline{y})$ for either $i=1$ or $2$ (since the composition of two constant loops $\underline{y} * \underline{y}$ returns the same constant loop $\underline{y}$).
Therefore $\LL(0,2)\circ_1\LL(0,2) = \LL(0,2)\circ_2 \LL(0,2)$, and thus
\begin{align*}
	\frac{1}{2} [\LL,\LL](a,k) &= \LL\circ \LL(a,k) =  \begin{cases}
		\LL(0,2)\circ_1 \LL(0,2) - \LL(0,2) \circ_2\LL(0,2), &  a= 0, k=2\\
		0, &\textup{otherwise}
	\end{cases}\\
	&= 0
\end{align*}
for any $a\in H_1(L;\Z), k\in \Z_{\geq 0}$. 
This verifies that $\LL$ is a Maurer-Cartan element.
Now since $\partial^\dR \LL=0$, it follows immediately that modifying $\LL$ by a sign, $\wt{\LL}:= (-1)^{\dim L + 1} \LL$, also gives a Maurer-Cartan element.
\end{proof}

\begin{proof}[Proof of Lemma \ref{Lm: StructuresOnClosedStatesApp}]
We claim that $-[\wt{\LL}, -] = \partial^1$.
Given any $x\in \CFr$, we have, for each $a\in H_1(L;\Z)$ and $k\in\Z_{\geq 0}$,
\begin{align*}
	[\LL, x](a,k)  =& (\LL \circ x)(a,k) - (-1)^{|x|} (x\circ \LL)(a,k) \\
	=& (-1)^{|x|+k} \LL(0,2)\circ_1 x(a,k-1) + (-1)^{|x|} \LL(0,2)\circ_2 x(a,k-1)\\
	&-(-1)^{|x|}\sum_{i=1}^{k-1} (-1)^{i-1} x(a,k-1)\circ_i \LL(0,2).
\end{align*}

Now recall that (see \eqref{Eq: differentialExplicit}) 
\[(\partial^1 x)(a,k):= 
	(-1)^{\dim L + |x|}\sum_{i=0}^k (-1)^i(\delta_i)_*(x(a,k-1)),\]where 
		\[\delta_i \colon \mathscr{L}^{i}L\to \mathscr{L}^{i+1} L; \quad \delta_i(c_0,\dots, c_{k-1}):= \begin{cases}
		(c_0,\dots, c_{i-1},\underline{\textsf{s}(c_i)}, c_i,\dots, c_{k-1}), & 0 \leq i \leq k-1\\
		(c_0,\dots, c_{k-1}, \underline{\textsf{t}(c_{k-1})}), & i=k
	\end{cases}.
	\]

	One checks that  (see section \ref{Subsc: Ori} for orientation convention)
	\begin{align*}
		\LL(0,2) \circ_1 x(a,k-1) = (\delta_k)_*(x(a,k-1)),\quad \LL(0,2)\circ_2 x(a,k-1) = (\delta_0)_*(x(a,k-1)),
	\end{align*}
	and for each $i=1,\dots, k-1$, \[ x(a,k-1) \circ_i \LL(0,2) = (\delta_i)_*(x(a,k-1)). \]
	We see that the signs in front of each terms also match up once we replace $\LL$ by $-\wt{\LL}$:
	\begin{align}\label{Eq: [L,-]=delta}
		-[\wt{\LL}, x] =   \partial^1 x.
	\end{align}

From general deformation theory of dg Lie algebras 
(Lemma \ref{Lm: dgDefo} (1)), since $\wt{\LL}$ is a Maurer-Cartan element, the deformed algebra $(\CFr,\partial^\dR- [\wt{\LL},-], [,])$ is a dg Lie algebra. Since $-[\wt{\LL},-] = \partial^1$, we can conclude that $(\CFr, \partial=\partial^0+\partial^1, [,])$ is a dg Lie algebra.
\end{proof}

\subsubsection{Open string --- the de Rham chain part}
For the verification of part (2) of Lemma \ref{Lm: StructuresOnStates} and Lemma \ref{Lm: dgLieMorphism}, we partially imitate the strategy in section \ref{Subsubsc: AppClosedState}; e.g. to show $(\CBa,\partial,[,])$ is a dg associative algebra,
\begin{enumerate}
	\item Prove the statements for the ``de Rham chain'' part of the structures, i.e. $(\CBa, \partial^\dR, [,])$ is a dg associative algebra;
	\item Use general deformation theory machinery to deform the the dg associative structures using tautological elements in $\CFr$ and $\CBa$;
	\item Identify the deformation with $\partial^1 = \partial - \partial^0=\partial - \partial^\dR$.
\end{enumerate}

We start with part (1) in this section.

\begin{lemma}
	$(\CBa,\partial^\dR, \bullet)$ is a dg associative algebra with strict unit $[\underline{\star}]\in \Cba{0}$.
\end{lemma}
\begin{proof}
	The Leibniz rule \eqref{Eq: dgaLeibniz} follows from the Leibniz rule of $\bullet$ (Lemma \ref{Lm: dRchainBulletProdApp} (1)), and 
	associativity follows from the associativity of $\bullet$ (Lemma \ref{Lm: dRchainBulletProdApp} (2)).
	The strict unitality property is clear.
\end{proof}

\subsubsection{Closed-open string map preserves the Lie bracket}

A remark on notations:
recall that, given $x\in \CFr$, we have $\CO_0(x) \in \CBa$ and $\CO_1(x)\in \Hom(\CBa,\CBa)$. For ease of reading, we will put a curly bracket $\big\{\CO_1(x)\big\}$ when thinking of it as a map, so that e.g. $\big\{\CO_1(x)\big\} \big(\CO_0(y)\big)$ means applying the homomorphism $\CO_1(x)\colon \CBa\to \CBa$ on the element $\CO_0(y)\in \CBa$.

\begin{lemma}\label{Lm: COpreserveLieApp}
	The map $\CO\colon \CFr\to \CC^*(\CBa,\CBa)$ preserves the Lie bracket.
\end{lemma}
\begin{proof}
	We need to verify that for all $x,y\in \CFr$,
\[ \CO([x,y])=[\CO(x),\CO(y)] .\]
	We verify it arity-by-arity.
\begin{itemize}
\item  \emph{For the 0-ary part}: 
We need to show that \[ \CO_0(x\circ y - (-1)^{|x||y|} y\circ x) = \big\{ \CO_1(x) \big\}\big(\CO_0(y)\big) - (-1)^{|x||y|} \big\{\CO_1(y)\big\} \big(\CO_0(x)\big). \]

It suffices to show that \[ \CO_0(x\circ y) = (-1)^{|x||y|+1}  \big\{\CO_1(y)\big\}\big(\CO_0(x)\big).\]
Using \eqref{Eq: CO1 openclosed}, this is the same as
\begin{align*}
	\CO_0(x\circ y) = (-1)^{|x||y|+1} (-1)^{(|x|-1)|y|+1} \CO_0(x)\circ y = (-1)^{|y|} \CO_0(x)\circ y.
\end{align*}
where $\circ$ on the right-hand side is the open-closed product.
Using \eqref{Eq: CO0 OpenClosed} and \eqref{Eq: CO1 Expand} to expand, for each $a\in H_1(L;\Z)$ and $k\in \Z_{\geq 0}$, we have
\begin{align*}
	\CO_0(x\circ y)(a,k) = \sum_{\substack{k'+k''=k+1\\1\leq i\leq k'\\ a'+a'' = a}}(-1)^{\ddagger_1} (-1)^{|x|+|y|}\mathfrak{o}\big(x(a',k')\circ_i y(a'',k'')\big)
\end{align*}
and 
\begin{align*}
	(-1)^{|y|}(\CO_0(x)\circ y)(a,k)&= \sum_{\substack{k'+k''=k+1\\1\leq i\leq k'\\ a'+a'' = a}}(-1)^{\ddagger_2} (-1)^{|x|+|y|}\mathfrak{o}\big(x(a',k')\big) \circ_i y(a'',k'')
\end{align*}
where (see equations \eqref{Eq: LoopPreBracket} and \eqref{Eq: CO1 Expand})
\begin{align*}
	\ddagger_1 &= (i-1)(k''-1)+(k'-1)(|y|+1+k'')\\
	\ddagger_2 &= (i-1)(k''-1) + k'(|y|+1+k'')
\end{align*}
and therefore the discrepancy is $|y|+1+k'' \equiv \deg(y(a'',k'')) -\dim L \mod 2$ (see \eqref{Eq: DegConvertFree}).
The result then follows from the fact that (Lemma \ref{Lm: AnomalyApp} (1))
\begin{align*}
	\mathfrak{o}\big(x(a',k')\circ_i y(a'',k'')\big) = (-1)^{\deg(y(a'',k'')) - \dim L} \mathfrak{o}\big(x(a',k')\big) \circ_i y(a'',k'').
\end{align*}

	\bigskip

\item 	
\emph{For the unary part:} We need to show that 
\[   \big\{\CO_1([x,y])\big\}(\alpha) = \big\{[\CO(x),\CO(y)]_1\big\}(\alpha) \] for $\alpha\in \CBa$ and $x, y\in \CFr$.

		Expanding both sides, this really means 
\begin{align*}
	\big\{\CO_1(x\circ y)\big\} \alpha &- (-1)^{|x||y|} \big\{ \CO_1(y\circ x)\big\} \alpha \\
	&= \big\{\CO_1(x)\big\}\big(\{\CO_1(y)\}\alpha\big) - (-1)^{|x||y|}\big\{\CO_1(y)\big\} \big( \{\CO_1(x)\} \alpha\big).
\end{align*}
Using \eqref{Eq: CO1 openclosed}, this turns into \[ \alpha\circ (x\circ y) - (\alpha\circ x)\circ y = (-1)^{|x||y|} \left(\alpha\circ (y\circ x) - (\alpha\circ y ) \circ x\right). \]
This formula is analogous to the Jacobi for the loop bracket (or really that the $\circ$-product is a pre-Lie product), e.g. Lemma 4.2 of \cite{ChasSullivan}.
It follows from Lemma part (2) (the associativity part) of \ref{Lm: dRchainOpenClosedApp}, analogous to that in the proof of Lemma \ref{Lm: StructureOnClosedStatesDRApp}.

\bigskip

\item \emph{For higher arity parts:} both sides of the equation automatically vanish by construction. 
\end{itemize}
\end{proof}

\subsubsection{Closed-open string map --- the de Rham chain part}
We will temporarily use the notation $(\CFr)^\dR$ to denote the dg Lie algebra $(\CFr,\partial^\dR, [-,-])$, and $(\CBa)^\dR$ to denote the dg associative algebra $(\CBa, \partial^\dR, [-,-])$.
We temporarily use the following notation
\begin{align*}
	\CO_1\colon (\CFr)^\dR \to \CC^*\big((\CBa)^\dR, (\CBa)^\dR\big)
\end{align*}
to mean the map with only the unary component being non-zero and is given by $\CO_1$ as in \eqref{Eq: CO1 openclosed}, and all the $\ell$-ary component where $\ell\neq 1$ are set to 0.
\begin{lemma}\label{Lm: OCdR dgLaHomomApp}
	The map $\CO_1\colon (\CFr)^\dR \to \CC^*\big((\CBa)^\dR, (\CBa)^\dR\big)$ is a homomorphism of dg Lie algebras.
\end{lemma}

\begin{proof}
By arity reasons, to verify that $\CO_1$ preserves the Lie bracket, the only equation needed to be checked is
\[   \big\{\CO_1([x,y])\big\}(\alpha) = \big\{[\CO(x),\CO(y)]_1\big\}(\alpha) \quad \quad  \textup{for all }\alpha\in \CBa\textup{ and }x, y\in \CFr.\]
This follows from (the proof of) Lemma \ref{Lm: COpreserveLieApp}.

To show that $\CO$ preserves the differential, we need to verify that for all $x\in \CFr$, 
\[			\delta\big( \CO(x)\big)=\CO(\partial^\dR x).\]

\begin{itemize}
	\item \emph{For the unary part}: we need \[\big\{\delta(\CO(x))\big\}_1(\alpha) = \big\{\CO_1(\partial^\dR x)\big\}(\alpha)\] for $\alpha\in \CBa$.
	That is, 
	\begin{align*}
		  \big\{ \CO_1(\partial^\dR x)\big\} (\alpha)= \partial^\dR &\left( \big\{ \CO_1(x)\big\} \alpha \right) - (-1)^{|x|} \big\{ \CO_1(x)\big\} (\partial^\dR \alpha) .
	\end{align*}	
	Using \eqref{Eq: CO1 openclosed} we see that this is the same as 
	\[
		\partial^\dR(\alpha \circ x) = (\partial^\dR \alpha) \circ x + (-1)^{|\alpha|} \alpha\circ (\partial^\dR x).
	\]
	But this follows from the Leibniz rule of $\circ =\circ^\Omega$ in Lemma \ref{Lm: dRchainOpenClosedApp} (1).
\bigskip

	\item \emph{For the binary part}, we need 
	\[\big\{\delta(\CO(x))\big\}_2(\alpha_1,\alpha_2)=\big\{\CO_2(\partial x)\big\}(\alpha_1,\alpha_2) \] for $\alpha_1,\alpha_2\in \CBa$.
		The right-hand side  vanishes because $\CO_2$  is $0$ by definition.
		Expanding the left-hand side, we see that this equation is the same as  \begin{align*}
			\big\{\CO_1(x)\big\} (\alpha_1\bullet\alpha_2)= (-1)^{|\alpha_1||x|} \alpha_1 \bullet \big( \big\{ \CO_1(x)\big\}(\alpha_2) \big) + \big( \big\{ \CO_1(x)\big\} (\alpha_1) \big)\bullet \alpha_2,
		\end{align*}	
		or, using \eqref{Eq: CO1 openclosed}, \[ (\alpha_1\bullet \alpha_2) \circ x = \alpha_1\bullet(\alpha_2\circ x) + (-1)^{|\alpha_2||x|} (\alpha_1\circ x) \bullet \alpha_2. \]
		But this follows from the compatibility of $\circ=\circ^\Omega$ with $\bullet$ in Lemma \ref{Lm: dRchainOpenClosedApp} (3). 

\item \emph{For $0$-ary and higher arity parts:} both sides of the equation automatically vanish by construction.
\end{itemize}
\end{proof}

\subsubsection{Open string}
In this section we will show:
\begin{lemma}[Lemma \ref{Lm: StructuresOnStates} (2)]\label{Lm: StructuresOnOpenStatesApp}
	$(\CBa, \partial, \bullet)$ is a dg associative algebra with strict unit $[\underline{\star}] \in \Cba{0}$;
\end{lemma}

We have previously shown that $(\CBa)^\dR$ is a dg associative algebra.
The proof will proceed by deforming it to $\CBa$ in two stages:
\begin{itemize}
	\item \emph{Step 1} (Lemma \ref{Lm: OpenStrApp Step1}): We first consider an intermediate algebra $(\CBa)^\textup{int}:=\big(\CBa, \partial^{\textup{int}}, \bullet\big)$ with differential \begin{align*}
		\partial^{\textup{int}} := \partial^0 + (\partial^1)^{\textup{int}}
	\end{align*}
	where, for each $\alpha \in \CBa$, $a\in H_1(L;\Z)$, $k\in \Z_{\geq 0}$,
	\begin{align}\label{Eq: differentialInt}
			\partial^0:=\partial^\dR,\quad \big((\partial^1)^{\textup{int}} \alpha\big)(a,k):= 
	(-1)^{\dim L+|\alpha|}\sum_{i=1}^{k-1} (-1)^i(\delta_i)_*(\alpha(a,k-1)).
	\end{align}
	The difference between this and the differential $\partial=\partial^0+\partial^1$ (in  equations \eqref{Eq: differential}, \eqref{Eq: differentialExplicit}) is that the summation in $\partial^{\textup{int}}$ does not contain the two ``boundary'' marked points $i=0$ and $k$.
	This step is achieved by pushing forward the Maurer-Cartan element $\wt{\LL} \in (\CFr)^\dR$ along $\CO^1\colon (\CFr)^\dR\to \CC^*((\CBa)^\dR,(\CBa)^\dR)$ and using general deformation theoretic machinery.
	\item \emph{Step 2} (Lemma \ref{Lm: OpenStrApp Step2}): We then identify a Maurer-Cartan element of $(\CBa)^{\textup{int}}$ to deform it into the desired dg associative algebra $\CBa$.
\end{itemize}

\begin{lemma}\label{Lm: OpenStrApp Step1}
	$(\CBa)^{\textup{int}}$ is a dg associative algebra.
\end{lemma}

\begin{proof}
 By Lemma \ref{Lm: E=0 MC elt}, $\wt{\LL}\in \Cfr{-1}$ is a Maurer-Cartan element. By Lemma \ref{Lm: OCdR dgLaHomomApp}, the element $\CO_1(\wt{\LL}) \in \CC^{-1}((\CBa)^\dR, (\CBa)^\dR)$ is a Maurer-Cartan element. By general deformation theory (Lemma \ref{Lm: dgDefo} (3)), the chain complex $(\CBa, \partial^\dR - \CO_1(\wt{\LL}), \bullet)$ is a dg associative algebra, since the only non-zero-arity component of $\CO_1\colon (\CFr)^\dR \to \CC^*\big((\CBa)^\dR,(\CBa)^\dR\big)$ is the unary one. 
 Therefore we just need to identify $\{\CO_1(\wt{\LL})\}\in \Hom(\CBa,\CBa)$ with $(\partial^1)^{\textup{int}}$.
 
 Given $\alpha\in \CBa$, by definition \eqref{Eq: CO1 openclosed} and that $\wt{\LL} = (-1)^{\dim L + 1} \LL$,  we have $\big\{\CO_1(\wt{\LL})\big\}(\alpha) = (-1)^{|\alpha|+1}\alpha\circ \wt{\LL} = (-1)^{\dim L + |\alpha|} \alpha \circ \LL$. By \eqref{Eq: CO1 Expand}, for each $k\in \Z_{\geq 0}$ and $a\in H_1(L;\Z)$:
 \begin{align*}
 	\big(\big\{\CO_1(\wt{\LL})\big\}\alpha \big)(a,k) &= (-1)^{\dim L + |\alpha|}(\alpha\circ \LL)(a,k) \\&= (-1)^{\dim L + |\alpha|}\sum_{i=1}^{k-1} (-1)^{i-1}  \alpha (a,k-1)\circ_i \LL(0,2).
 \end{align*}
 On the other hand, 
 \begin{align*}
 	\big((\partial^1)^{\textup{int}}\alpha\big)(a,k)=(-1)^{\dim L +|\alpha|} \sum_{i=1}^{k-1} (-1)^{i}(\delta_i)_*(\alpha (a,k-1)).
 \end{align*}

Following the definition and the orientation conventions, one can check that $\alpha(a,k-1)\circ_i \LL(0,2) = (\delta_i)_*(\alpha(a,k-1)).$
Thus we have identified $-\CO_1(\wt{\LL})$ with $(\partial^1)^{\textup{int}}$.
\end{proof}

\begin{definition}
	Define an element \[ \star^1\in \Cba{-1}(0) = \prod_{k\in \Z_{\geq 0}} C^\dR_{k-1}(\Omega_\star^{k+1}(0))\] as follows: 
	\begin{itemize}
		\item For $k=1$, consider the map $\pt\xrightarrow{\iota} \Omega^2_\star(0)$ mapping to $(\underline{\star},\underline{\star})$ (recall that $\star \in L$ is the chosen basepoint, and $\underline{\star}$ means the constant path at $\star$). The de Rham chain $[(\pt\xrightarrow{\iota} \Omega^2_\star(0); 1\in \mathscr{A}^0_c(\pt))]$ defines a closed cycle in $C^\dR_{0}(\Omega^2_\star(0))$, which we set to be the $k=1$ component of $\star^1$;
		\item For any other $k\neq 1$, set the component $\star^1(0,k)=0$.
	\end{itemize}
	Define \begin{align}\label{Eq: BasedLoopMCApp} \wt{\star^1}:= (-1)^{(\dim L + 1)} \star^1 \in \Cba{-1}(0). \end{align}
\end{definition}

\begin{lemma}\label{Lm: OpenStrApp Step2}
	$\wt{\star^1}=(-1)^{(\dim L  +1)}\star^1$ is a Maurer-Cartan element (bounding chain) in $(\CBa)^{\textup{int}}$, i.e. 
	\begin{align*}
		\partial^{\textup{int}}\wt{\star^1} - \wt{\star^1} \bullet \wt{\star^1} = 0.
	\end{align*}
\end{lemma}

\begin{proof}
	We have $\partial^\dR \star^1=0$, and $\big((\partial^1)^{\textup{int}}\star^1\big) (a,k) = 0$ unless $a=0$ and $k=2$, in which case
	\begin{align*} 
		\big((\partial^1)^{\textup{int}}\star^1\big) (0,2) = (-1)^{\dim L} (\delta_1)_*[(\pt \xrightarrow{\iota}\Omega^2_\star(0);1) ].
	\end{align*}
	On the other hand, we also have $\star^1 \bullet \star^1(a,k)=0$ unless $a=0$ and $k=2$, in which case
	\[
		(\star^1 \bullet \star^1)(0,2) = -[(\pt \xrightarrow{\iota}\Omega^2_\star(0);1) ] \bullet [(\pt \xrightarrow{\iota}\Omega^2_\star(0);1) ].
	\]
	Now both $(\delta_1)_*[(\pt \xrightarrow{\iota}\Omega^2_\star(0);1) ]$ and $[(\pt \xrightarrow{\iota}\Omega^2_\star(0);1) ]\bullet [(\pt \xrightarrow{\iota}\Omega^2_\star(0);1) ]$ are equal to 
	\[
		[(\pt \xrightarrow{\iota} \Omega^3_\star(0);1)]\in C_0^\dR(\Omega^3_\star(0)), \quad \iota(\pt) := (\underline{\star},\underline{\star},\underline{\star}).
	\]
	Therefore \[ (-1)^{\dim L}(\partial^1)^{\textup{int}}\star^1 + (\star^1\bullet \star^1) = 0. \]
	Thus with $\wt{\star^1}:= (-1)^{\dim L + 1}\star^1$, we have \[ (\partial^1)^{\textup{int}} \wt{\star^1} - \wt{\star^1} \bullet \wt{\star^1} =0.  \]
\end{proof}

\begin{proof}
	[Proof of Lemma \ref{Lm: StructuresOnOpenStatesApp}]
	It remains to identify the deformation provided by  $\wt{\star^1}$ with $\partial = \partial^0+\partial^1$.
	That $[\underline{\star}]$ is the identity is straightforward from definition.
	
	We claim that for any $\alpha \in \CBa$, $a\in H_1(L;\Z)$, $k\in \Z_{\geq 0}$,
	\begin{align*}
		\star^1(0,1)\bullet \alpha (a,k-1) = (\delta_0)_*(\alpha(a,k-1)), \quad \alpha(a,k-1) \bullet \star^1(0,1) = (\delta_k)_*(\alpha(a,k-1)).
	\end{align*}
	Say $\alpha(a,k-1) = [(U\xrightarrow{\varphi} \Omega^k_\star(a);\omega \in\mathscr{A}_c^*(U))]$. Then from the definitions, both sides of the first equation above are equal to $[(U\xrightarrow{ \delta \circ\varphi}\Omega^{k+1}_\star(a); \omega \in \mathscr{A}_c^*(U))]$ where $\Omega_\star^k(a)\xrightarrow{\delta} \Omega^{k+1}_\star(a)$ is given by \[ (c_0,\dots, c_{k-1}) \in \Omega_\star^k(a) \mapsto  (\underline{\star}, c_0,\dots, c_{k-1}) \in \Omega^{k+1}_\star(a). \]
	The second equation is similar.
	
	These equations then translate to (using \eqref{Eq: PontryaginProd} and \eqref{Eq: BasedLoopMCApp}) \[-(\wt{\star^1}\bullet \alpha)(a,k) = (-1)^{|\alpha|+\dim L}  (\delta_0)_* (\alpha(a,k-1))  \] and \[ (-1)^{|\alpha|} (\alpha\bullet\wt{\star^1})(a,k) = (-1)^{|\alpha|+\dim L+k} (\delta_k)_* (\alpha(a,k-1)). \]

	It follows from general deformation theory (Lemma \ref{Lm: dgDefo} (3)) that 
	\[
		\alpha \mapsto \partial^{\textup{int}}\alpha - (\wt{\star^1}\bullet \alpha) + (-1)^{|\alpha|} (\alpha\bullet \wt{\star^1})
	\]
	is a differential, and together with $\bullet$, they make $C_*^\Omega$ into a dg associative algebra. But $\partial = \partial^{\textup{int}} - (\wt{\star^1}\bullet -) + (-1)^{|\alpha|} (-\bullet \wt{\star^1}).$ Therefore $(\CBa,\partial,\bullet)$ is a dg associative algebra.
\end{proof}

\subsubsection{Closed-open string map}
In this section we finish the proof of:
\begin{lemma}[Lemma \ref{Lm: dgLieMorphism}]\label{Lm: dgLieMorphismApp}
	$\CO$ is a homomorphism of dg Lie algebra.
	Moreover, $\CO$ respects the decompositions of $\CFr$ and $\CC^*(\CBa,\CBa)$ into $a\in H_1(L;\Z)$.
\end{lemma}

\begin{lemma}\label{Lm: COint preserves}
	$\CO_1\colon \CFr\to \CC^*((\CBa)^{\textup{int}}, (\CBa)^{\textup{int}})$ is a homomorphism of dg Lie algebras.
\end{lemma}
As before, $\CO_1$ means the map with only the unary component being non-zero and is given by $\CO_1$ as in \eqref{Eq: CO1 openclosed}.
\begin{proof}
	This follows from the facts that the dg Lie algebra $\CFr$ is given by deforming $(\CFr)^\dR$ using the Maurer-Cartan element $\wt{\LL}$, that the dg associative algebra $(\CBa)^{\textup{int}}$ is given by deforming $(\CBa)^\dR$ using $\CO_1(\wt{\LL})$, that $\CO_1\colon (\CFr)^\dR \to \CC^*((\CBa)^\dR,(\CBa)^\dR)$ is a homomorphism of dg Lie algebras (Lemma \ref{Lm: OCdR dgLaHomomApp}), and the general deformation theory fact that deforming the dg Lie algebra $\CC^*((\CBa)^\dR,(\CBa)^\dR)$ using the Maurer-Cartan element $\CO_1(\wt{\LL})$ is the same as taking Hochschild cochains of the dg associative algebra obtained by deforming $(\CBa)^\dR$ using $\CO_1(\wt{\LL})$ (see Lemma \ref{Lm: dgDefo} (4)).
\end{proof}

\begin{lemma}\label{Lm: dRchainStarAnon}
Suppose $\alpha \in C_*^\dR(\Omega_\star^{k'+1}(a'))$ and  $x\in C_*^\dR(\mathscr{L}^{k''+1}(a''))$. 
	\begin{enumerate}
		\item  For $1\leq i \leq k'$, i.e. $2\leq i+1 \leq k'+1$,\begin{align*}
			\star^1(0,1) \bullet\big( \alpha \circ_i x\big) &= \big(\star^1(0,1)\bullet \alpha\big) \circ_{i+1} x;
		\end{align*}
		\item For $i+1=1$, \[ (-1)^{\deg x + ( \deg x - \dim L)(\deg \alpha)} \mathfrak{o}(x) \bullet \alpha =- \big(\star^1(0,1)\bullet \alpha\big) \circ_1 x; \]
		\item For $1\leq i \leq k'$, \[			\big(\alpha \circ_i x\big) \bullet\star^1(0,1) = \big( \alpha \bullet \star^1(0,1)\big) \circ_i x; \]
		\item For $i=k'+1$, \[ (-1)^{\deg x}\alpha \bullet \mathfrak{o}(x) = -\big(\alpha \bullet \star^1(0,1)\big)\circ_{k'+1} x. \]
	\end{enumerate}
\end{lemma}
\begin{proof}
 Let $\alpha:= [(V\xrightarrow{\psi} \Omega_\star^{k'+1}(a');\eta)]$ and $x:= [(U\xrightarrow{\varphi}\mathscr{L}^{k''+1}(a'');\omega)]$.
Then for $1\leq i\leq k'$,
  \begin{align*}
 	\star^1(0,1)\bullet \big(\alpha\circ_i x\big) &= (-1)^{(\deg \alpha )|\omega|} \star^1(0,1) \bullet [( V\fiberprod{i}{0}U\xrightarrow{\psi\circ_i \varphi} \Omega_\star^{k'+k''}(a); \eta\times \omega )]\\
 	&= (-1)^{(\deg\alpha)|\omega|} [( V\fiberprod{i}{0} U \xrightarrow{ \star^1 \bullet(\psi\circ_i\varphi) } \Omega_\star^{k'+k''+1}(a) ; \eta\times \omega)].
 \end{align*}
 where the map $V\fiberprod{i}{0} U \xrightarrow{\star^1 \bullet(\psi\circ_i\varphi)} \Omega_\star^{k'+k''+1}(a)$ is the composition 
 \[
 \pt \times (V\fiberprod{i}{0}U)\xrightarrow{(\underline{\star},\underline{\star})\times (\psi\circ_i\varphi)}\Omega^2_\star(0) \times \Omega^{k'+k''}_\star(a) \xrightarrow{*}\Omega^{k'+k''+1}_\star(a).
 \]
 On the other hand, 
 \begin{align*}
 	&\big(\star^1(0,1)\bullet \alpha\big) \circ_{i+1} x\\ =& [(V \xrightarrow{\star^1\bullet \psi} \Omega^{k'+2}_\star(a'); \eta )] \circ_{i+1} x\\
 	=& (-1)^{(\deg \alpha)|\omega|} [ (V\fiberprod{ \ev_{i+1}\circ(\star^1\bullet \psi)}{\ev_0\circ \varphi}U \xrightarrow{(\star^1\bullet\psi) \circ_{i+1} \varphi} \Omega_\star^{k'+k''+1}(a); \eta\times \omega) ] 
 \end{align*}
 where the map $V\xrightarrow{\star^1\bullet \psi} \Omega_\star^{k'+2}(a')$ is the composition 
 \begin{align*}
 	\pt\times V\xrightarrow{(\underline{\star},\underline{\star}) \times \psi} \Omega^2_\star(0)\times \Omega^{k'+1}_\star(a) \xrightarrow{ *} \Omega_\star^{k'+2}(a).
 \end{align*}
Then (1) follows from the fact that the two maps $V\xrightarrow{\ev_i\circ \psi} L$ and $V\xrightarrow{\ev_{i+1}\circ (\star^1\bullet \psi)}L$ are identical, and that the two maps \[ V\fiberprod{\ev_i\circ \psi}{\ev_0\circ\varphi} U \xrightarrow{\psi\circ_i\varphi} \Omega_\star^{k'+k''}(a) \xrightarrow{\star^1\bullet-} \Omega^{k'+k''+1}_\star(a) \]
	and 
	\[
		V\fiberprod{\ev_{i+1}\circ(\star^1\bullet \psi)}{\ev_0\circ \varphi}U \xrightarrow{(\star^1\bullet\psi)\times \varphi} \Omega_\star^{k'+2}(a')\fiberprod{\ev_{i+1}}{\ev_0} \mathscr{L}^{k''+1}(a'')  \xrightarrow{\circ_{i+1}} \Omega_\star^{k'+k''+1}(a)
	\]
	are identical.
	
	For (2), 
	\begin{align*}
		\mathfrak{o}(x) \bullet \alpha &= (-1)^{\deg x + (\deg x - \dim L)|\eta|}[((\ev_0\circ \varphi)^{-1}(\star) \times V \xrightarrow{ \varphi\bullet \psi} \Omega_\star^{k'+k''+1}(a);\omega\times \eta)]
	\end{align*}
	Notice that 
	\begin{align*}
		\big(\star^1(0,1)\bullet \alpha\big) \circ_{1} x &= (-1)^{(\deg\alpha)|\omega|}[ (V\fiberprod{ \ev_{1}\circ(\star^1\bullet \psi)}{\ev_0\circ \varphi}U \xrightarrow{(\star^1\bullet\psi) \circ_{i+1} \varphi} \Omega_\star^{k'+k''+1}(a); \eta\times \omega) ] \\
		&=(-1)^{(\deg \alpha)|\omega|}[(V\fiberprod{\star}{\ev_0\circ \varphi}U \xrightarrow{(\star^1\bullet \psi) \circ_{i+1} \varphi} \Omega_\star^{k'+k''+1}(a);\eta\times \omega)]
	\end{align*}
	where $V\xrightarrow{\star} L$ denotes the constant map to $\star \in L$. Therefore $V\fiberprod{\star}{\ev_0\circ \varphi}U \cong  V\times (\ev_0\circ \varphi)^{-1}(\star)$, and therefore the difference in orientations of the two de Rham chains is $(-1)^{(\dim V)(\dim U-\dim L)}$. Also, the de Rham forms differ by $\eta\times \omega = (-1)^{|\eta||\omega|}\omega\times \eta$. Therefore 
	\begin{align*}
		\mathfrak{o}(x) \bullet \alpha &= (-1)^{(\dim V)(\dim U - \dim L)+|\eta||\omega|+ \deg x+(\deg x-\dim L)|\eta|+(\deg \alpha) |\omega|} \big(\star^1(0,1)\bullet \alpha\big) \circ_{1} x \\
		&=(-1)^{\deg x + ( \deg x - \dim L)(\deg \alpha)}\big(\star^1(0,1)\bullet \alpha\big) \circ_{1} x.
	\end{align*}
	This gives (2). The verifications of (3) and (4) are analogous.
\end{proof}

\bigskip

\begin{proof}
	[Proof of Lemma \ref{Lm: dgLieMorphismApp}]
We have already checked that the Lie bracket is preserved in Lemma \ref{Lm: COpreserveLieApp}. To check that the differential is preserved, we again check that $\delta(\CO(x)) = \CO(\partial x)$ arity-by-arity.

\begin{itemize}
	\item 	\emph{For the 0-ary part}: We need to show that \[\big(\delta(\CO(x))\big)_0 = \CO_0(\partial x),\]
	that is, 
	\[
		\partial \big(\CO_0(x)\big)= \CO_0(\partial x).
	\]
	Recalling that $\partial = \partial^0+\partial^1$ (see \eqref{Eq: differential}) and $\CO_0(x)(a,k)= (-1)^{|x|}\mathfrak{o}(x(a,k))$ (see \eqref{Eq: CO0 OpenClosed}), we have that for $a\in H_1(L;\Z)$ and $k\in \Z_{\geq 0}$,
	\begin{align*}
		\partial\big(\CO_0(x)\big)(a,k) =  (-1)^{|x|}\big( \partial^\dR( \mathfrak{o}(x(a,k)))\big)+ \partial^1(\CO_0(x))(a,k) 
	\end{align*}
	and 
	\begin{align*}
		\CO_0(\partial x)(a,k) =(-1)^{|x|-1}\big(\mathfrak{o}(\partial^\dR x (a,k))\big) + \CO_0(\partial^1 x)(a,k).
	\end{align*}
	First, we have $\partial^\dR(\mathfrak{o}(x(a,k))) = - \mathfrak{o}(\partial^\dR x(a,k))$ by Lemma \ref{Lm: AnomalyApp} (2). 
	On the other hand, we have, for each $i=0,\dots, k$, \[\mathfrak{o}\big((\delta_i)_* (x(a,k-1))\big)  = (\delta_i)_* \big( \mathfrak{o}(x(a,k-1))\big)
	\]by Lemma \ref{Lm: AnomalyApp} (3). Therefore 
	\begin{align*}
		\partial^1( \CO_0(x))(a,k) &= (-1)^{\dim L + |x|-1}\sum_{i=0}^k(-1)^i (\delta_i)_*(\CO_0(x)(a,k-1))\\
		&= (-1)^{\dim L -1}\sum_{i=0}^k (-1)^i (\delta_i)_* (\mathfrak{o}(x(a,k-1)))
	\end{align*}
	and 
	\begin{align*}
		\CO_0(\partial^1 x)(a,k)&= (-1)^{|x|-1}\mathfrak{o}(\partial^1 x(a,k))\\
		 &=  (-1)^{\dim L -1}\sum_{i=0}^k (-1)^i \mathfrak{o}((\delta_i)_* (x(a,k-1)))
	\end{align*}
	are the same.
\bigskip

	\item	\emph{For the unary part}: we need \[\big\{\delta(\CO(x))\big\}_1(\alpha) = \big\{\CO_1(\partial x)\big\}(\alpha)\] for $\alpha\in \CBa$.
	That is, 
	\begin{align*}
		  \big\{ \CO_1(\partial x)\big\} (\alpha)= \partial &\left( \big\{ \CO_1(x)\big\} \alpha \right) - (-1)^{|x|} \big\{ \CO_1(x)\big\} (\partial \alpha) \\ 
		  &+ (-1)^{|x|-1} \left( (-1)^{|\alpha|(|x|-1)} \alpha\bullet \CO_0(x) - \CO_0(x)\bullet \alpha \right) .
	\end{align*}	
	It follows from Lemma \ref{Lm: COint preserves} that 
	\begin{align*}
		\big\{(\CO_1)(\partial x)\big\}(\alpha) = \partial^\textup{int} \left( \big\{\CO_1(x)\big\}\alpha\right)  -(-1)^{|x|} \big\{ \CO_1(x)\big\}(\partial^\textup{int} \alpha).
	\end{align*}
	Also, from (the proof of) Lemma \ref{Lm: StructuresOnOpenStatesApp} we have \[ (\partial - \partial^\textup{int})(\alpha) = - (\wt{\star^1}\bullet \alpha) + (-1)^{|\alpha|}(\alpha \bullet \wt{\star^1}). \] 	Therefore it suffices to show that 
	\begin{align*}
		0= &-\wt{\star^1}  \bullet \left( \big\{ \CO_1(x) \big\}\alpha \right) + (-1)^{|x|+|\alpha|} \left( \big\{ \CO_1(x) \big\}\alpha \right) \bullet  \wt{\star^1}\\
		&- (-1)^{|x|} \big\{ \CO_1(x)\big\} \big(-(\wt{\star^1}\bullet \alpha ) + (-1)^{|\alpha|} (\alpha\bullet \wt{\star^1})\big)\\
		& + (-1)^{|x|-1} \left( (-1)^{|\alpha|(|x|-1)} \alpha\bullet \CO_0(x) - \CO_0(x)\bullet \alpha \right).
	\end{align*}
	Unwinding the definitions, it suffices to show that 
	\begin{align*}
		&-\mathfrak{o}(x)\bullet \alpha + (-1)^{|\alpha|(|x|-1)} \alpha\bullet \mathfrak{o}(x) \\
		=& (-1)^{|\alpha||x|}\left( \wt{\star^1}\bullet (\alpha \circ x) -  (-1)^{|\alpha|+|x|}(\alpha\circ x) \bullet \wt{\star^1} - (\wt{\star^1}\bullet\alpha)\circ x+(-1)^{|\alpha|} (\alpha\bullet \wt{\star^1})\circ x\right).
	\end{align*}

	At $a\in H_1(L;\Z)$ and $k\in\Z_{\geq 0}$, we have 
	\begin{align*}
		\big(\wt{\star^1}\bullet (\alpha\circ x)\big)(a,k) &= \sum_{\substack{k'+k''= k\\ 1\leq i \leq k' \\ a'+a''= a}} (-1)^{\ddagger_1} \wt{\star^1}(0,1) \bullet \big( \alpha(a',k')\circ_i x(a'',k'')\big);\\
		\big((\alpha\circ x)\bullet \wt{\star^1}\big)(a,k)
		&= \sum_{\substack{k'+k'' = k\\1\leq i \leq k'\\ a'+a''=a}}(-1)^{\ddagger_2} \big(\alpha(a',k') \circ_i x(a'',k'') \big)\bullet \wt{\star^1}(0,1);\\
		\big((\wt{\star^1}\bullet \alpha)\circ x\big)(a,k)&= \sum_{\substack{(k'+1)+k''=k+1 \\ 1\leq i \leq k'+1 \\ a'+a''=a}}(-1)^{\ddagger_3} \big(\wt{\star^1}(0,1)\bullet \alpha(a',k')\big)\circ_i x(a'',k'');\\
		\big((\alpha\bullet \wt{\star^1}) \circ x\big)(a,k)&= \sum_{\substack{(k'+1)+k''=k+1 \\ 1\leq i \leq k'+1 \\ a'+a''=a}}(-1)^{\ddagger_4} \big( \alpha(a',k')\bullet \wt{\star^1}(0,1) \big) \circ_i x(a'',k'');\\
		\mathfrak{o}(x)\bullet \alpha (a,k) &= \sum_{\substack{k'+k'' = k \\ a'+a'' = a}} (-1)^{\ddagger_5} \mathfrak{o}(x(a'',k'')) \bullet \alpha(a',k');\\
		\alpha\bullet \mathfrak{o}(x)(a,k)&= \sum_{\substack{k'+k''= k\\ a'+a'' = a}} (-1)^{\ddagger_6} \alpha(a',k')\bullet \mathfrak{o}(x)(a'',k'').
	\end{align*}
	where 
	\begin{align*}
		\ddagger_1 &=|\alpha|+|x| + (i-1)(k''-1)+k'(|x|+1+k'');\\
		\ddagger_2 &= (k'+k''-1)+(i-1)(k''-1)+k'(|x|+1+k'');\\
		\ddagger_3 &= (i-1)(k''-1) + (k'+1) (|x|+1+k'') + |\alpha|;\\
		\ddagger_4 &= (i-1)(k''-1) + (k'+1)(|x|+1+k'') + k'; \\
		\ddagger_5 &= k''|\alpha|;\\ 
		\ddagger_6 &= k'(|x|-1).
	\end{align*}
By parts (1) and (3) of Lemma \ref{Lm: dRchainStarAnon} we have 
\begin{align*}
	&(-1)^{|\alpha||x|}\big( \wt{\star^1} \bullet (\alpha\circ x)-(\wt{\star^1}\bullet \alpha)\circ x \big)(a,k)\\
	=& (-1)^{|\alpha||x|} \left(-\sum(-1)^{(k'+1)(|x|+1+k'')+|\alpha|} (\wt{\star^1}(0,1)\bullet \alpha(a',k')) \circ_1 x(a'',k'')\right)
\end{align*}
and 
\begin{align*}
	&(-1)^{|\alpha||x|}\big( -  (-1)^{|\alpha|+|x|}(\alpha\circ x) \bullet \wt{\star^1} +(-1)^{|\alpha|} (\alpha\bullet \wt{\star^1})\circ x \big)(a,k)\\
	=& (-1)^{|\alpha|(|x|+1)} \sum (-1)^{k'k'' + (k'+1)(|x|+1+k'')} \big(\alpha(a',k')\bullet \wt{\star^1}(0,1)\big)\circ_{k'+1}x(a'',k'').
\end{align*}
The desired equality then follows from  parts (2) and (4) of Lemma \ref{Lm: dRchainStarAnon}.

	\bigskip

 \item \emph{For the binary part}, we need 
	\[\big\{\delta(\CO(x))\big\}_2(\alpha_1,\alpha_2)=\big\{\CO_2(\partial x)\big\}(\alpha_1,\alpha_2) \] for $\alpha_1,\alpha_2\in \CBa$. This follows from the proof of Lemma \ref{Lm: OCdR dgLaHomomApp}.
\end{itemize}
\end{proof}

\bigskip

\section{Kuranishi structures and virtual fundamental chains}
\label{Sec: Kuranishi}

In this appendix, we explain the proof of Theorem \ref{Thm: VFCmain} using the theory of Kuranishi structures and virtual techniques, following \cite{FOOO20book,Irie2}.
The structure of the proof is completely analogous to that in \cite{Irie2} (and much simpler since in working with the model of the free loop space from \cite{WangThesis}, we only need to work with finite-dimensional spaces).

\subsection{Kuranishi structures on the moduli spaces}

To construct the necessary virtual fundamental chains of the moduli spaces and verify their compatibilities at the chain level in general requires using virtual techniques. We will use the theory of Kuranishi structures in \cite{FOOO1,FOOO2,FOOO20book}. 

We briefly recall the basic definitions in the theory of Kuranishi structures.
We follow Section 10 of \cite{Irie2}. 
We remark that our proof is largely independent of the details of the constructions of Kuranishi structures, and only relies on certain expected properties of the Kuranishi structures. For example, one could follow the global Kuranishi chart approach in e.g. \cite{AMS1,Rabah,HirschiHugtenburg}.
Also, the main extra ingredient relative to \cite{Irie2} is that extra care needs to be taken for energy-zero moduli spaces, which can be covered with one Kuranishi chart.
For these reasons, and to avoid lengthy discussions of chart transitions, we introduce various required notions on one Kuranishi chart, and refer to \cite{FOOO20book} for more details on chart transitions.

\begin{remark}
	The notion of Kuranishi structures we use differs from the notion of Kuranishi charts in e.g. \cite{FOOO2} in that we use smooth manifolds and vector bundles instead of orbifolds and orbibundles, because all of our moduli spaces contain at least one marked point and has no sphere bubblings (see Remark 10.1 in \cite{Irie2}). 
\end{remark}

\subsubsection{Kuranishi charts and Kuranishi spaces}

\begin{definition}[\cite{FOOO20book}, Definition 3.1]\label{Defn: KuranishiChart}Let $X$ be a separable, metrizable topological space.
	A \emph{Kuranishi chart} on $X$ is a tuple $\mathscr{U}:= (U,\mathscr{E}, s,\psi)$ such that 
	\begin{enumerate}
		\item $U$ is a smooth manifold;
		\item $\mathscr{E}\to U$ is a smooth vector bundle;
		\item $s\colon U\to \mathscr{E}$ is a section of the bundle $\mathscr{E}$;
		\item $\psi\colon s^{-1}(0) \to X$ is a homeomorphism onto an open set in $X$.
	\end{enumerate}
	The (virtual) \emph{dimension} of $\mathscr{U}$ is by definition $\vdim\mathscr{U}:=\dim U -\rank   \mathscr{E}$.
	An \emph{orientation} on $\mathscr{U}$ is a pair of orientations on $U$ and $\mathscr{E}$.
\end{definition}

There is a notion of \emph{oriented-preserving embeddings} of  Kuranishi charts (\cite{FOOO20book}, Definition 3.2, 3.4) and of \emph{coordinate changes }of Kuranishi charts (\emph{op. cit.}, Definition 3.6). A \emph{Kuranishi structure} $\wh{\mathscr{U}}$ (\emph{op. cit.}, Definition 3.9) on a space $X$ is a collection of Kuranishi charts $\mathscr{U}_p=(U_p,\mathscr{E}_p,s_p,\psi_p)$ at each $p\in X$ together with coordinate changes between overlapping charts satisfying certain consistency. The pair $(X,\wh{\mathscr{U}})$ is called a \emph{Kuranishi space} (\emph{op. cit.}, Definition 3.11).

\begin{definition}[\cite{FOOO20book}, Definition 3.40]
	Let $\mathscr{U}=(U,\mathscr{E},s,\psi)$ be a Kuranishi chart and $M$ be a $C^\infty$-manifold.
	A \emph{strongly continuous map} from $\mathscr{U}$ to $M$ is a continuous map $f\colon U\to M$. It is called \emph{strongly smooth} if $f\colon U\to M$ is smooth.  It is called \emph{weakly submersive} if $f$ is a submersion.
	
	A strongly continuous (\textit{resp.} strongly smooth, weakly submersive) map $\wh{f}$ from a Kuranishi space $(X,\wh{\mathscr{U}})$ to $M$  assigns a continuous (\textit{resp.} smooth, submersive) map $f_p\colon U_p\to Y$ to each $p\in X$, satisfying compatibility with chart transition maps.
\end{definition}

To work with moduli spaces of pseudo-holomorphic discs with Lagrangian boundary conditions, we need to introduce Kuranishi spaces with boundaries and corners, as well as the notions of \emph{admissible Kuranishi charts} and \emph{admissible Kuranishi spaces}, where admissibility roughly means that the coordinate changes satisfy exponential decay estimates near the boundaries. Notions like embeddings and coordinate changes also have ``admissible'' versions. See \cite{FOOO20book}, Section 25.

Maps between admissible Kuranishi spaces with corners are required to respect the corner stratifications. There are notions of \emph{corner-stratified smooth maps} and \emph{corner-stratified wewak submersions}. See \cite{FOOO20book}, Section 26.

An \emph{isomorphism} of admissible Kuranishi spaces $(X_1,\wh{\mathscr{U}_1})$ and   $(X_2,\wh{\mathscr{U}_2})$ is a pair of admissible embeddings $(X_1,\wh{\mathscr{U}_1})\to (X_2,\wh{\mathscr{U}_2})$ and $(X_2,\wh{\mathscr{U}_2})\to (X_1,\wh{\mathscr{U}_1})$ whose compositions are the identity embeddings (\cite{FOOO20book}, Definition 4.24).

\subsubsection{Existence of Kuranishi structures}
\label{Subsc: Kuranishi}

The theorem we will use is Theorem 7.20 of \cite{Irie2}. We will not reproduce the entire statement of the theorem here due to its length, but we will recall the main points and properties below.

Let $L$ be a closed, connected, oriented, and spin manifold of dimension $n$, together with a Lagrangian embedding into $\C^n$ equipped with the standard symplectic structure.
Take the standard complex structure $J$ on $\C^n$.
Then there exists $\varepsilon>0$ such that $2\varepsilon$ is less than the minimal energy of non-constant $J$-holomorphic discs with boundaries on $L$.

Take a Hamiltonian $H\in C^\infty_c(\C^n \times [0,1]_t)$ satisfying the displaceability Assumption \ref{Asmp: Disp Hamil}.
Also recall that $\norm{H}$ is the Hofer norm \eqref{Eq: HoferNorm} of $H$.
Let $U\in \Z_{> 0}$ be such that $\varepsilon(U-1) \geq 2\norm{H}$.

We consider the moduli spaces $\ms{M}_{k+1}(a), \ms{N}^{\geq 0}_{k+1}(a), \ms{N}^{0}_{k+1}(a)$	 for $k\in \Z_{\geq 0}$ and $\beta \in H_1(L;\Z)$. These are defined in detail in section 7.2.1 and 7.2.2 of \cite{Irie2}, and are Gromov compatifications/bordifications of the uncompactified moduli spaces in section \ref{Sec: SketchPf}.

Define \[ \mathfrak{G}_0:= \{ \beta \in H_1(L;\Z)\mid \ms{M}_{k+1}(\beta) \neq \emptyset \textup{ for some }k\}, \] and let  $\mathfrak{G}\subset H_1(L;\Z)$ be the submonoid generated by $\mathfrak{G}_0$. Then by Gromov compactness, $\mathfrak{G}$ satisfies the condition in Definition \ref{Defn: monoid}, i.e. it is a monoid of curve classes (see e.g. (3.1.8) in \cite{FOOO1}).
Similarly, define 
\[
\mathfrak{N}_0:= \{\eta\in H_1(L;\Z) \mid \ms{N}^{\geq 0}_{k+1}(\eta)\neq \emptyset\textup{ for some }k\},
\]
and let $\mathfrak{N}\subset H_1(L;\Z)$ be the $\mathfrak{G}$-module generated by $\mathfrak{N}_0$. Then $\mathfrak{N}$ satisfies the condition in Definition \ref{Defn: module}, i.e. it is a module of $H$-perturbed curve classes over $\mathfrak{G}$.

\begin{theorem}
	[\cite{Irie2}, Theorem 7.20]\label{Thm: IrieKuranishi}
	For each $k\in \Z_{\geq 0}$, $m\in \Z_{\geq 0}$, and $P\in \{\{m\}, [m,m+1]\}$, there exist the following data:
	\begin{enumerate}
		\item Compact, oriented, admissible Kuranishi spaces \begin{align*}
			\ms{M}_{k+1}(\beta:P),\quad &\textup{ where }\beta \in \mathfrak{G}, \,\,E(\beta) < \varepsilon(m+1-k) ;\\
			 \ms{N}_{k+1}^0(\eta: P), \quad &\textup{ where }\eta \in \mathfrak{N}, \,\, E(\eta) < \varepsilon(m-1-k);\\
			 \ms{N}^{\geq 0}_{k+1}(\eta: P), \quad &\textup{ where }\eta \in \mathfrak{N}, \,\, E(\eta) < \varepsilon(m-k-U),
		\end{align*}
		whose underlying topological spaces are \[ P \times \ms{M}_{k+1}(\beta), \quad P\times \ms{N}^0_{k+1}(\eta), \quad P\times \ms{N}^{\geq 0}_{k+1}(\eta)\]respectively. The (virtual) dimensions of these Kuranishi spaces are \begin{align*}
			 \dim \ms{M}_{k+1}(\beta:P) &= \mu(\beta) + n + k - 2 + \dim P; \\
		\dim \ms{N}^0_{k+1}(\eta: P) &= \mu(\eta) + n + k + \dim P;\\
		\dim  \ms{N}^{\geq 0}_{k+1}(\eta:P) &= \mu(\eta) + n + k + 1 + \dim P.
		\end{align*}		
		\item Corner stratified strongly smooth maps (\cite{FOOO20book}, Definition 26.6 (1)):\begin{align*}
			\ev^{\ms{M},P}&\colon \ms{M}_{k+1}(\beta:P)\to P\times L^{k+1};\\
			\ev^{\ms{N}^0,P}&\colon \ms{N}^0_{k+1}(\eta:P)\to P\times L^{k+1};\\
			\ev^{\ms{N}^{\geq 0},P}&\colon \ms{N}^{\geq 0}_{k+1}(\eta:P)\to P\times L^{k+1},
		\end{align*}
		such that their underlying maps are \[ \id_P\times \ev^\ms{M},\quad \id_P \times \ev^{\ms{N}^0},\quad \id_P\times \ev^{\ms{N}^{\geq 0}} \]respectively. We require that the maps $(\id_P\times \pr_0)\circ \ev$ are corner stratified weak submersions (\cite{FOOO20book}, Definition 26.6 (2)), for each of these evaluation maps, where $\pr_0\colon L^{\times (k+1)}\to L$ is the projection to the first factor.
		\item An isomorphism of admissible Kuranishi structures \begin{align}\label{Eq: EnergyZeroKura}
			\ms{M}_{k+1}(0:P)\cong P\times L \times D^{k-2} 
		\end{align}for each $k\geq 2$, so that $\ev^{\ms{M},P}\colon \ms{M}_{k+1}(0:P)\to P\times L^{\times(k+1)}$ coincides with $\pr_P\times (\pr_L)^{k+1}$. Here $D^{k-2}$ is identified with the Stasheff cell (\cite{FukayaOh}). 
		\item Various boundary and corner compatibility conditions, spelled out in parts (iv)-(vi) in Theorem 7.20 of \cite{Irie2}.
	\end{enumerate}
\end{theorem}

Furthermore, we need maps from these moduli spaces to $\mathscr{L}^{k+1}L$. In \cite{Irie2}, this is done in sections 7.3 and 7.4. Our situation is much simpler because our $\mathscr{L}^{k+1} L$ is simpler than the infinite-dimensional ones in \cite{Irie2}.
The following proposition is an analogue of Proposition 7.26 in \cite{Irie2}.
For example, we can obtain \emph{strongly smooth} maps in the statement, due to the simpler definition of $\mathscr{L}^{k+1} L$.
\begin{proposition}\label{Prop: KuraStrSmMap}
	For every $k\in \Z_{\geq 0}$, $m\in \Z_{\geq 0}$, and $P\in \{\{m\}, [m,m+1]\}$, one can define strongly smooth maps 
	\begin{align*}
			\ev^\mathscr{M}\colon \mathscr{M}_{k+1}(\beta:P) \to P\times \mathscr{L}^{k+1}(\beta),\quad &\textup{ where }\beta \in \mathfrak{G}, \,\,E(\beta) < \varepsilon(m+1-k) ;\\
			 \ev^{\mathscr{N}^0}\colon \mathscr{N}_{k+1}^0(\eta: P) \to P\times \mathscr{L}^{k+1}(\eta), \quad &\textup{ where }\eta \in \mathfrak{N}, \,\, E(\eta) < \varepsilon(m-1-k);\\
			 \ev^{\mathscr{N}^{\geq 0}}\colon \mathscr{N}^{\geq 0}_{k+1}(\eta: P) \to P\times \mathscr{L}^{k+1}(\eta), \quad &\textup{ where }\eta \in \mathfrak{N}, \,\, E(\eta) < \varepsilon(m-k-U),
	\end{align*}
	so that the diagrams analogous to those in Proposition 7.26 of \cite{Irie2} commute.
	Moreover, when $\beta=0$, under the identification \eqref{Eq: EnergyZeroKura}, the map $\ev^{\ms{M}}$ is given by \begin{align}\label{Eq: EvEnergyZero}
		\ev^\ms{M}\colon \ms{M}_{k+1}(0:P)\cong P\times L \times D^{k-2} &\to P\times \ms{L}^{k+1}(\beta)\\
		(s, y, z) &\mapsto(s, ([\underbrace{\underline{y}],\dots, [\underline{y}]}_{(k+1) \textup{ times}}))\nonumber
	\end{align} 
	where $[\underline{y}]\in \Pi_1 L$ denotes the constant path at $y\in L$ (see Definition \ref{Defn: PathFundGroupoid}).
\end{proposition}

\begin{proof}
	We give the proof for $\ms{M}_{k+1}(\beta:P)$ only since the others are analogous.

	We need the explicit description of the Kuranishi charts $\ms{M}_{k+1}(\beta:P)$ from Lemma 7.22 of \cite{Irie2}. 
	We will borrow terminologies from \cite{Irie2}, section 7.2.2, on e.g. decorated rooted ribbon trees.
	Let $\ms{M}\ms{M}_{k+1}(\beta)$ be the set of tuples $(u,z_0,z_1,\dots, z_k)$ where $u\colon (\mathbb{D},\partial \mathbb{D})\to (\C^n,L)$ is a $C^\infty$-map such that $\bar\partial u = 0$ on a neighborhood of $\partial \mathbb{D}$ and $[u]=\beta$, and $z_0,\dots, z_k\in \partial \mathbb{D}$ are distinct points aligned in anti-clockwise order. Let $p\in \ms{M}_{k+1}(\beta:P)$. Then there is a Kuranishi chart  $\mathscr{U}_p=(U_p,\mathscr{E}_p,s_p,\psi_p)$ at $p$ under the Kuranishi structure in Theorem \ref{Thm: IrieKuranishi}, where for a decorated rooted ribbon tree $(T,B)\in \mathcal{G}(k+1 : \beta)$  such that 
	\begin{align*}
		p\in P\times \left(\prod_{e\in C_{1,\textup{int}}(T)} L\right) \fiberprod{\Delta}{\ev_{\textup{int}}} \left( \prod_{v\in C_{0,\textup{int}}(T)} \mathring{\ms{M}}_{k_v+1}(B(v))\right)
	\end{align*}
	then $U_p$ can be embedded into 
	\begin{align*}
		\bigsqcup_{(T',B')} P\times \left(\prod_{e\in C_{1,\textup{int}}(T')} L\right) \fiberprod{\Delta}{\ev_{\textup{int}}} \left( \prod_{v\in C_{0,\textup{int}}(T')} \ms{M}\ms{M}_{k_v+1}(B'(v))\right)
	\end{align*}
	where $(T',B')$ runs over all reductions of $(T,B)$. 
	
	Then we define a map 
		\[\ev_p^\mathscr{M}\colon U_p\to P\times \mathscr{L}^{k+1}(\beta)\]
	as follows. Each $x\in U_p$ is identified with an element \[
		x=\left(\pi, (u^v,z_0^v,\dots, z_{k_v}^v)_{v}\right) \in P\times \left(\prod_{e\in C_{1,\textup{int}}(T')} L\right) \fiberprod{\Delta}{\ev_{\textup{int}}} \left( \prod_{v\in C_{0,\textup{int}}(T')} \ms{M}\ms{M}_{k_v+1}(B'(v))\right),
	\]
	where $(T',B')$ is a reduction of $(T,B)$.
	For each $v\in C_{0,\textup{int}}(T')$, we define $\ev(u^v, z_0^v,\dots, z_{k_v}^v)\in \mathscr{L}^{k_v+1}(B'(v))$ as follows. For each $j=0,\dots, k_v$, let \[\gamma_j^{u^v}=\left(\ev_j(u), \ev_{j+1}(u), \left[u|_{[z_j^v,z_{j+1}^v]}\right]\right)\in \Pi_1L\] in the fundamental groupoid, where $j+1$ is taken modulo $k_v+1$,  $u|_{[z_j^v,z_{j+1}^v]}\in \mathcal{P}_{\ev_j(u),\ev_{j+1}(u)}$ is the path in $L$ given by restricting $u$ to the arc from $z_j^v$ to $z_{j+1}^v$ in $\partial \mathbb{D}$ and $\left[u|_{[z_j^v,z_{j+1}^v]}\right]$ is its homotopy class.
	Then define \[\ev(u^v,z_0^v,\dots, z^v_{k_v}):= \left(\gamma_0^{u^v},\dots \gamma_{k_v}^{u^v}\right)\in \mathscr{L}^{k_v+1}(B'(v)).\]
	Finally, define \begin{align*}
		\left(\prod_{e\in C_{1,\textup{int}}(T)}P\times L\right)\fiberprod{\Delta}{\ev_{\textup{int}}} \left( \prod_{v\in C_{0,\textup{int}}(T)} P\times \mathscr{L}^{k_v+1}(B(v))\right) \to P\times \mathscr{L}^{k+1}(\beta),
	\end{align*}
	where the fibre product is taken over $\prod_{e\in C_{1,\textup{int}}(T)}(P\times L)^{\times 2}$,
	by concatenating paths using \eqref{Eq: Concat}. 
	We then compose  it with 
	\begin{align*}
				U_p&\to \left(\prod_{e\in C_{1,\textup{int}}(T)}P\times L\right)\fiberprod{\Delta}{\ev_{\textup{int}}} \left( \prod_{v\in C_{0,\textup{int}}(T)} P\times \mathscr{L}^{k_v+1}(B(v))\right);\\
				x&=\left(\pi, (u^v,z_0^v,\dots, z_{k_v}^v)_{v}\right) \mapsto \left(\pi,  \gamma_j^{u^v}\right)_v
	\end{align*}
	to get the desired map $\ev_p^\mathscr{M}\colon U_p\to P\times \mathscr{L}^{k+1}(\beta)$.
	The family of maps $\left( \ev_p^\mathscr{M}\right)_{p\in \mathscr{M}_{k+1}(\beta:P)}$ is compatible with coordinate changes by the construction of the Kuranishi structure (see e.g. Lemma 7.28 of \cite{Irie2}, although in our case this is a lot easier since we don't have to deal with reparametrizations of loops and paths).
	Moreover, the smoothness of $\ev_p^\mathscr{M}$ follows from the fact that the smooth structure on $\mathscr{L}^{k+1}$ is given  so that the evaluation map
	\[
		\ev_0\times \dots \times \ev_k\colon \mathscr{L}^{k+1}\to L^{\times (k+1)}
	\]
	is a smooth covering map, and that the composition \[
		U_p \xrightarrow{\ev^{\mathscr{M}}}  P \times \mathscr{L}^{k+1} \xrightarrow{\id_P\times (\ev_0\times \dots \times \ev_k)} P\times L^{\times (k+1)}
	\]
	is equal to $\ev^{\mathscr{M}, P}$ in part (2) of Theorem \ref{Thm: IrieKuranishi}, which is smooth.
	The commutativity of the diagrams are analogous to that in \cite{Irie2}.
	The statement about energy-zero moduli spaces also follow directly from the description.
\end{proof}

\subsection{CF perturbations}
Again we explain the notion of CF perturbation in one chart.
Consider a Kuranishi chart $\mathscr{U} = (U,\mathscr{E},s,\psi)$ on $X$ as in Definition \ref{Defn: KuranishiChart}. 

\begin{definition}[\cite{FOOO20book}, Definition 7.16; also see section 8.1.1 of \cite{Irie2}]\label{Defn: CF-perturb} 
	A (representative of) a \emph{CF-perturbation} $\mf{S}=(\mf{S}^\varepsilon)_{0<\varepsilon\leq 1}$ of $\ms{U}$ is the data of $(\mf{V}_\mf{r}, \ms{S}_\mf{r})_{\mf{r}\in\mf{R}}$ where 
	\begin{itemize}
		\item $\mf{V}_\mf{r}=(V_\mf{r}, E_\mf{r}, \phi_\mf{r}, \wh{\phi}_\mf{r})$ is a manifold chart of $(U,\ms{E})$ such that $(\phi_\mf{r}(V_\mf{r}))_{\mf{r}\in\mf{R}}$ covers $U$. Let $s_\mf{r}\colon V_\mf{r}\to E_\mf{r}$ be the pullback of $s$ by $\phi_\mf{r}$;
		\item $\ms{S}_\mf{r}=(W_\mf{r}, \eta_\mf{r}, \{\mf{s}_\mf{r}^\varepsilon\}_\varepsilon)$ is a CF-perturbation of $\ms{U}$ on $\mf{V}_\mf{r}$:
		\begin{enumerate}
			\item $W_\mf{r}$ is an open neighborhood of 0 in a finite-dimensional oriented real vector space;
			\item $\mf{s}_\mf{r}^\varepsilon\colon V_\mf{r}\times W_\mf{r}\to E_\mf{r}$   is a family of maps,  depending smoothly on $\varepsilon$, such that $\mf{s}^\varepsilon_\mf{r}$ is transversal to $0$ for each $\varepsilon\in (0,1]$, and \begin{align*}
				\lim_{\varepsilon\to 0} \mf{s}^\varepsilon_\mf{r}(y,\xi) = s_\mf{r}(y)
			\end{align*} 
			in compact $C^1$-topology on $V_\mf{r}\times W_\mf{r}$;
			\item $\eta_\mf{r}\in \ms{A}_c^{\dim W_\mf{r}}(W_\mf{r})$ is a differential form such that $\int_{W_\mf{r}} \eta_\mf{r} = 1$.
		\end{enumerate}
	\end{itemize}
\end{definition}

For each $\varepsilon\in (0,1]$, we write
\begin{align*}
	\ms{S}_\mf{r}^\varepsilon = (W_\mf{r}, \eta_\mf{r}, \mf{s}_\mf{r}^\varepsilon),\quad 
	\mf{S}^\varepsilon=\{\mf{V}_\mf{r}, \ms{S}^\varepsilon_\mf{r}\}_{\mf{r}\in\mf{R}}.
\end{align*}

\begin{definition}[\cite{FOOO20book}, Definition 7.9, Definition-Lemma 7.26]
	Let $f\colon U\to M$ be a smooth submersion to a smooth manifold $M$.
	We say $f$ is \emph{strongly submersive} with respect to $\mf{S}$ if for each $\mf{r}\in \mf{R}$, the map \begin{align*}
		f\circ \phi_\mf{r}\circ \textup{pr}_{V_\mf{r}}\colon (\mf{s}_\mf{r}^\varepsilon)^{-1}(0) \to L
	\end{align*}
	is a submersion for every $\varepsilon\in (0,1]$.
\end{definition}

There are definitions for an admissible CF-perturbation  of an admissible Kuranishi chart and for a map $f\colon U\to M$ to be a stratified strong submersion with respect to $\mf{S}$.

\bigskip
Let $X$ be one of the moduli spaces in (1) of  Theorem \ref{Thm: IrieKuranishi}. Assign $\tau(X)\in (1/2,1)$ as in Remark 7.23 of \cite{Irie2} (as well as section 3 in \cite{IrieErratum}).
The following plays the role of Theorem 7.33 (and Remark 7.34) of \cite{Irie2}.

\begin{theorem}\label{Thm: ExistCF} 
	Let $X$ be one of the moduli spaces in (1) of  Theorem \ref{Thm: IrieKuranishi} where $P=\{m\}$, with Kuranishi structure $(X,\wh{\mathscr{U}})$, together with the corner-stratified admissible strongly smooth maps $\wh{\ev} \colon (X,\wh{\ms{U}})\to L^{k+1}$, such that $\pr_0\circ \wh{\ev}\colon (X,\mathscr{U})\to L$ is a stratified weak submersion.  Then, upon shrinking $\tau(X)$ to $\tau':=\tau'(X)\in (0,\tau(X))$, there exist the following data:
	\begin{itemize}
		\item A $\tau'$-collared Kuranishi structure $\wh{\mathscr{U}^+}$ on $X$, which is a thickening of $\wh{\mathscr{U}}$ (see section 5.2 of \cite{FOOO20book});
		\item An isomorphism of $\tau'$-collared Kuranishi structures $\wh{\ms{U}^+}|_{\wh{S}_\ell (X)}\cong \wh{\ms{U}_\ell^+}$ for every $\ell \in \Z_{\geq 1}$;
		\item A $\tau'$-collared CF-perturbation $\wh{\mf{S}^+}$ of $(X,\wh{\ms{U}^+})$ such that $\wh{\mf{S}^+}|_{\wh{S}_\ell(X)}$ coincides with $\wh{\mf{S}^+_\ell}$ via the isomorphism of Kuranishi spaces $\wh{\ms{U}^+}|_{\wh{S}_\ell(X)}\cong \wh{\ms{U}_\ell^+}$;
		\item A $\tau'$-collared admissible map $\wh{\ev^+}\colon (X,\wh{\ms{U}^+})\to \ms{L}^{k+1}$ such that: 
		\begin{itemize}
			\item $\wh{\ev^+}$ coincides with $\wh{\ev}$ under the KK-embedding $\wh{\ms{U}}\to \wh{\ms{U}^+}$;
			\item $\pr_0\circ \wh{\ev^+}\colon (X,\wh{\ms{U}^+})\to L$ is a stratified strong submersion with respect to $\wh{\mf{S}^+}$;
		\end{itemize}
			\item For $\beta=0$ and $k>2$, the structures associated to $X=\ms{M}_{k+1}(0:P)$ are isomorphic to $P\times L \times (D^{k-2})^{\boxplus \tau'}$ where $D^{k-2}$ is the Stasheff cell, and  the CF perturbation is trivial: that is, in the notations of Definition \ref{Defn: CF-perturb}, the Kuranishi structure on $\ms{M}_{k+1}(0:P)$ is given by a single Kuranishi chart, covered by a set of manifold charts $\{\mf{V}_\mf{r}=(V_\mf{r}, E_\mf{r}, \phi_\mf{r}, \wh{\phi}_\mf{r})\}_{\mf{r}\in\mf{R}}$ with $E_\mf{r}=\{0\}$, and the CF-perturbation $\ms{S}_\mf{r}=(W_\mf{r},\eta_\mf{r},\{\mf{s}_\mf{r}^\varepsilon\}_\varepsilon)$ is given by $W_\mf{r}=\{0\}$,  $\eta=1\in \ms{A}_c^0(W)$, and $\mf{s}_\mf{r}^\varepsilon=0$ for all $\varepsilon\in(0,1]$.
	\end{itemize}
\end{theorem}
As in the  Remark 7.34 of \cite{Irie2}, there is an analogous statement for the cases where $P=[m,m+1]$ and \[ \wh{\ev}\colon (X,\wh{\mathscr{U}})\to [a,b]^{\boxplus \tau}\times \mathscr{L}^{k+1}  \]which we do not spell out.

\begin{proof}
	By Proposition \ref{Prop: KuraStrSmMap}, we have a strongly smooth map $X\to \mathscr{L}^{k+1}(\beta)$. 
	By Lemma-Definition 17.38 and Lemma 17.40 (3) of \cite{FOOO20book}, this map extends to a 1-collared strongly smooth map $X^{\boxplus 1}\to \ms{L}^{k+1}(\beta)$. 
	Then successively apply Proposition 17.78 and Proposition 17.81 of \cite{FOOO20book} to obtain Kuranishi structures $\wh{\ms{U}^+}$ and CF perturbations $\wh{\mf{S}^+}$.  
	In the case $\beta=0$, the moduli spaces $\ms{M}_{k+1}(0:\{m\})$ are not the boundaries of other moduli spaces. We do not need to change the Kuranishi structure for the thickening because the obstruction bundle is trivial, and we can choose the CF perturbations to be trivial because the identity map $L\to L$ is already a submersion.
\end{proof}

\subsection{Pushforward of differential forms}
\subsubsection{Chain model of \texorpdfstring{$[-1,1] \times \mathscr{L}^{k+1}$}{[-1,1] x \mathscr{L}^{k+1}}}
For each $k\in\Z_{\geq 0}$ and $a\in H_1(L;\Z)$, define a chain complex $\ol{C}_*^\dR(\mathscr{L}^{k+1}(a))$ in exactly the same way as section 4.4 of \cite{Irie2}. 
Roughly speaking, an element in $\ol{C}_*^\dR(\mathscr{L}^{k+1}(a))$ is represented by a quintuple $(U,\varphi,\tau_+,\tau_-,\omega)$ where 
\begin{itemize}
	\item $U\in \mathscr{U}$ is an oriented submanifold of $\R^N$ for some $N$.
	\item $\varphi:=(\varphi_\R,\varphi_\mathscr{L})\colon U\to \R\times \mathscr{L}^{k+1}(a)$, such that $\varphi_\R$ and $\varphi_{\mathscr{L}}$ are $C^\infty$-smooth (notice that unlike in \cite{Irie2}, here $\mathscr{L}^{k+1}(a)$ is an ordinary finite-dimensional $C^\infty$-manifold and the notion of $C^\infty$-smooth is just the usual notion), and \[ U\to \R\times L;\quad u\mapsto (\varphi_\R(u), \ev_0\circ \varphi_\mathscr{L}(u))\] is a submersion. For each interval $I\subset \R$, we denote $U_I:= (\varphi_\R)^{-1}(I)$. 
	\item $\tau_+\colon U_{\geq 1}\to \R_{\geq 1}\times U_1$ is a diffeomorphism such that \[ \varphi|_{U_{\geq 1}} = (\iota_{\geq 1} \times \varphi_{\mathscr{L}}|_{U_1}) \circ \tau_+ \]where $\iota_{\geq 1}\colon \R_{\geq 1}\hookrightarrow \R$ is the inclusion map.
	\item $\tau_-\colon U_{\leq -1}\to \R_{\leq -1}\times U_{-1}$ is a diffeomorphism such that \[ \varphi|_{U_{\leq -1}} = (\iota_{\leq -1} \times \varphi_{\mathscr{L}}|_{U_{-1}}) \circ \tau_- \]where $\iota_{\leq -1}\colon \R_{\leq -1}\hookrightarrow \R$ is the inclusion map.
	\item $\omega\in \mathscr{A}^{\dim U -*+1}(U)$ such that $\omega|_{U_{[-1,1]}}$ is compactly supported, and \[ \omega|_{U_{\geq 1}} = (\tau_+)^*(1\times \omega|_{U_1}),\quad \omega|_{U_{\leq -1}} = (\tau_-)^*(1\times \omega|_{U_{-1}}). \]
	Denote the vector space of all such differential forms on $U$ as $\ms{A}^*(U,\varphi,\tau_+,\tau_-).$
\end{itemize}
Two such elements are identified if there is a submersion, respecting the ends of $\R$, between the domains of the de Rham chains, such that the differential forms pushforward from one to the other. See Section 4.4 of \cite{Irie2}.

There are naturally defined chain maps
\begin{align*}
	i\colon C_*^\dR(\mathscr{L}^{k+1}(a))\to \ol{C}_*^\dR(\mathscr{L}^{k+1}(a))
\end{align*}
\begin{align*}
	e_+\colon \ol{C}_*^\dR(\mathscr{L}^{k+1}(a))\to C_*^\dR(\mathscr{L}^{k+1}(a)),\quad e_-\colon \ol{C}_*^\dR(\mathscr{L}^{k+1}(a))\to C_*^\dR(\mathscr{L}^{k+1}(a))
\end{align*}
satisfying that
\begin{itemize}
	\item $e_+\circ i = e_-\circ i = \id$;
	\item $(e_+,e_-)\colon \ol{C}_*^\dR(\mathscr{L}^{k+1}(a)) \to C_*^\dR(\mathscr{L}^{k+1}(a))\oplus C_*^\dR(\mathscr{L}^{k+1}(a))$ is surjective;
	\item $i\circ e_+$ and $i\circ e_-$ are chain homotopic to $\id$.
\end{itemize}

For each $k\in \Z_{\geq 1}$, $k'\in \Z_{\geq 0}$, $j=1,\dots, k$ and $a,a'\in H_1(L;\Z)$, the following fibre product operation is also defined in Section 4 of \cite{Irie2}:
\begin{align*}
	\circ_j\colon \ol{C}^\dR_{n+d}(\mathscr{L}^{k+1}(a))\otimes \ol{C}_{n+d'}^\dR(\mathscr{L}^{k'+1}(a')) \to \ol{C}^\dR_{n+d'+d'}(\mathscr{L}^{k+k'}(a+a'));\quad x\otimes y \mapsto x\circ_j y.
\end{align*}

We then consider 
\begin{align*}
	\CCFr:= \bigoplus_{a\in H_1(L;\Z)} \prod_{k\in \Z_{\geq 0}} \ol{C}^\dR_{*+n+\mu(a)+k-1} (\mathscr{L}^{k+1}(a))
\end{align*}
and its completion $\wh{\CCFr}$, and define a dg Lie algebra structure on it in the same way as $\CFr$.
Then there are naturally defined morphisms of dg Lie algebras
\begin{align*}
	i\colon \CFr\to \CCFr,\quad e_+\colon \CCFr\to \CFr, \quad e_-\colon \CCFr\to \CFr
\end{align*}
satisfying that 
\begin{itemize} 
	\item $e_+\circ i = e_-\circ i = \id_{\CFr}$;
	\item $(e_+,e_-)\colon \CCFr\to \CFr\oplus \CFr$ is surjective;
	\item $i\circ e_+$ and $i\circ e_-$ are chain homotopic to $\id_{\CCFr}$. One can take chain homotopies to respect the decompositions over $(a,k)\in H_1(L;\Z)\times \Z_{\geq 0}$.
\end{itemize}

\begin{definition}\label{Defn: IntervalE=0}
	Define \[\ol{\LL}\in \ol{\Cfr{-1}}(0) = \prod_{k\in \Z_{\geq 0}} \ol{C}^\dR_{n +k-2} (\ms{L}^{k+1}(0)) \]
	 as follows:
	\begin{itemize}
		\item For $k=2$, consider the map $\R\times L\xrightarrow{\varphi} \R\times \ms{L}^3(0)$ defined by $(r,y)\mapsto (r, (\underline{y},\underline{y},\underline{y}))$; then for any interval $i\subset \R$, $\varphi_\R^{-1}(I)\cong I\times L$, and define $\tau_+,\tau_-$ to be the obvious diffeomorphisms. Set $\ol{\LL}(0,2):= (-1)^{n+1}[(\R\times L\to \R\times \ms{L}^3(0);1\in \ms{A}^0(\R\times L, \varphi, \tau_+,\tau_-))] \in \ol{C}^\dR_{n}(\ms{L}^{3}(0))$;
		\item For all $k\neq 2$, set $\ol{\LL}(0,k)=0$.
	\end{itemize}
\end{definition}

Then it follows that
\begin{align}\label{Eq: [0,1]xL}
	e_+(\ol{\LL}) = \LL.	
\end{align}

\subsubsection{Strongly smooth maps from a K-space with a CF-perturbation gives a de Rham chain}

The following is an analogue of Theorem 7.9 in \cite{Irie2} (where the meaning of $\mathscr{L}^{k+1}$ differs: in our case, $\mathscr{L}^{k+1}$ is a finite-dimensional smooth manifold whereas in \cite{Irie2} it is an infinite-dimensional differentiable space).
For the definition of a map $f\colon (X,\wh{U})\to \ms{L}^{k+1}$ to be admissible, see \cite{Irie2} Definition 7.7 with the modifications of the meaning of $\ms{L}^{k+1}$ taken into account.

\begin{theorem}[\cite{Irie2} Theorem 7.9]\label{Thm: CFgivesDR}
	Let 
	\begin{enumerate}
		\item $(X,\wh{\mathscr{U}})$ be a compact, oriented, admissible Kuranishi space of dimension $d$;
		\item $\wh{f}\colon (X,\wh{\mathscr{U}})\to \mathscr{L}^{k+1}$ be an admissible map;
		\item $\wh{\omega}$ be an admissible differential form on $(X,\wh{\mathscr{U}})$;
		\item  $\wh{\mathfrak{S}}$ be an admissible CF-perturbation of $(X,\wh{\mathscr{U}})$ such that $\wh{\mathfrak{S}}$ is transversal to 0, and such that $\ev_0\circ \wh{f}\colon (X,\wh{\mathscr{U}})\to L$ is a stratified strong submersion with respect to $\wh{\mathfrak{S}}$.
	\end{enumerate}
	Then one can define \[ \wh{f}_*(X,\wh{\mathscr{U}}, \wh{\omega}, \wh{\mathfrak{S}}^\varepsilon) \in C^\dR_{d-|\wh{\omega}|}(\mathscr{L}^{k+1}) \]for sufficiently small $\varepsilon>0$, so that Stoke's formula (analogous to Theorem 7.11 of \cite{Irie2}) and the fibre product formula (analogous to Theorem 7.12 of \cite{Irie2}) hold.
\end{theorem}

We also need a version of this for admissible K-spaces over an interval. For the statement see Theorem 7.14 of \cite{Irie2}.
For the proofs see section 8 of \cite{Irie2}.

For later purposes, we recall part of the explicit construction from section 8.1.2 of \cite{Irie2}.
Suppose $X$ can be covered by a single Kuranishi chart with boundary. Given the data of
\begin{enumerate}
	\item $\ms{U}=(U,\ms{E},s,\psi)$ an admissible Kuranishi chart on $X$;
	\item $f\colon U\to \ms{L}^{k+1}$ an admissible map;
	\item $\omega\in \ms{A}_c^*(U)$ a admissible differential form;
	\item $\mf{S}=(\mf{S}^\varepsilon)_{0<\varepsilon\leq 1}$ an admissible CF-perturbation of $\textup{supp }\omega$ such that $\ev_0\circ f\colon U\to L$ is a stratified strong submersion with respect to $\mf{S}$,
\end{enumerate}
and for each $\varepsilon\in (0,1]$, 
the de Rham chain $f_*(\ms{U},\omega,\mf{S}^\varepsilon)$ is defined as follows.
Let $(\mf{V}_\mf{r}, \ms{S}_\mf{r})_{\mf{r}\in\mf{R}}$ be a representative of the CF-perturbation $\mf{S}$, and $(\chi_\mf{r})_{\mf{r}\in\mf{R}}$ be a partition of unity subordinate to $(\phi_\mf{r}(V_\mf{r}))_{\mf{r}\in\mf{R}}$. 
We first define $f_*(\ms{U},\chi_\mf{r}\omega, \mf{V}_\mf{r}, \ms{S}_\mf{r}^\varepsilon)$ as follows. Let $D:=\dim U$. For a manifold chart $V_\mf{r}$ in $\mf{V}_\mf{r}$, given as an open neighborhood of $(t_1,\dots, t_D)\in (\R_{\geq 0})^D$, define the \emph{retraction map} (terminology from Definition 17.7 of \cite{FOOO20book})
\begin{align*}
	\ms{R}\colon \R^D\to (\R_{\geq 0})^D;\quad (t_1,\dots, t_D)\mapsto (t_1',\dots, t_D'),\quad \textup{where } t_i':= \begin{cases}
		t_i & t_i\geq 0\\
		0, & t_i<0
	\end{cases}.
\end{align*}
Take a cutoff function $\kappa\in C^\infty(\R,[0,1])$ such that $\kappa\equiv 1$ on a neighborhood of $\R_{\geq 0}$ and $\kappa \equiv 0$ on a neighborhood of $\R_{\leq -1}$. Define 
\begin{enumerate}
	\item $\ol{V}_{\mf{r}}:= \ms{R}^{-1}(V_\mf{r})$, $\ol{E}_\mf{r}:=\ms{R}^*E_\mf{r}$;
	\item $\ol{\mf{s}}_\mf{r}^\varepsilon:= (\ms{R}|_{\ol{V}_\mf{r}}\times \id_{W_\mf{r}})^*(\mf{s}_\mf{r}^\varepsilon)$;
	\item $\ol{f}_\mf{r}:= f\circ \phi_\mf{r}\circ \ms{R}|_{\ol{V}_\mf{r}}$;
	\item $\ol{\chi_\mf{r}\omega}(t_1,\dots, t_D):= \kappa(t_1)\cdots \kappa(t_D)\cdot  (\phi_\mf{r}\circ \ms{R}|_{\ol{V}_\mf{r}})^*(\chi_\mf{r}\omega)$.
\end{enumerate}
Then define 
\begin{align}\label{Eq: CFpushforward}
	f_*(\ms{U}, \chi_\mf{r}\omega, \mf{V}_\mf{r}, \ms{S}_\mf{r}^\varepsilon) := (-1)^\dagger\big((\ol{\mf{s}}_\mf{r}^\varepsilon)^{-1}(0), \ol{f}_\mf{r}\circ \pr_{\ol{V}_\mf{r}}, \pr_{\ol{V}_\mf{r}}^*(\ol{\chi_{\mf{r}}\omega}) \wedge \pr^*_{W_\mf{r}}(\eta_\mf{r})\big),
\end{align}
where 
\begin{align*}
	\dagger := \dim W_\mf{r} \cdot (\rank \,\ms{E} + |\omega|).
\end{align*}

\subsection{Construction of low-energy approximate solutions}

The following is analogous to Theorem 6.1 of \cite{Irie2} (and Theorem $6.1^+$ in \cite{IrieErratum}).
Recall that $\varepsilon>0$ is chosen such that $2\varepsilon$ is less than the minimal energy of non-constant $J$-holomorphic discs with boundaries on $L$.
For each $m\in \Z$, define the filtration
\begin{align*}
	F^m \CFr:= \bigoplus_{\substack{a\in H_1(L;\Z)\\ k\in\Z_{\geq 0} \\ E(a)\geq \varepsilon(m+1-k)}} C^\mathscr{L}(a,k),
\end{align*}
and similarly for $\CCFr$.  We abbreviate $\CFr$ and $\CCFr$ by $C$ and $\ol{C}$.
Also, for each of the moduli spaces $X$ in Theorem \ref{Thm: IrieKuranishi}, denote by $\bar X:=X^{\boxplus 1/2}$; Theorem \ref{Thm: ExistCF} provides $\bar X$ with an admissible CF-perturbation and an admissible strongly smooth map to $\ms{L}^{k+1} L$.

\begin{theorem}\label{Thm: ApproxSoln}
	There exists integers $I,U\geq 2$ and a sequence $(\MM_i,\NN^{\geq 0}_i,\NN^0_i, \ol{\MM_i},\ol{\NN^{\geq 0}_i},\ol{\NN^0_i})_{i\geq I}$ for every $i\geq I$, where 
	\begin{align*}
		\MM_i \in F^1C_{-1},\,\, \ol{\MM_i}\in F^1 \ol{C}_{-1},\,\, \NN^{\geq 0}_i\in F^{-U} C_2,\,\, \ol{\NN^{\geq 0}_i} \in F^{-U}\ol{C}_2,\,\, \NN^0_i \in F^{-1} C_1,\,\, \ol{\NN^0_i}\in F^{-1}\ol{C}_1,
	\end{align*}
	such that 
	\begin{align*}
		\MM_i = e_-(\ol{\MM_i}),\quad \NN^{\geq 0}_i=e_-(\ol{\NN^{\geq 0}_i}),\quad \NN^0_i =e_-(\ol{\NN^0_i}),
	\end{align*}
	\begin{align*}
		\partial \ol{\MM_i}- \frac{1}{2}[\ol{\MM_i},\ol{\MM_i}] \in F^i\ol{C}_{-2},\quad \partial \ol{\NN^{\geq 0}_i} - [\ol{\MM_i},\ol{\NN^{\geq 0}_i}] -\ol{\NN^0_i}\in F^{i-U-1}\ol{C}_1 \quad \partial \ol{\NN^0_i} - [\ol{\MM_i},\ol{\NN^0_i}] \in F^{i-2} \ol{C}_0,
	\end{align*}
	\begin{align*}
		\MM_{i+1}-e_+(\ol{\MM_i}) \in F^iC_{-1},\quad \NN^{\geq 0}_{i+1} - e_+(\ol{\NN^{\geq 0}_i}) \in F^{i-U-1}C_2,\quad \NN^0_{i+1} - e_+(\ol{\NN^0_i}) \in F^{i-2}C_1
	\end{align*}
	and such that
	\begin{enumerate}
		\item $\MM_i(a,k)\neq 0$ only if  \emph{(i)} $E(a)\geq 2\varepsilon$, or \emph{(ii)} $a = 0$, $k\geq 2$. Moreover, in case $a = 0$, $\{ \MM_i(0,k)\}_{k\in\Z_{\geq 0}}=\wt{\LL}\in\Cfr{-1}(0)$ (see the definition of $\LL$ in Definition \ref{Defn: SpecialEltLL}) 
			and, $\{\ol{\MM_i}(0,k)\}_{k\in \Z_{\geq 0}} = \ol{\LL}\in \ol{\Cfr{-1}}(0)$   (see Definition \ref{Defn: IntervalE=0});
		\item $\NN^0_i(a,k)\neq 0$ only if \emph{(i)} $E(a)\geq 2\varepsilon$, or \emph{(ii)} $a=0$. Moreover, $\NN^0_i(0)$ is a cycle which is homologous to $\LL^0$ (see Definition \ref{Defn: LL^0});
		\item Define the following subsets of $H_1(L;\Z)$:\begin{align*}
			A_\MM&:= \{a\in H_1(L;\Z)\mid \ol{\MM_i}(a,k)\neq 0 \textup{ for some }(i,k)\};\\
			A_\MM^+&:= \{a_1+\dots + a_m\in H_1(L;\Z) \mid m\geq 1 , a_1,\dots, a_m \in A_\MM\};\\
			A_{\NN}^+&:= \{ a\in H_1(L;\Z) \mid (\ol{\NN^{\geq 0}_i}(a,k),\ol{\NN^0_i}(a,k))\neq (0,0) \textup{ for some } (i,k)\};\\
			A_{\NN}^+&:=\{a_1+\dots + a_m\in H_1(L;\Z)\mid m\geq 1, a_1\in A_{\NN}, a_2,\dots, a_m\in A_\MM\}.
		\end{align*}
		Then for any $\lambda>0$, \[A_\MM^+(\lambda):= \{a\in A_\MM^+\mid E(a)<\lambda\}\textup{ and }A^+_{\NN}(E):=\{ a\in A^+_{\NN}\mid E(a)<\lambda\}\] are finite sets.
	\end{enumerate}
\end{theorem}
\begin{remark}
	One difference between our version of Theorem \ref{Thm: ApproxSoln} with Theorem 6.1$^+$ of \cite{IrieErratum} is that we require the zero-energy elements $\MM_i(a,k)$ and $\ol{\MM_i}(a,k)$ to match up with $\wt{\LL}$ and $\ol{\LL}$ on the chain-level exactly. The proof will be exactly the same as that for Theorem 6.1 in \cite{Irie2} and Theorem $6.1^+$ of \cite{IrieErratum} other than this point, so in the proof sketch below we will only be focused on verifying this.
\end{remark}

\begin{proof}[Proof sketch]
	Using Theorem \ref{Thm: CFgivesDR}, we get
	\begin{align*}
		\MM_m &:= \sum_{E(\beta)< \varepsilon(m+1-k)} (-1)^{n+1}\ev_* (\bar{\ms{M}}_{k+1}(\beta,\{m\})),\\
		\ol{\MM_m}&:= \sum_{E(\beta)<\varepsilon(m+1-k)} (-1)^{k+1} \ev_*(\bar{\ms{M}}_{k+1}(\beta,[m,m+1])),\\
		\NN^{\geq 0}_m&:= \sum_{E(\beta) < \varepsilon(m-U-k)} (-1)^{n+k+1} \ev_*(\bar{\ms{N}}_{k+1}^{\geq 0}(\beta, \{m\})),\\
		\ol{\NN^{\geq 0}_m}&:= \sum_{E(\beta) < \varepsilon(m-U-k)} \ev_*(\bar{\ms{N}}_{k+1}^{\geq 0}(\beta, [m,m+1])), \\
		\NN^{0}_m &:= \sum_{E(\beta) < \varepsilon(m-1-k)} (-1)^{n+k+1} \ev_* (\bar {\ms{N}}_{k+1}^0 (\beta,\{m\})), \\
		\ol{\NN^{0}_m} &:= \sum_{E(\beta) < \varepsilon(m-1-k)} -\ev_*(\bar{\ms{N}}_{k+1}^0 (\beta, [m,m+1])). 
	\end{align*}
	See Remarks 7.35 and 7.36 of \cite{Irie2} for notations (also see 5(v) in \cite{IrieErratum} for signs). In particular, the differential forms taken to be $1$.
	The properties in the theorem (other than the energy-zero chains) are proved using the same argument as in section 7.6 in \cite{Irie2}. In case $\beta=0$,  by Theorem \ref{Thm: ExistCF} we we have $\bar{\ms{M}}_{k+1}(0:P)\cong P\times L\times (D^{k-2})^{\boxplus \tau'}$ with trivial obstruction bundles and trivial CF-perturbations.
	By the explicit description of the de Rham chain $\ev_*(\bar{\ms{M}}_{k+1}(0:P))$ discussed after Theorem \ref{Thm: CFgivesDR}, we have the following cases:

	\begin{itemize}
		\item In case $k<2$, $\bar{\ms{M}}_{k+1}(0:P)=\emptyset$, so $\ev_*(\bar{\ms{M}}_{k+1}(0:P))=0$;
		\item In case $k=2$, $\bar{\ms{M}}_{3}(0:P)\cong P\times L$. In the case $P=\{m\}$, the evaluation map $\bar{\ms{M}}_3(0:\{m\})\xrightarrow{\ev^\ms{M}} \ms{L}^3(0)$ is given by \begin{align*}	
				\bar{\ms{M}}_3(0:\{m\})\cong L&\xrightarrow{\iota_3} \ms{L}^3(0); \\ y &\mapsto ([\underline{y}],[\underline{y}],[\underline{y}])
			\end{align*}
			 by \eqref{Eq: EvEnergyZero}. 
		By the formula \eqref{Eq: CFpushforward} and Definition \ref{Defn: SpecialEltLL}, $\ev_*(\bar{\ms{M}}_{3}(0,\{m\}))= \LL$;
		\item In case $k>2$, $\bar{\ms{M}}_{k+1}(0:P)\cong P\times L \times (D^{k-2})^{\boxplus\tau'}$ and in the case $P=\{m\}$, the evaluation map $\bar{\ms{M}}_{k+1}(0:\{m\})\xrightarrow{\ev^\ms{M}} \ms{L}^{k+1}(0)$  is given by \begin{align*}
			 \bar{\ms{M}}_{k+1}(0:\{m\}) \cong L\times (D^{k-2})^{\boxplus \tau'} \xrightarrow{\pr_L} L &\xrightarrow{\iota_{k+1}} \ms{L}^{k+1}(0);\\
			 y &\mapsto ([\underline{y}],[\underline{y}],\dots,[\underline{y}])  		
			 \end{align*} 
			 by \eqref{Eq: EvEnergyZero}. The factorization of $\ev^\ms{M}=\iota_{k+1}\circ \pr_L$ shows that $\ev_*(\bar{\ms{M}}_{k+1}(0,\{m\}))$ is a degenerate de Rham chain: more precisely,  by \eqref{Eq: dRChainRelation} in Definition \ref{Defn: dRchainCplx},
			 \begin{align*}
			 	\ev^\ms{M}_*(\bar{\ms{M}}_{k+1}(0:\{m\})) = [(L\xrightarrow{\iota_{k+1}} \ms{L}^{k+1}(0); (\pr_L)_! 1=0)]=0.
			 \end{align*}
	\end{itemize}
	The discussion with the case $P=[m,m+1]$ is similar.
\end{proof}

\subsection{Taking limits of approximate solutions}
\label{Subsc: VFC}
We now take limits of the low-energy approximations in Theorem \ref{Thm: ApproxSoln} to prove Theorem \ref{Thm: VFC}, which is analogous to Theorem 5.1 in \cite{Irie2}.

\begin{theorem}\label{Thm: VFC}
Under the setup in Theorem \ref{Thm: VFCmain}, 
there exists the following data:
\begin{itemize}
	\item  A constant $\hbar>0$,
	\item A monoid of curve classes $\mathfrak{G}\subset H_1(L;\Z)$ (Definition \ref{Defn: monoid})  and a module $\mathfrak{N}\subset H_1(L;\Z)$ of $H$-perturbed curve classes over $\mathfrak{G}$ (Definition \ref{Defn: module});
	\item For each $\beta\in\mathfrak{G}$ a chain $\MM(\beta)\in \Cfr{-1}(\beta)$, and for each $\eta\in \mathfrak{N}$ a chain $\NN^{\geq 0}(\eta)\in \Cfr{2}(\eta)$ and a chain $\NN^0(\eta)\in \Cfr{1}(\eta)$;
\end{itemize}
such that 
\begin{enumerate}
	\item The element \[\MM := \sum_{\beta\in\mathfrak{G}} \MM(\beta) \in \wh{\Cfr{-1}}_\mathfrak{G}, \]where $\MM(\beta)\in \Cfr{-1}(\beta)$,
	satisfies the Maurer-Cartan equation  \[ \partial^\dR \MM + \frac{1}{2} \big[\MM,\MM\big]  = 0. \]
	\item The elements \[\NN^{\geq 0} := \sum_{\eta\in\mathfrak{N}} \NN^{\geq 0}(\eta) \in \wh{\Cfr{2}}_\mathfrak{N},\quad 
	\NN^{ 0} := \sum_{\eta\in\mathfrak{N}} \NN^{ 0}(\eta)\in \wh{\Cfr{1}}_\mathfrak{N}, \]where $\NN^{\geq 0}(\eta)\in \Cfr{2}(\eta)$, $\NN^0(\eta) \in \Cfr{1}(\eta)$,
	satisfy \[ \partial^\dR  \NN^{\geq 0}- \big[\MM, \NN^{\geq 0}\big] = \NN^{0}.   \]
	\item Let $\MM(\beta,k)$ be the component of $\MM(\beta)\in \Cfr{-1}(\beta)$ in $C^\mathscr{L}(\beta, k)_{-1}$. Then $\MM(\beta, k)\neq 0$ only if \emph{(i)} $E(\beta)\geq \hbar$, or \emph{(ii)} $\beta = 0$, $k\geq 2$. Moreover, in case $\beta = 0$, $\MM(0)=\LL$ (see the definition of $\LL$ in Definition \ref{Defn: SpecialEltLL}).
	\item Let $\NN^0(\eta,k)$ be the component of $\NN^0(\eta)\in \Cfr{1}(\eta)$ in $C^\mathscr{L}(\eta, k)_1$. Then $\NN^0(\eta, k)\neq 0$ only if \emph{(i)} $E(\eta)\geq \hbar$, or \emph{(ii)} $\eta = 0$. Moreover, in case $\eta = 0$, $\NN^0(0)$ is a cycle homologous to $\LL^0$ (see the definition of $\LL^0$ in Definition \ref{Defn: LL^0}).
\end{enumerate}
 \end{theorem}

\begin{proof}
	The proof is the same as section 6 of \cite{Irie2} (also see \cite{IrieErratum}), but we need to keep track of the energy-zero parts. For simplicity we just focus on $\MM(0,k)$, by following the construction in \cite{Irie2}. 
	The construction in \cite{Irie2} Lemma 6.4 has the following steps (with changes in notation):
	\begin{enumerate}
		\item Define $\MM_{i,0}=\MM_i$ and $\ol{\MM_{i,0}}=\ol{\MM_i}$, so that $\{\MM_{i,0}(0,k)\}_{k\in\Z_{\geq 0}}= \wt{\LL}$, $\{\ol{\MM_{i,0}}(0,k)\}_{k\in\Z_{\geq 0}}= \ol{\LL}$;
		\item Inductively define $\Delta^i_\MM:= \MM_{i+1,j}-e_+(\ol{\MM_{i,j}})$;  by \eqref{Eq: [0,1]xL},  $e_+\ol{\LL} = \wt{\LL}$, so that $\Delta^i_\MM(0,k)=0$ for all $k\in \Z_{\geq 0}$;
		\item Pick $\ol{\Delta}_\MM^i\in F^{i+j} \ol{C}_{-1}$ such that \[e_+(\ol{\Delta}_\MM^i)-\Delta^i_\MM\in F^{i+j+1}C_{-1} \,\, \textup{ and } \,\, \partial \ol{\Delta}_\MM^i + \left(\partial \ol{\MM_{i,j}} - \frac{1}{2} \left[\ol{\MM_{i,j}},\ol{\MM_{i,j}}\right]\right) \in F^{i+j+1}\ol{C}_{-2}.\] Since $\Delta^i_\MM(0,k)=0$ for all $k\in \Z_{\geq 0}$, we may take $\ol{\Delta}_\MM^i(0,k)=0$ for all $k\in\Z_{\geq 0}$;
		\item Set \[ \ol{\MM_{i,j+1}}:= \ol{\MM_{i,j}}+\ol{\Delta}_\MM^i,\quad \MM_{i,j+1}:= e_-(\ol{\MM_{i,j+1}}). \]Thus $\{\ol{\MM_{i,j+1}}(0,k)\}_{k\in\Z_{\geq 0}} = \ol{\LL}$ and $\{\MM_{i,j+1}(0,k)\}_{k\in\Z_{\geq 0}} = \wt{\LL}$;
		\item Finally, define $\MM:= \lim_{j\to\infty} \MM_{i,j}$  for a fixed $i\geq I$. It then follows that $\{\MM(0,k)\}_{k\in \Z_{\geq 0}} = \wt{\LL}.$ 
	\end{enumerate}
\end{proof}

\subsection{Proof of Theorem \ref{Thm: VFCmain}}

\begin{lemma}
	The element \[\MM^0:=\sum_{k\geq 2} \MM(0,k) \in \wh{\Cfr{-1}}_\mathfrak{G}\] satisfies the Maurer-Cartan equation \[ \partial^\dR \MM^0 - \frac{1}{2}[\MM^0,\MM^0] = 0, \] and \[ -[\MM^0, x]=\partial^1 x\quad \textup{ for any }x\in \CFr\] 
	where $\partial^1$ is defined in \eqref{Eq: differentialExplicit}.
\end{lemma}
\begin{proof}
	By Theorem \ref{Thm: VFC} (3), $\MM^0 = \wt{\LL}$. 
	That $\MM^0$ satisfies Maurer-Cartan equation then follows from Lemma \ref{Lm: E=0 MC elt} (this also follows from conditions (1) and (3) of Theorem \ref{Thm: VFC} and energy considerations, analogous to the paragraph after Theorem 5.1 of \cite{Irie2}). 
	That $-[\MM^0, -]= \partial^1$ follows from	\eqref{Eq: [L,-]=delta}.
\end{proof}

Recall our notation $\mathfrak{G}^+:= \mathfrak{G}\setminus\{0\}$.
\begin{corollary}
	The element \[ \MM^+:= \sum_{\beta\in\mathfrak{G}^+} \MM(\beta) \in \wh{\Cfr{-1}}_{\mathfrak{G} ^+}\]satisfies the Maurer-Cartan equation \[ \partial \MM^+ + \frac{1}{2}[\MM^+,\MM^+] = 0, \]where $\partial = \partial^0+\partial^1$ is the differential in \eqref{Eq: differential}.
	Also \[ \partial \NN^{\geq 0} - [ \MM^+,  \NN^{\geq 0}] = \NN^0. \]
\end{corollary}
\begin{proof}
	This is the same as \cite{Irie2}, Lemma 5.3.
\end{proof}

Theorem \ref{Thm: VFCmain} then follows by replacing $\MM$ with $\MM^+$.

\bigskip


\AtNextBibliography{\Small}
\printbibliography

\end{document}